\documentclass[a4paper,11pt]{article}
\usepackage[T1]{fontenc}
\usepackage[utf8]{inputenc}
\usepackage[english]{babel}
\usepackage[margin=0.8in]{geometry}
\usepackage{microtype}
\usepackage{braket}
\usepackage{amsmath}
\usepackage{amssymb}
\usepackage{amsthm, aliascnt}
\usepackage{mathtools}
\usepackage{mathrsfs}
\usepackage{xcolor}
\usepackage[hyphens]{url}
\usepackage{hyperref}
\usepackage{enumerate}
\usepackage{tikz}
\usepackage[maxbibnames=99]{biblatex}
\usepackage{csquotes}
\usetikzlibrary{patterns}
\usepackage{tikz}
\usepackage{esint}
\usepackage{mleftright}
\usepackage{yhmath}

\hypersetup{
                colorlinks=true,
                linkcolor={black},
                citecolor={cyan},
                breaklinks=true}

\theoremstyle{plain}
\newtheorem{teor}{Theorem}
\numberwithin{teor}{section}
\numberwithin{equation}{section}

\theoremstyle{definition}
\newaliascnt{defi}{teor}
\newtheorem{defi}[defi]{Definition}
\aliascntresetthe{defi}

\theoremstyle{plain}
\newaliascnt{lemma}{teor}
\newtheorem{lemma}[lemma]{Lemma}
\aliascntresetthe{lemma}
\theoremstyle{plain}
\newaliascnt{prop}{teor}
\newtheorem{prop}[prop]{Proposition}
\aliascntresetthe{prop}
\theoremstyle{plain}
\newaliascnt{conjecture}{teor}

\aliascntresetthe{conjecture}
\theoremstyle{plain}
\newaliascnt{cor}{teor}

\aliascntresetthe{cor}
\theoremstyle{definition}
\newaliascnt{ex}{teor}

\aliascntresetthe{ex}
\theoremstyle{definition}
\newaliascnt{oss}{teor}
\newtheorem{oss}[oss]{Remark}
\aliascntresetthe{oss}
\theoremstyle{plain}

\DeclarePairedDelimiter{\abs}{\lvert}{\rvert}
\DeclarePairedDelimiter{\norma}{\lVert}{\rVert}

\newcommand{\R}{\mathbb{R}}

\newcommand{\N}{\mathbb{N}}

\newcommand{\Hn}{\mathcal{H}^{n-1}}

\newcommand{\hdu}{\widehat{\dot{u_t}}}
\newcommand{\B}{\mathbf{B}}
\newcommand{\dint}{\displaystyle\int}
\newcommand{\Om}{\Omega}
\newcommand{\la}{\lambda}

\newcommand{\ssubset}{\subset\joinrel\subset}
\DeclareMathOperator{\divv}{div}

\newcommand{\IClabel}{%
  \label{IC}%
  \textnormal{\textbf{(IC$_{H^s,X}$)}}%
}

\newcommand{\ICref}{%
  \textnormal{\textbf{(\hyperref[IC]{\textcolor{magenta}{IC$_{H^s,X}$}})}}%
}

\newcommand{\Cslabel}{%
  \label{Cs}%
  \textnormal{\textbf{(C$_{H^s}$)}}%
}

\newcommand{\Csref}{%
  \textnormal{\textbf{(\hyperref[Cs]{\textcolor{magenta}{C$_{H^s}$}})}}%
}

\newcommand{\Addresses}{{ 
\bigskip
\footnotesize 
\noindent\textit{E-mail address}, E.~Cristoforoni: \texttt{emanuele.cristoforoni@unimi.it} \nopagebreak
 
\noindent\textsc{Dipartimento di Matematica ``Federigo Enriques'', Universit\'a degli Studi di Milano La Statale, Via Saldini 50 20123, Milano, Italy.}
  \medskip

\noindent\textit{E-mail address}, 
  F.~Villone: \texttt{f.villone@ssmeridionale.it} 
}\nopagebreak

\footnotesize
\noindent\textsc{Mathematical and Physical Sciences for Advanced Materials and Technologies, Scuola Superiore Meridionale, Largo San Marcellino 10, 80138, Napoli, Italy.}
}

\title{Symmetry-breaking and local stability of a two-phase eigenvalue problem in optimal insulation}
\author{Emanuele Cristoforoni, Federico Villone}
\date{}

\begin{document}
\maketitle
\begin{abstract}
We consider the first eigenvalue, $\lambda_\beta(\Omega,A)$, of a two-phase eigenvalue problem for the Laplacian with Robin boundary conditions, where the two phases are characterised by different ellipticity constants. We characterise the conditions under which a ball $B_R$ is a local minimum under a volume constraint for the minimisation problem $A\mapsto\lambda_\beta(B_r,A)$, in terms of the principal Neumann eigenvalue of the fixed inner ball $B_r$. \smallskip

    \textsc{Keywords: symmetry-breaking, Faber-Krahn, shape optimisation, Robin Boundary conditions, thermal insulation}
    
    \textsc{MSC 2020: 49Q10; 49K20; 35P15; }
\end{abstract}

\begin{center}
\begin{minipage}{10cm}
{\small
\tableofcontents}
\end{minipage}
\end{center}

\hypersetup{
                linkcolor={magenta},
}

\section{Introduction}\label{intro}
Let $n\ge2$ and let $\Om, A$ be bounded open subsets of  $\R^n$ with smooth boundary such that $A$ is connected and $\Om\ssubset A$. Let $\nu_\Om$ and $\nu_A$ be the unit outer normal to $\Om$ and $A$ respectively. Fixed $\beta, a>0$, we consider the following two-phase eigenvalue problem with Robin boundary conditions

 \begin{equation}\label{eq forte autofunz}
\begin{cases}
-\Delta v = {\lambda_\beta} v & \text{in }\Om, \\[5 pt]
-a\Delta v = \lambda_\beta v & \text{in }A\setminus\bar\Om\\[5pt]
v^-=v^+ & \text{on } \partial \Om, \\[5 pt]
\partial_{\nu_\Om} {{v}}^-= a \partial_{\nu_\Om} {v}^+ & \text{on } \partial\Omega,\\[5 pt]
a\partial_{\nu_A} v + \beta v = 0 & \text{on } \partial A, 
\end{cases}
\end{equation}
where $u^-$ and $u^+$ denote the traces of $u$ on $\partial\Omega$ from $\Omega$ and from $A\setminus\bar{\Om}$ respectively. By the classical theory of elliptic operators in Hilbert spaces, the eigenvalue problem \eqref{eq forte autofunz} admits a discrete spectrum, and the first eigenvalue $\lambda_\beta(\Om,A)$ can be variationally characterised as 

   \[\displaystyle{\lambda_\beta(\Om,A)=\min_{v\in H^1(A)\setminus\set{0}}\dfrac{\displaystyle\int_\Om \abs{\nabla v}^2\,dx+a\int_{A\setminus \bar\Om} \abs{\nabla v}^2\,dx+\beta \int_{\partial A}v^2\,d\Hn}{\displaystyle\int_A v^2\,dx}}.
\]
Moreover, the associated eigenfunction $u$ is a solution to the equation
\[\int_\Om{\nabla u}\nabla \psi\,dx+a\int_{A\setminus \bar\Om} {\nabla u}\nabla \psi\,dx+\beta \int_{\partial A}u\psi \,d\Hn=\lambda_\beta(\Om,A)\int_A u\psi\,dx\]
for all $\psi\in H^1(A)$. We remark that, for a fixed $\Om$, in contrast with the one-phase case (i.e., $a=1$), this eigenvalue is not invariant under translations of the set $A$.\medskip

The interest in the eigenvalue problem \eqref{eq forte autofunz} is motivated by the study of thermal insulation. Indeed, if $\Omega$ represents a thermal conductor with thermal diffusivity $\kappa_\Omega$ surrounded by a layer of highly insulated material $\Sigma=A\setminus\bar\Om$ with thermal diffusivity $\kappa_\Sigma$, setting $a=\kappa_\Sigma/\kappa_\Omega$, we have that, assuming the outside temperature is equal to zero, the temperature $T$ in  $A$ is a solution to the heat equation
\[\begin{cases}
  \partial_t T -\Delta T = f &\text{ in } \Omega\times(0,+\infty),\\[5 pt]
     \partial_t T -a\Delta T = f &\text{ in } A\setminus\bar\Om\times(0,+\infty),\\[5 pt]
       T^- = T^+ &\text{ on }\partial\Omega\times(0,+\infty),\\[5 pt]
     \partial_{\nu_\Om} T^- = a \partial_{\nu_\Om}  T^+ &\text{ on }\partial\Omega\times(0,+\infty),\\[5 pt]
     a\partial_{\nu_A} T+ \beta T=0 &\text{ on }\partial A\times(0,+\infty),\\[5pt]
     T(\cdot,0)=T_{0}&\text{ in }A,
\end{cases}\]
where $f$ represents the heat source, $T_0$ is the initial temperature, and the Robin boundary conditions, according to Newton's law of cooling, model the case in which convection is the main mode of heat transfer with the outside environment. Hence, denoting by $(\lambda_{\beta,i},u^i)$ the eigendata of the eigenvalue problem \eqref{eq forte autofunz}, and denoting by $T_\infty$ the stationary solution to the heat equation, i.e., the solution to the Poisson problem
\[\begin{cases}
     -\Delta T_\infty = f &\text{ in } \Omega,\\[5 pt]
     -a\Delta T_\infty = f &\text{ in } A\setminus\bar\Om,\\[5 pt]
       T^-_\infty = T^+_\infty &\text{ on }\partial\Omega,\\[5 pt]
     \partial_{\nu_\Om} T^-_\infty = a  \partial_{\nu_\Om} T^+_\infty &\text{ on }\partial\Omega\,\\[5 pt]
     a \partial_{\nu_A} T_\infty+ \beta T_\infty=0 &\text{ on }\partial A;
\end{cases}\]
we have that the temperature $T$ can be expressed  as 
\[T(x,t)= T_\infty(x)+\displaystyle\sum_{i=1}^\infty \alpha_i e^{-\lambda_{\beta,i} t}u^i(x),\]
for appropriate constants $\alpha_i\in\R$ depending on the initial temperature $T_0$. 
Thus, the first eigenvalue $\lambda_\beta(\Om,A)$ governs the decay of the temperature inside the body $\Om$. We refer to \cite{C59} for the physical interpretation of the problem.\medskip

Hence, it is physically relevant, for a given set $\Om$, to consider the problem of minimising $\lambda_\beta(\Om,A)$ under a volume constraint for the insulating layer $A\setminus \bar\Om$, or, equivalently, for the set $A$. More precisely, it is interesting to study 
\begin{equation}\label{min problem A}
    \inf \set{\lambda_\beta(\Om,A): A\text{ is Lipschitz, }{\Om}\ssubset A \text{ and } \abs{A}=m},
\end{equation}
where $\Om$ and $m>0$ are fixed.\medskip 

The optimisation of two-phase problems, and in particular two-phase eigenvalue problems, has been extensively studied in the context of optimal design. Namely, fixing the outer set $A$ and allowing the inner set $\Omega$, or rather the ellipticity constant, to vary. We refer, for instance, to \cite{CL96, CMS09, DK10, MNP22}. It is well known that this type of control problem usually does not admit a solution (see, for instance, \cite{CD15, CD17} and the references within). However, in the particular case in which $A$ is a ball, the problem admits a radially symmetric solution \cite[Theorem 2.1]{AT83}, but not necessarily given by a ball \cite{CLM12, L14}.\medskip

The case of \eqref{min problem A} where the inner set $\Om$ is fixed and $A$ is allowed to vary, has been partially addressed in the case of Dirichlet boundary conditions in \cite[Section 3]{DK10} from the point of view of shape variation. On the other hand,
optimisation problems of this kind have been extensively studied in the thin-layer regime, that is, in the case where the insulating layer $\Sigma_{ah}=A\setminus\bar{\Om}$ can be described as a normal sub-graph on the boundary of $\Om$ of a smooth, positive function $h$ times the parameter $a$,
\[\Sigma_{ah}=\{ x \in \mathbb{R}^n : x = y + t \nu_\Om(y),\ y \in \partial \Omega,\ 0 < t < a h(y) \},\]
and the problem is studied in the limit as $a \to 0^+$. We refer to \cite{F80, BCF80, AB86} for the characterisation of the limit equations  (see also \cite{ZRWZ09, LWZZ12, Y18, AC25, CV26}).\medskip

In this asymptotic regime, optimisation problems of the form \eqref{min problem A} reduce to an optimisation problem in the control variable $h$, and the measure constraint on $A $ is approximated as a constraint on the integral of $h$. (see for instance \cite{B88, DPNST21, ACNT24, CNT26, DPO26}). In particular, letting
\[\lambda_\beta(\Om,h)=\lim_{a\to 0^+}\lambda_\beta(\Om, \Sigma_{ah}\cup\bar\Om),\]
the eigenvalue problem \eqref{min problem A}, in the limit regime, reads as 
\begin{equation}\label{minimisation in h}
    \inf\Set{\lambda_\beta(\Om,h): h: \partial \Om\to (0,+\infty),\, h \text{ is regular and } \int_{\partial \Om} h\,d\Hn=\tilde m},
\end{equation}
where $\tilde m$ is fixed. problem \eqref{minimisation in h} was studied in \cite{BBN17} for the Dirichlet case (see also \cite{BB19, HLL22, HLL24}), and in \cite{DPO25} for the Robin one. The authors show that, in the case where $\Omega$ is a ball, an interesting symmetry-breaking phenomenon occurs. They prove that the constant function is a minimiser of problem \eqref{minimisation in h} only when the total amount of insulating material $\tilde m$ is sufficiently large, whereas, below a certain threshold value, symmetry-breaking occurs. Namely, in \cite{DPO25}, they prove the following theorem.
\begin{teor}[\cite{DPO25} Theorem 4.1]\label{DPOthm}
    Let $\Om=B_r$ be a ball and let $h^*$ be a minimiser for problem \eqref{minimisation in h}. There exists $\beta^*>0$ such that
    \begin{itemize}
    \item[]if $\beta<\beta^*$, then for every $\tilde{m}>0$, $h^*$ is constant, and symmetry-breaking does not occur;\medskip
 
\item[]if $\beta>\beta^*$, there exists $\tilde{m}(\beta)>0$ such that: if $\tilde{m}>m(\beta)$, $h^*$ is constant, while,
if $\tilde{m}<\tilde{m}(\beta)$, $h^*$ is not constant, hence symmetry-breaking occurs. 
\end{itemize}
\end{teor}
Let $h_0$ denote the constant function satisfying the integral constraint in \eqref{minimisation in h}, and let $\lambda^N(B_r)$ be the principal Neumann eigenvalue (i.e. the first positive eigenvalue of the Neumann Laplacian on  $B_r$). From the proof of the result, we have that the symmetry-breaking condition can be restated as: 
\begin{itemize}

\item[]if
\[\lambda_\beta(B_r,h_0)<\lambda^N(B_r)\]
then $h^*=h_0$, and symmetry-breaking does not occur
\item[]if 
\[\lambda_\beta(B_r,h_0)>\lambda^N(B_r)\]
then $h^*\ne h_0$, and symmetry-breaking occurs.
\end{itemize}\medskip

This behaviour is in stark contrast with the classical Faber--Krahn inequality \cite{F1923, K1925}, and its Robin counterpart due to Bossel and Daners \cite{B86, D06}, which assert that, in the one-phase case (i.e. $a=1$), the ball uniquely minimises \eqref{min problem A} for every $\beta > 0$ and every $m>0$. Thus, when $a$  and the thickness of the layer $A \setminus \Omega$ approach $0$, symmetry-breaking phenomena may occur, while in the case $a=1$ this is not the case.
This raises a natural question:\medskip

"In which regimes can symmetry-breaking be detected and characterised directly for problem \eqref{min problem A}, without passing to a limiting problem?"\medskip

In the present paper we answer this question for smooth, local perturbations of the outer ball. Our main result shows that, locally, symmetry-breaking can be established directly at the level of the original functional, under suitable conditions on the parameters of the problem. 
\medskip

In this spirit, we fix $\Om=B_r$  the ball centred at the origin of radius $r$. We are interested in understanding whether balls centred at the origin are local minima of \eqref{min problem A}. We consider $R>0$ such that $|B_R|=m$ and we optimise ${\lambda_\beta}(B_r,\cdot)$ in the class of \emph{nearly spherical domains with fixed measure}, defined in \autoref{def nearly spherical}. We divide the cases $a<1$ and $a>1$, and we obtain the following results. 

\begin{teor}\label{teor: main a<1}    
Let $\beta, r>0$, $R>r$ and $a<1$. Let $\lambda^N(B_r)$ be the principal Neumann eigenvalue on $B_r$ i.e. the first  positive real number such that the problem 
    \begin{equation}
        \begin{cases}
            -\Delta v=\lambda^N v& \text{in } B_r,\\[5pt]
            \partial_{\nu_{B_r}} v=0& \text{on }\partial B_r,
        \end{cases}
    \end{equation}
    admits a non-trivial solution.
    Then the following properties hold:
    \begin{itemize}
    \item If 
        \[\lambda_\beta(B_r,B_R)>\lambda^N(B_r),\]
        then $B_R$ is not a minimum for the problem \eqref{min problem A}. In particular, there exists $\eta>0$ such that for $t\in(0,\eta) $ and for any unit vector $\mathbf{v}$, $B_r\ssubset B_R+t\mathbf{v}$ and
        \[\lambda_\beta (B_r,B_R)>\lambda_\beta (B_r,B_R+t\mathbf{v}),\]
        where $B_R+t\mathbf{v}$ denotes a translation of the outer ball $B_R$. 
        \item If 
        \[\lambda_\beta(B_r,B_R)<\lambda^N(B_r),\]
        then $B_R$ is a $C^{2,\gamma}-$strictly stable local minimum under volume constraint for $\lambda_\beta(B_r,\,\cdot\,)$ in the sense of \autoref{def ottimalità locale}.
    \end{itemize}
\end{teor}

\begin{oss}
Thanks to the assumption $a<1$, in \autoref{oss k0}, we prove that the function
\[
        R\mapsto \lambda_\beta(B_r,B_R)
\]
is monotone-decreasing. Moreover
\[
        \lim_{R\to r^+}\lambda_\beta(B_r,B_R)
        =
        \lambda_\beta(B_r),
\]
where $\lambda_\beta(B_r)$ denotes the one-phase Robin eigenvalue of
$B_r$. Therefore, the threshold in \autoref{teor: main a<1} can be described as follows. Since
$\lambda_\beta(B_r)$ is increasing with respect to $\beta$, one can show that there exists
a critical value $\beta^*>0$ defined by
\[
        \lambda_{\beta^*}(B_r)=\lambda^N(B_r),
\]
such that:\medskip

\noindent
If $\beta\leq \beta^*$, then
\[
        \lambda_\beta(B_r,B_R)<\lambda^N(B_r)
        \qquad\text{for every } R>r,
\]
and the ball $B_R$ is a strictly stable local minimum under volume constraint. Hence, at least locally, no symmetry-breaking occurs.\medskip

\noindent
If $\beta>\beta^*$, then
\[
        \lambda_\beta(B_r)>\lambda^N(B_r),
\]
and, by monotonicity, there exists a critical volume $m^*=m^*(r,\beta)>|B_r|$ and a radius $R^*>r$,
such that $\abs{B_{R^*}}=m^*$ and
\[
        \lambda_\beta(B_r,B_{R^*})
        =
        \lambda^N(B_r).
\]
Then, for
\[
        |B_r|<m<m^*
\]
symmetry-breaking occurs, whereas for
\[
        m>m^*
\]
the concentric ball $B_R$ of measure $\abs{B_R}=m$, is a strictly stable local minimum under volume constraint.
Thus, (locally) symmetry-breaking occurs only if $\beta$ is sufficiently large and $m$ sufficiently small, and our result is in complete analogy to the limiting case of \cite{DPO25}.
\end{oss}

In the case $a>1$, the symmetry-breaking condition is almost the opposite one. Namely, we prove the following theorem.

\begin{teor}\label{teor: main a>1}
    Let $\beta, r>0$, $R>r$ and $a>1$. Let $\lambda^N(B_r)$ be the principal Neumann eigenvalue on $B_r$. 
    Then the following properties hold:
    \begin{itemize}
    \item If either
        \[\lambda_\beta(B_r,B_R)<\lambda^N(B_r),\quad \text{or}\quad \lambda_\beta(B_r,B_R)<\beta \dfrac{n-1}{R}-\dfrac{\beta^2}{a}
        \]
        then $B_R$ is not a minimum for the problem \eqref{min problem A}. In particular, there exists $\eta>0$ such that for $t\in(0,\eta) $ and for any unit vector $\mathbf{v}$, $B_r\ssubset B_R+t\mathbf{v}$ and
        \[\lambda_\beta(B_r,B_R)>\lambda_\beta (B_r,B_R+t\mathbf{v}).\]
        
        \item If
        \[\lambda_\beta(B_r,B_R)>\lambda^N(B_r),\quad \text{and}\quad \lambda_\beta(B_r,B_R)>\beta \dfrac{n-1}{R}-\dfrac{\beta^2}{a}\]
        then $B_R$ is a $C^{2,\gamma}-$strictly stable local minimum under volume constraint for $\lambda_\beta(B_r,\,\cdot\,)$ in the sense of \autoref{def ottimalità locale}.
    \end{itemize}
\end{teor}

\definecolor{outercircle}{RGB}{95,95,105}
\definecolor{innercircle}{RGB}{150,150,160}

\begin{figure}
    \centering
\begin{tikzpicture}

\fill[outercircle] (-2,0) circle (1.9);
\fill[innercircle] (-2,0) circle (1.4);


\fill[outercircle] (6.4,0.3) circle (1.9);
\fill[innercircle] (6.5,0) circle (1.4);


\end{tikzpicture}
\caption{Depiction of the symmetric configuration (left) and of the translated one (right)}

\end{figure}
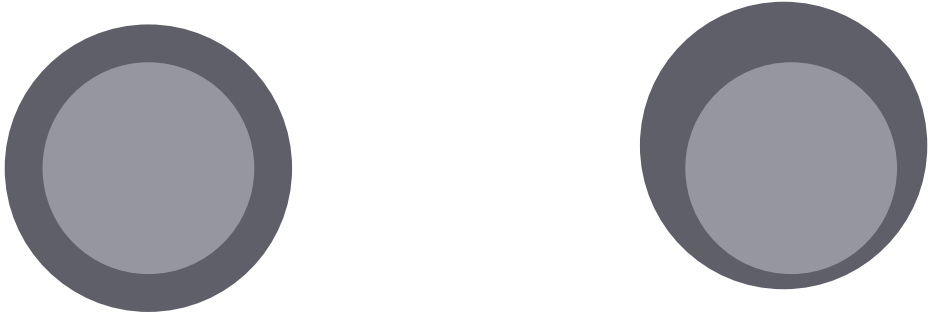

The paper is structured as follows.\medskip 

\noindent In \autoref{N&T} we fix the notations that will be used throughout the paper and recall some preliminary results. Namely, in \autoref{SDLO} we recall general notions and results on shape derivatives, while in \autoref{RBC} we discuss the main features of the two-phase Robin boundary conditions.\medskip 

\noindent Next, in \autoref{SDcomputed}, we compute the first and the second shape derivatives of the shape functional
\[A\mapsto\lambda_\beta(\Om,A)\]
and observe that, when $\Om$ is a ball $B_r$, then concentric balls of larger radius, $B_R$, are \emph{critical shapes} under volume constraint (see \autoref{def ottimalità locale}).\medskip

\noindent In \autoref{PMT} we use the general theory formalised by \cite{DL19} (see \autoref{DL}) to prove  \autoref{teor: main a<1} and \autoref{teor: main a>1}. In particular, in \autoref{coercivity} we characterise the condition under which the second derivative of $\lambda_\beta(B_r,\cdot)$, at the ball $B_R$, satisfies the necessary second-order condition for minimality (i.e. the condition under which $B_R$ is a \emph{strictly stable shape} under volume constraint in the sense of \autoref{def ottimalità locale}).
In \autoref{improved} we prove the so called \emph{improved continuity of the second derivative} (see condition \ICref\,).\medskip

\noindent Finally, in \autoref{secFR}, we conclude the paper with a brief discussion of the case where the Robin boundary conditions are replaced by Dirichlet ones in \autoref{BB}, and of the limit problem \eqref{minimisation in h} in \autoref{limsect} 

\section{Notation and tools}\label{N&T}
We introduce some notation and results that will be used throughout the paper. 

\subsection{Shape derivatives and local optimality}\label{SDLO}

Let $\mathcal{O}$ denote a family of subsets of $\R^n$ and let $J$ be a shape functional, that is a function \[J\colon\mathcal{O}\to\R.\]
Let $A\in\mathcal{O}$ and assume that $(I+\theta)A\in\mathcal{O}$ for $\theta$ in a $W^{1,\infty}(\R^n,\R^n)$-neighbourhood of $0$. If the function
\[J_A\colon \theta\in W^{1,\infty}(\R^n,\R^n)\mapsto J((I+\theta)A)\in\R \]
is twice Fr\'echet-differentiable at $\theta=0$ we call the derivative
\[J'(A):=J_A'(0),\quad\text{and}\quad J''(A):=J_A''(0)\]
the first and second shape derivatives of $J$ at $A$.
 
\begin{teor}(\textit{Structure Theorem of first and second shape derivatives}, \cite[Theorem 5.9.2]{HP18shape}). \label{Teorema struttura}
Let $A\subset \R^n$ be a bounded, open set with $C^2$ boundary. Let 
\[
\mathcal{F}(A) =\set{(I+\theta)A: \theta\in W^{1,\infty}(\R^n,\R^n): \norma{\theta}_{W^{1,\infty}}<1},
\]
and let $J$ be a shape functional defined on $\mathcal{F}(A)$.
Consider the function 
\[J_A:\theta\in\{\xi\in W^{1,\infty}(\R^n,\R^n):\norma{\xi}_{W^{1,\infty}}< 1\}\mapsto J((I + \theta)(A))\in\R,\]
and assume that $J_A(\theta)$ is twice Fr\'echet-differentiable in $0$.
Then
\begin{enumerate}
  \item[i.] there exists a continuous linear form $\ell_1[J](A)$ on ${C}^1(\partial A)$ such that 
  \[J_A'(0)\xi = \ell_1[J](A)(\xi|_{\partial A} \cdot \nu)\] for all 
  $\xi \in {C}^\infty(\mathbb{R}^n, \mathbb{R}^n)$, where ${\nu}$ denotes the unit exterior normal vector on $\partial A$.

  \item[ii.] If, moreover, $A$ has $C^3$ boundary, there exists a continuous symmetric bilinear form $\ell_2[J](A)$ on  ${C}^2(\partial A)\times {C}^2(\partial A)$ 
  such that for all $(\xi, \zeta) \in {C}^\infty(\R^n, \R^n)\times {C}^\infty(\R^n, \R^n)$
  \[
 {J}_A''(0)(\xi, \zeta) = \ell_2[J](A)(\xi \cdot \nu, \zeta \cdot \nu) 
  + \ell_1[J](A)\big(\mathbf{B}(\zeta_\tau, \xi_\tau) 
  - \nabla^\tau (\zeta \cdot \nu) \cdot \xi_\tau 
  - \nabla^\tau (\xi \cdot \nu) \cdot \zeta_\tau\big),
  \]
  where $\nabla^\tau$ is the tangential gradient, $\xi_\tau$ and $\zeta_\tau$ are the tangential components of $\xi$ and $\zeta$, and 
$\mathbf{B} = D^\tau \nu$ is the second fundamental form of $\partial A$.
\end{enumerate}
\end{teor}

The forms $\ell_1$ and $\ell_2$ are sometimes called the shape gradient and the shape Hessian.

\begin{defi}\label{def nearly spherical}
 Let $A$ be a bounded, open set with $C^2$ boundary, with unit outer normal $\nu=\nu_A$. For all $h\in C(\partial A)$  we denote by $A_h$ the normal deformation of $A$, defined through its boundary as  
\[\partial A_h:= \Set{ x \in \R^n : x = y+ h(y)\nu(y), \, y \in \partial A },\]
which is well defined as soon as $\norma{h}_{\infty}<c(A)$.
When $A$ is a ball, $A_h$ is called a  nearly spherical set.\medskip

Let $\mathcal{V}(E)=\abs{E}$ denote the Lebesgue measure of a measurable set $E$ in $\R^n$; for any $\delta>0$ and  any Banach space $X$ with $C^\infty(\partial A) \subset X \subset W^{1,\infty}(\partial A)$, we define the $X-$neighbourhood of $A$ with fixed measure $S_{\delta,X}$ as 
\begin{equation}
\label{nearly spherical class}
    S_{\delta,X}(A):=\Set{A_h: h\in X, \, \mathcal{V}(A_h)=\mathcal{V}(A),\,\text{and}\,\norma{h}_{X}\leq \delta}.
\end{equation}
\end{defi} 

\begin{oss} \label{oss calcolo l2}
Let $A$ and $J$ be as in \autoref{Teorema struttura} and let $h \in C^2(\partial A)$. Let $A_t=A_{th}$ be a normal deformation of $A$.
If we consider $j:I\subset \R\to \R$ the function defined as
\[
j(t) := J\left( A_t \right),
\]
then 
\[
j'(0) = \ell_1[J](A)  (h), \quad \text{ and } \quad j''(0) = \ell_2[J](A)(h,h).
\]
For further details, see \cite[Section $5.9.4$]{HP18shape}. 

\end{oss}
We recall the Hadamard formulas for the differentiation of integral functions on a variable domain.
\begin{prop}\label{Teorema HAdamart}
Let $A\subset \mathbb{R}^n$ be an open set with $C^2$ boundary and let 
\[
h \in C^{2,\gamma}(\partial A), \qquad \gamma>0.
\]
Let $I\subset \mathbb{R}$ be an interval with $0\in I$, and, for every $t\in I$, consider
\[
\Phi_t : x\in\mathbb{R}^n \longmapsto x + t\,\tilde{h}(x)\tilde{\nu}(x),
\]
where $\tilde{\nu}$ is a smooth extension of the unit outer normal to $\partial A$ and $\tilde{h}$ is an extension of $h$.
Define
\[
A_t := \Phi_t(A),
\]
and denote by $\nu_t$ the unit outer normal to $\partial A_t$, and by $H_t=H_{\partial A_t}$ its mean curvature understood as the sum of the principal curvatures.\medskip

Let $f(t,\cdot)\colon A_t\to\R$, $t\in I$, be a family of functions such that
\[
f(t,\cdot)\circ \Phi_ t\in C^1(I;  L^{1}(A))\cap C(I;W^{1,1}(A)),
\]

Then the map
\[
t \mapsto \int_{A_t} f(t,x)\,dx
\]
is of class $C^1$ on $I$, and for every $t\in I$ the following Hadamard formula holds:
\begin{equation}\label{derivata integrale su At}
\frac{d}{dt}\int_{A_t} f(t,x)\,dx
=
\int_{A_t} \dot{f}(t,x)\,dx
+
\int_{\partial A_t} f(t,\sigma)\,h_t(\sigma)\, d\Hn,
\end{equation}
where $h_t=(h \nu)\circ\Phi_t^{-1}\cdot  \nu_t$ and 
\[\dot{f}(t,x)=\left[\partial_t(f(t,\cdot)\circ \Phi_t)\right]\circ\Phi_t^{-1}(x)-\nabla f(t,x)\cdot\left[\partial_t\Phi_t\right]\circ\Phi^{-1}_t(x).\]
Similarly, if $g(t,\cdot)\colon \partial A_t\to\R$, $t\in I$, be a family of functions such that
\[
g(t,\cdot)\circ \Phi_t \in C^1\big(I; W^{1,1}(\partial A)\big),\]
the map
\[
t \mapsto \int_{\partial A_t} g(t,\sigma)\,d\Hn
\]
is of class $C^1$ on $I$, and for every $t\in I$ one has
\begin{equation}\label{derivata integrale partial >t}
    \frac{d}{dt}\int_{\partial A_t} g(t,\sigma)\,d\Hn
=
\int_{\partial A_t}
 \dot{g}(t,\sigma)
+
\Big(\partial_{\nu_t} g(t,\sigma)\,
+
\,H_t(\sigma)\,g(t,\sigma)\,
\Big)h_t(\sigma)\,d\Hn,\end{equation}
where 
\[\dot{g}(t,\sigma)=\left[\partial_ t(g(t,\cdot)\circ \Phi_t)\right]\circ\Phi_t^{-1}(\sigma)-\nabla \tilde g(t,\sigma)\cdot\left[\partial_ t\Phi_t\right]\circ\Phi^{-1}_t(\sigma),\]
and $\tilde g$ is an extension of $g$ in a neighbourhood of $\partial A_t$. 
\end{prop}

Let $P(A)$ denote the perimeter of a smooth set $A\subset\R^n$. We recall some well-known differential properties of $P$ and $\mathcal{V}$ (see for instance \cite[Chapter 5]{HP18shape}).
\begin{lemma}\label{lem:perimeter-volume-variations}
Let $A\subset\mathbb{R}^n$ be a smooth bounded set. Then
\begin{equation}\label{eq:first-variation-volume}
\ell_1[\mathcal V](A)(h)
=
\int_{\partial A}h\,d\Hn ,
\end{equation}
and
\begin{equation}\label{eq:first-variation-perimeter}
\ell_1[P](A)(h)
=
\int_{\partial A}Hh\,d\Hn .
\end{equation}
Moreover,
\begin{equation}\label{eq:second-variation-volume}
\ell_2[\mathcal V](A)(h,h)
=
\int_{\partial A}Hh^2\,d\Hn ,
\end{equation}
and
\begin{equation}\label{eq:second-variation-perimeter}
\ell_2[P](A)(h,h)
=
\int_{\partial A}
\left(
\left|{\nabla}^{\tau}h\right|^2
+
\left(H^2-|\B|^2\right)h^2
\right)\,d\Hn .
\end{equation}

\end{lemma}

Following \cite{DL19}, we introduce the assumptions \Csref\, and \ICref, as well as additional definitions required in our framework. They will be used to derive sufficient conditions for local optimality.

\begin{defi} \label{def ottimalità locale}
Let $J$ and $A$ be as in \autoref{Teorema struttura}. \begin{enumerate}
    \item $A$ is a \emph{critical shape} for $J$ under volume constraint if there exists $\Lambda\in \R$ such that
\[
\forall h \in C^\infty(\partial A), \quad \ell_1[J+ \Lambda\mathcal{V}](A) (h) = 0.
\]

\item If $A$ is a critical shape, we say that $A$ is a \emph{strictly stable shape}  under volume constraint if:

\begin{enumerate}[(a)]
    \item $\ell_2[J](A)$ extends continuously to $H^s(\partial A)$, for some $s\in(0,1]$;
    \item For all $h \in \mathcal{T}(\partial A) \setminus \{0\}$,
    \[
    \ell_2[J+\Lambda \mathcal{V}](A) (h,h) > 0,
    \]
    where
    \begin{equation}\label{def Ts}
        \mathcal{T}(\partial A) = \Set{ \varphi \in H^s(\partial A) : \int_{\partial A} \varphi\,d\Hn=0  }.
    \end{equation}
    
\end{enumerate}

\item Let $X$ be a Banach space with $C^\infty(\partial A) \subset X \subset W^{1,\infty}(\partial A)$. We say that $A$ is a $X-$strictly stable local minimum for $J$ under volume constraint if  there exist $\delta, c>0$ and $s\in(0,1]$ such that for all $h \in X$ with $A_h\in S_{\delta, X}(A)$, we have

\[
J(A_h) - J(A) \ge c \norma{h}_{H^s(\partial A)}^2.
\]
\end{enumerate}
\end{defi}

\noindent \textbf{Assumption} \Cslabel{}: Let $A$ be a bounded open set and let $s\in(0,1]$. 

We say that $J$ satisfies the condition \Csref \, at $A$ if the bilinear form $\ell=\ell_2[J](A)$ acting on $\mathcal{C}^\infty(\partial A)$ satisfies:

there exist $s_1\in[0,s)$ and $c_1>0$ such that $\ell=\ell_m+\ell_r$ with

\[
\begin{cases}
\ell_m \text{ is lower semi-continuous in } H^{s}(\partial A), \\[5pt]
\ell_m(\varphi, \varphi) \geq c_1 |\varphi|^2_{H^{s}(\partial A)}, \quad \forall \varphi \in \mathcal{C}^\infty(\partial A), \\[5pt]
\ell_r \text{ continuous in } H^{s_1}(\partial A),
\end{cases}
\]

where $|\cdot|_{H^{s}(\partial A)}$ denotes the $H^{s}(\partial A)$ semi-norm. In that case, $\ell$ is naturally extended (by density) to the space $H^{s}(\partial A)$.\medskip

\noindent\textbf{Assumption }\IClabel{}:
 Let $A$ be a bounded open set, $s\in(0,1]$, and $X\subset W^{1,\infty}(\partial A)$ a Banach space, such that $h\in X\mapsto J(A_h)$ is twice Fr\'echet differentiable at zero in $X$.
 
 We say that $J$ satisfies assumption \ICref \,  at $A$  if there exist $\eta > 0$ and a modulus of continuity $\omega$ such that for every $h\in X$ with $\|h\|_X \leq \eta$, and all $t \in [0,1]$:
\[
|j''(t) - j''(0)| \leq \omega(\|h\|_X)\, \|h\|_{H^s(\partial A)}^2,
\]
where $j : t \in [0,1] \mapsto j(A_{th})\in\R$.\medskip

Using the assumptions and the definitions introduced, the following theorem holds.
\begin{teor}\cite[Theorem 1.3]{DL19}
\label{DL}
Let $J$ and $A^*$ be as in \autoref{Teorema struttura}. Assume the following:
\begin{enumerate}[(a)]
    \item There exists $s \in (0,1]$ and a Banach space $X$ with $C^\infty(\partial A^*) \subset X \subset W^{1,\infty}(\partial A^*)$ such that $J$ satisfies assumptions \ICref\, and \Csref\, at $A^*$;
    \item $A^*$ is a \emph{critical shape} for $J$ under volume constraint;
    \item $A^*$ is a \emph{strictly stable shape}, under volume constraint, in $H^s(\partial A^*)$.
\end{enumerate}
Then, $A^*$ is an $X-$strictly stable local minimum under volume constraint for $J$ in the sense of \autoref{def ottimalità locale}. 
\end{teor}

For completeness's sake, we illustrate the main idea of the proof of \autoref{DL}.
Assume $A^*$ is a critical shape for a shape functional $J$ satisfying assumption \ICref\, and the following strict coercivity assumption on the shape Hessian
\begin{equation}\label{sca}\ell_2[J](A^*)(h,h)\ge C \norma{h}_{H^s(\partial A^*)}^2.\end{equation}

Then for any $\norma{h}_X<\eta$, by Taylor's formula, we have 
\[\begin{split}J(A^*_h)-J(A^*)&=j(1)-j(0)=j'(0)+\int_0^1 j''(t)(1-t)\,dt\\[10pt]
&=\dfrac{1}{2}j''(0)+\int_0^1 (j''(t)-j''(0))(1-t)\,dt\ge \left(\dfrac{C}{2}-\omega(\norma{h}_X)\right)\norma{h}_{H^s(\partial A^*)}^2.\end{split}\]
Then, there exists $\delta>0$ such that for every $h\in X$ with $\norma{h}_X<\delta<\eta$ we have
\[J(A^*_h)-J(A^*)\ge \dfrac{C}{4}\norma{h}_{H^s(\partial A^*)}^2.\]

On the other hand, for the constrained case, we use the assumptions on $A^*$ to construct an auxiliary functional for which assumption \ICref\, and \eqref{sca} hold. Namely, if $A^*$ is a critical shape under volume constraint for $J$, then it is a critical shape for the auxiliary functional 
\[J_K(A)=J(A)+\Lambda \mathcal{V}(A) + K (\mathcal{V}(A)-\mathcal{V}(A^*))^2.\]
Moreover, $J_K$ still satisfies assumption \ICref, and,  using assumption \Csref\, and the fact that $A^*$ is a strictly stable shape under volume constraint, one can prove that, for $K>0$ sufficiently large, $\ell_2[J_K](A^*)$  satisfies \eqref{sca} for some constant $C>0$. Hence, there exist $c,\delta>0$ such that
\[J(A^*_h)-J(A^*)+\Lambda(\mathcal{V}(A^*_h)-\mathcal{V}(A^*))+K(\mathcal{V}(A^*_h)-\mathcal{V}(A^*))^2\ge c \norma{h}_{H^s(\partial A^*)}^2,\]
for every $\norma{h}_X<\delta$, that is
for every $h$ such that $A_h^*\in S_{\delta,X}(A^*)$, we have
\[J(A_h^*)-J(A^*)\ge c\norma{h}_{H^s(\partial A^*)}^2.\]

\subsection{The Robin boundary condition}\label{RBC}
We now recall some basic properties of the Robin boundary value problem.\medskip

Let $\Om\subset\R^n$ be a bounded open set and let $\rho : \R^n \to \mathbb{R}$ be defined by
\[
\rho(x) =
\begin{cases}
1 & \text{if } x\in\Om, \\[5pt]
a & \text{if } x\in\R^n\setminus\Om.
\end{cases}
\]
With a slight abuse of notation, if $\Om$ is a ball $B_r$, we identify $\rho$ with its radial profile, that is, $\rho(\abs{x})=\rho(x)$.  With this convention, for every set $A\subset\R^n$ such that $\Om\ssubset A$ and every $v \in L^1(A)$, we can write
\[
\int_{\Om} v \, dx + a \int_{A \setminus \Om} v \, dx = \int_A \rho v \, dx.
\]

Fix $\beta>0$ and let $A\subset\R^n$ be a bounded open set with Lipschitz boundary, such that $\Om\ssubset A$, and  consider the bilinear form on $H^1(A)$ defined as
\[Q(u,v)=\int_A \rho \nabla u\nabla v\,dx+\beta\int_{\partial A}uv\,d\Hn.\]
By the Poincar\'e inequality with trace term and the assumption on $\rho$, we have that $Q$ is equivalent to the standard scalar product on $H^1(A)$ 
\[\langle u,v\rangle_{H^1}=\int_A \nabla u\nabla v\,dx+\int_A uv\,dx.\]
Hence, by the Lax--Milgram theorem, for every $f\in L^2(A)$, and $g\in L^2(\partial A)$ the boundary value problem 
\begin{equation}\label{difraction}\begin{cases}
    -\Delta v = f  & \text{in }\Om, \\[5 pt]
-a\Delta v =  f  & \text{in }A\setminus\bar\Om\\[5pt]
v^-=v^+ & \text{on } \partial \Om, \\[5 pt]
\partial_{\nu_\Om} {{v}}^- = a \partial_{\nu_\Om} {v}^+ & \text{on } \partial\Omega,\\[5 pt]
a \partial_{\nu_A} v + \beta v = g & \text{on } \partial A, 
\end{cases}\end{equation}
admits a unique weak solution, that is, a unique solution of the equation
\[Q(v,\psi)=\int_A f\psi\,dx+\int_{\partial A} g\psi\,d\Hn,\]
for every $\psi\in H^1(A)$. In particular, by the regularity theory of elliptic PDEs (we refer in particular to for instance \cite[Ch. 3 Theorem 3.2]{LU68}, \cite[Ch. 3 Theorem 16.2]{LU68}, and \cite[Theorem 1.1]{Z21}), if all the data of the problem are smooth, then so is the solution to \eqref{difraction} on the two phases $\bar\Om$ and $\bar A\setminus\bar\Om$. More precisely, the following theorem holds.
\begin{teor}\label{reg}
    Let $\Om,A\subset\R^n$ be bounded open sets of class $C^{2,\gamma}$ with $\Om\ssubset A$, let $f\in C^{0,\gamma}(\Om)\cap C^{0,\gamma}(A\setminus\bar{\Om})$, and let $g\in C^{1,\gamma}(\partial A)$. Then the function $v$, weak solution to the boundary value problem
    \begin{equation}\label{difraction2}\begin{cases}
    -\Delta v = f  & \text{in }\Om, \\[5 pt]
-a\Delta v =  f  & \text{in }A\setminus\bar\Om\\[5pt]
v^-=v^+ & \text{on } \partial \Om, \\[5 pt]
\partial_{\nu_\Om} {{v}}^- = a \partial_{\nu_\Om} {v}^+ & \text{on } \partial\Omega,\\[5 pt]
a \partial_{\nu_A} v + \beta v = g & \text{on } \partial A, 
\end{cases}\end{equation}
is of classes $C^{0,\gamma}(\bar A)$, and $C^{2,\gamma}(\bar{\Om})\cap C^{2,\gamma}(\bar{A}\setminus \Om)$. In particular, $v$ is a solution to \eqref{difraction2} in the classical sense.
\end{teor}

As noted in the introduction, by the classical theory of elliptic operators, and by the compactness of the embedding of $H^1(A)$ in $L^2(A)$, the eigenvalue problem
\[\begin{cases}
-\Delta v = {\lambda_\beta} v & \text{in }\Om, \\[5 pt]
-a\Delta v = \lambda_\beta v & \text{in }A\setminus\bar\Om\\[5pt]
v^-=v^+ & \text{on } \partial \Om, \\[5 pt]
\partial_{\nu_\Om} {{v}}^- = a \partial_{\nu_\Om} {v}^+ & \text{on } \partial\Omega,\\[5 pt]
a\partial_{\nu_A} v + \beta v = 0 & \text{on } \partial A, 
\end{cases}\]
admits a discrete spectrum of positive eigenvalues 
\[0<\lambda_{\beta,1}(\Om,A)\le\lambda_{\beta,2}(\Om,A)\le\dots\le\lambda_{\beta,k}(\Om,A)\le\dots\to+\infty\]
Moreover, the first eigenvalue $\lambda_{\beta}(\Om,A)=\lambda_{\beta,1}(\Om,A)$ can be variationally characterised  as
\begin{equation}\label{lambda variaz}\lambda_{\beta}(\Om,A)=\min_{v\in H^1(A)\setminus\set{0}}\dfrac{\displaystyle\int_A \rho \abs{\nabla v}^2\,dx+\beta\int_{\partial A} v^2\,d\Hn}{\displaystyle\int_A v^2\,dx}.\end{equation}
Then we can choose a first eigenfunction $u$ to be non-negative, and, by the maximum principle and the Hopf lemma, if $A$ is sufficiently smooth, in each connected component $u$ is either zero or strictly positive. In particular,  
if $A$ is connected, then $u$ is strictly positive up to the boundary of $A$ and $\lambda_{\beta}(\Om,A)$ is simple. Finally, if $\Omega$ and $A$ are concentric balls, by rotation invariance, the associated first eigenfunction is also radial.\medskip

We notice, in addition, that, by the variational characterisation \eqref{lambda variaz}, it follows that $\lambda_\beta(\Om,A)$ is increasing with respect to $\beta$ and with respect to $a$. Let $\lambda^D(\Om,A)$ be the first (two-phase) Dirichlet eigenvalue of $A$, i.e., the smallest $\lambda$ such that the problem
\begin{equation} \begin{cases} -\Delta v=\lambda v& \text{in }\Om\\[5pt]
-a\Delta v = \lambda v &\text{in }A\setminus\bar\Om\\[5pt] v^-=v^+&\text{on }\partial \Om\\[5pt] \partial_{\nu_\Om} v^-=a \partial_{\nu_\Om} v^+& \text{on } \partial \Om\\[5pt] v=0&\text{on } \partial A. \end{cases} \end{equation}
admits a non-trivial solution. $\lambda^D(\Om,A)$ can be variationally characterised as 
\begin{equation}\label{lambdaD variaz}
   \displaystyle{\lambda^D(\Om,A)=\min_{v\in H_0^1(A)}\dfrac{\displaystyle\int_A \rho \abs{\nabla v}^2\,dx}{\displaystyle\int_A v^2\,dx}}.
\end{equation}
Then, for every $\beta > 0$, one has
\begin{equation}
0 < \lambda_\beta(\Om,A) < \lambda^D(\Om,A).
\end{equation}
Moreover, one can easily prove that
\[
\lim_{\beta \to +\infty} \lambda_\beta(\Om,A) = \lambda^D(\Om,A),
\qquad
\lim_{\beta \to 0^+} \lambda_\beta(\Om,A) = 0.
\]

Finally, we recall the following simple lemma.
\begin{lemma}\label{usefullemma}
Let $\beta>0$, let $A\subset\R^n$ be a connected, bounded, open set with Lipschitz boundary such that $\Om\ssubset A$, and let $u$ be the positive first eigenfunction normalised in $L^2(A)$. Then, for every $f\in L^2(A)$, $F\in L^2(A,\R^n)$, $g\in L^2(\partial A)$, and $c\in\R$, there exists a unique couple $(w,s)\in H^1(A)\times\R$ such that $w$ is the unique weak solution to  
\begin{equation}\label{eq-lambda}\begin{cases}
    -\Delta w -\lambda_\beta w = f - \divv(F) - su & \text{in }\Om, \\[5 pt]
-a\Delta w -\lambda_\beta w =  f - \divv(F) -su & \text{in }A\setminus\bar\Om\\[5pt]
w^-=w^+ & \text{on } \partial \Om, \\[5 pt]
\partial_{\nu_\Om}{w}^- - a \partial_{\nu_\Om}{w}^+ = \left(F^--F^+\right)\cdot\nu_\Om& \text{on } \partial\Omega,\\[5 pt]
a\partial_{\nu_A} w + \beta w = g+F\cdot\nu_A & \text{on } \partial A\\[10pt]
\displaystyle\int_A wu \,dx=c.
\end{cases}\end{equation}
In particular, $s$ is given by the compatibility condition
\begin{equation}\label{compatibility}
s=\int_A \left(fu+F\cdot \nabla u\right)\,dx+\int_{\partial A} g u\,d\Hn.\end{equation}
Moreover, there exists $C=C(\Om,A,a,\beta)>0$ such that
\begin{equation}\label{H1estimate}\norma{w}_{H^1(A)}\le C\left(\norma{f}_{L^2(A)}+\norma{F}_{L^2(A)}+\norma{g}_{L^2(\partial A)}+\abs{c}\right).\end{equation}
\end{lemma}
\begin{proof}
    Consider the continuous bilinear form on $H^1(A)$
    \[Q_\lambda(v,\psi)=\int_A\rho \nabla v\nabla\psi\,dx+\beta\int_{\partial A}v\psi\,d\Hn-\lambda_\beta\int_Av\psi.\]
    Then $w\in H^1(A)$ is a weak solution of \eqref{eq-lambda} if and only if
    \begin{equation}\label{eq-lambdaweak}Q_\lambda(w,\psi)=\int_A \left(f\psi+F\cdot \nabla\psi\right)\,dx+\int_{\partial A} g\psi\,d\Hn-s\int_A u\psi,\end{equation}
    for every $\psi\in H^1(A)$ and
    \[\int_A wu\,dx=c.\] As $A$ is connected, $\lambda_\beta$ is simple, and we can decompose $w$ as  $w=w^\perp +c u$ with $w^\perp\in (\operatorname{span}\set{u})^\perp$. By definition, for every $\psi\in H^1(A)$ we have $Q_\lambda(u,\psi)=0$ so that necessarily
    \[s=\int_A \left(fu+F\cdot \nabla u\right)\,dx+\int_{\partial A} g u\,d\Hn,\]
    and \begin{equation}\label{eqperp}Q_\lambda(w^\perp,\psi)=\int_A \left(f\psi+F\cdot \nabla\psi\right)\,dx+\int_{\partial A} g\psi\,d\Hn,\end{equation}
    for every $\psi\in (\operatorname{span}\set{u})^\perp$.
    As a function of $\psi$, the right-hand side of \eqref{eqperp} is a continuous linear functional on $H^1(A)$, hence, the existence and uniqueness of $ w^\perp$ follow from the Lax--Milgram theorem once we show that, on $(\operatorname{span}\set{u})^\perp$,  $Q_\lambda$ is coercive. Indeed, for every $\psi\in (\operatorname{span}\set{u})^\perp$ we have 
    \[\int_A \psi^2\,dx\le \dfrac{\dint_A \rho\abs{\nabla\psi}^2\,dx+\beta\dint_{\partial A} \psi^2\,d\Hn }{\lambda_{\beta,2}}=\dfrac{ Q(\psi,\psi)}{\lambda_{\beta,2}},\]
    so that 
    \[Q_\lambda(\psi,\psi)\ge \dfrac{\lambda_{\beta,2}-\lambda_\beta}{\lambda_{\beta,2}} Q(\psi,\psi)\ge C_1 \norma{\psi}_{H^1(A)}^2.\]
    Finally, to prove \eqref{H1estimate} we use $w^\perp$ as a test function in \eqref{eqperp} to get
    \[\begin{split}C_1 \norma{w^\perp}^2_{H^1(A)}\le Q_\lambda(w^\perp,w^\perp)&=\int_A \left(f w^\perp+F\cdot \nabla w^\perp\right)\,dx+\int_{\partial A} g w^\perp\,d\Hn\\[15pt]&\le C_2 \norma{w^\perp}_{H^1(A)}\left[\norma{f}_{L^2(A)}+\norma{F}_{L^2(A)}+\norma{g}_{L^2(\partial A)}\right],
    \end{split}\]
    so that finally
    \[\norma{w}_{H^1(A)}\le\norma{w^\perp}_{H^1(A)}+\abs{c}\norma{u}_{H^1(A)}\le C\left(\norma{f}_{L^2(A)}+\norma{F}_{L^2(A)}+\norma{g}_{L^2(\partial A)}+\abs{c}\right).\]
\end{proof}

\section{Shape derivatives}\label{SDcomputed}

 Let $\Omega,A\subset \mathbb{R}^n$ be bounded, open sets with $C^{2,\gamma}$ boundary, for some fixed $\gamma\in(0,1)$, such that $\Om\ssubset A$, and assume $A$ to be connected. We will always assume that $\nu=\nu_A$ is extended to $\R^n$ as $\nabla d$, where $d=d_A$ is the signed distance from $\partial A$.\medskip
 
 In this section, we compute the first and second shape derivatives of the eigenvalue $\lambda_\beta(\Om,A)$ for normal deformations of $A$.  In the one-phase case, namely $a=1$, similar computations were carried out in \cite[Chapter 9]{BW} using a different formalism. In particular, the authors derive analogous expressions for the first and second derivatives evaluated at $0$.\medskip

 We start by proving the differentiability of the first eigenvalue and associated eigenfunction for smooth deformations of the boundary $\partial A$.

\begin{lemma}\label{differentiability lemma}
 Fix two open sets $\Om_1$ and $\Om_2$ such that $\Om\ssubset\Om_1\ssubset\Om_2\ssubset A$ and let $\chi\in C^\infty(\R^n,\R)$ be a cut-off function such that $\chi=0$ in $\Om_1$ and $\chi=1$ in $A\setminus\bar\Om_2$. Let $\eta=(2\norma{\chi}_{W^{1,\infty}})^{-1}$, then for every 
$\theta \in C^{2,\gamma}(\mathbb{R}^n,\mathbb{R}^n)$ with $\|\theta\|_{W^{1,\infty}} < \eta$, the map 
\[\Phi_\theta(x)=x+\chi(x)\theta(x)\]
is a diffeomorphism and $\Om\subset\subset \Phi_\theta(A)$. Let
\[
A_\theta := \Phi_\theta(A),
\]
and let $\lambda_\theta=\lambda_\beta(\Om,A_\theta)$ be the first eigenvalue of \eqref{eq forte autofunz} in $A_\theta$ and let $u_\theta$ denote the associated positive $L^2(A_\theta)$-normalised first eigenfunction. Finally let
\[
\hat{u}_\theta := u_\theta \circ \Phi_\theta.
\]
Then, the map
\[
\theta \in \Set{ \xi \in C^{2,\gamma}(\mathbb{R}^n,\mathbb{R}^n) |\, \norma{\xi}_{W^{1,\infty}} < \eta }\longmapsto ({\lambda}_\theta, \hat{u}_\theta) \in \mathbb{R} \times \left(C^{2,\gamma}(\bar{\Omega})\cap C^{2,\gamma}(\bar{A}\setminus\Omega)\right)
\]
is of class $C^\infty$ in a neighbourhood of $\theta=0$ and 
\[\abs{\lambda_0-\lambda_\theta}+\norma{u_0-\hat{u}_\theta}_{C^{2,\gamma}(\overline{\Om})}+\norma{u_0-\hat{u}_\theta}_{C^{2,\gamma}(\overline{A}\setminus\Om)}\le C\norma{\theta}_{C^{2,\gamma}}\]
\end{lemma}
\begin{proof}
    The proof is based on the proof of \cite[Theorem 5.7.4]{HP18shape}, using different boundary conditions. By our assumptions on $\Om, A$ and the map $\Phi_\theta$, from \autoref{reg}, we have that $u_\theta\in C^{0,\gamma}(\overline{A_\theta})\cap C^{2,\gamma}(\overline{\Om})\cap C^{2,\gamma}(\overline{A_\theta}\setminus\Om)$. By direct computations, and using the fact that $\Phi_\theta(x)=x$ on $\Om_1$, we have that the pull-back function $\hat{u}_\theta$ is a solution to 
    \[\begin{cases}
        -\Delta \hat{u}_\theta = \lambda_\theta \hat{u}_\theta &\text{in } \Om,\\[5 pt]
         -a\divv(M_\theta \nabla \hat{u}_\theta)=\lambda_\theta J_\theta \hat{u}_\theta &\text{in } A\setminus\bar\Om, \\[5 pt]
        \hat{u}_\theta^- = \hat{u}_\theta^+ &\text{on } \partial\Om,\\[5pt]
        \partial_{\nu_\Om} \hat{u}_\theta^- = a \partial_{\nu_\Om} \hat{u}_\theta^+ &\text{on } \partial\Om,\\[5pt]
        a(M_\theta \nabla \hat{u}_\theta)\cdot\nu_A + \beta J^\tau_\theta \hat{u}_\theta=0 &\text{on }\partial A,
    \end{cases}\]
    where $J_\theta$ and $J^\tau_\theta$ are respectively the Jacobian and the tangential Jacobian on $\partial A$ of $\Phi_\theta$ and 
    \[M_\theta = J_\theta \left[D \Phi_\theta\right]^{-1}\left[D \Phi_\theta\right]^{-T}.\]

    Let \[X:=\Set{v\in C^{0,\gamma}(\overline A)\,| \,\begin{aligned} &v\in C^{2,\gamma}(\overline{\Om})\cap C^{2,\gamma}(\overline{A}\setminus\Om),\\[5pt]
        &\text{and}\quad\partial_{\nu_\Om} v^-= a \partial_{\nu_\Om} v^+\: \text{on } \partial\Om
    \end{aligned} }, \]
    and consider the function 
    \[F\colon \Set{ \xi \in C^{2,\gamma}(\mathbb{R}^n,\mathbb{R}^n) |\, \norma{\xi}_{W^{1,\infty}} < \eta }\times X \times \R\to \left( C^{0,\gamma}(\overline{\Om})\cap C^{0,\gamma}(\overline{A}\setminus\Om)\right)\times C^{1,\gamma}(\partial A)\times \R, \]
    defined as
    \[F(\theta, v, s)=\left(\rho\divv(M_\theta \nabla v)+s J_\theta v,a(M_\theta \nabla v)\nu + \beta J^\tau_\theta v,\int_A v^2 J_\theta \,dx-1\right),\]
     so that 
     \[F(\theta,\hat{u}_\theta,\lambda_\theta)=0.\]

     We want to use the implicit function theorem to prove the assertion. By definition, $F$ is polynomial in $v$ and $s$, and $D\Phi_\theta$ is affine in $\theta$. Hence, by the regularity of the determinant and of the matrix inversion, we have that $F$  is of class $C^{\infty}$. We are then left to prove that the differential of $F$ in $(0,u_0,\lambda_\beta)$, with respect to the variables $(v,s)$, is an isomorphism between the spaces $X\times \R$ and $\left( C^{0,\gamma}(\overline{\Om})\cap C^{0,\gamma}(\overline{A}\setminus\Om)\right)\times C^{1,\gamma}(\partial A)\times \R$. We have 
     \[D_{v,s}F(0,u_0,\lambda_0)\colon (w,t)\mapsto \left(\rho \Delta w+\lambda_0w+tu_0, a\partial_\nu w+\beta w, 2\int_A u_0w\,dx\right).\]
     By \autoref{usefullemma}, for every $f\in C^{0,\gamma}(\overline{\Om})\cap C^{0,\gamma}(\overline{A}\setminus\Om)$, $g\in C^{1,\gamma}(\partial A)$ and $c\in\R$, the problem 
     \[\begin{cases}
         -\Delta w -\lambda_0 w =tu_0+f &\text{in }\Om,\\[5pt]
         w^-=w^+ &\text{on }\partial\Om,\\[5pt]
         \partial_{\nu_\Om}w^-=a\partial_{\nu_\Om}w^+ &\text{on }\partial\Om,\\[5pt]
         -a\Delta w -\lambda_0 w =tu_0+f  &\text{in }A\setminus\overline{\Om},\\[5pt]
         a\partial_\nu w+\beta w= g &\text{on }\partial A,\\[10pt]
         \dint_A wu_0\,dx=c 
     \end{cases}\]
admits a unique solution $(w,t)\in H^1(A)\times\R$; moreover, by \autoref{reg}, $w\in X$, hence $D_{v,s}F(0,u_0,\lambda_0)$ is invertible and, by continuity and the Banach inverse mapping theorem, an isomorphism.
\end{proof}

\begin{oss}
    We remark that $\lambda_\theta$ is also Fr\'echet differentiable in the space $W^{1,\infty}(\R^n,\R^n)$ and so is $\hat{u}_\theta$ as a function from $W^{1,\infty}(\R^n,\R^n)$ to $H^1(A)$. Indeed, we can apply the implicit function theorem to the map
    \[F\colon\set{\xi\in W^{1,\infty}(\R^n,\R^n)| \norma{\xi}_{W^{1,\infty}}
    <\eta}\times H^1(A)\times\R\to (H^1(A))'\times\R\]
    defined as
    \[F(\theta,v,s)=\left(L(\theta,v,s),\int_A v^2 J_\theta\,dx-1\right),\]
    where
     \[L(\theta,v,s)w=\int_A\rho (M_\theta \nabla v)\cdot\nabla w\,dx+\beta\int_{\partial A} J^\tau_\theta v w\,d\Hn-s\int_AJ_\theta vw\,dx.\]
\end{oss}

Before we compute the shape derivatives of the eigenvalue,  we fix the following notation to ease the computations.\medskip

By the regularity of $A$, the orthogonal projection, $\pi_{\partial A}$, on its boundary is well-defined and smooth in a neighbourhood, $U_A$, of $\partial A$. Hence, without loss of generality, we assume that the cut-off function $\chi$ in \autoref{differentiability lemma} is zero on $A\setminus \overline{U_A}$.  Then, for every $h\in C^{2,\gamma}(\partial A)$, we can define the map
\[\Phi_{th}(x)=x+t\chi(x) h(\pi_{\partial A}(x))\nu(x), \]
and denote by $\lambda_\beta(t)$ and $u_{th}$ the first eigenvalue and eigenfunction of \eqref{eq forte autofunz} in $A_{th}=\Phi_{th}(A)$. We denote by $\nu_{th}$ the normal $\nu_{A_{th}}$ which can be extended as
\[\nu_{th}=\dfrac{(D\Phi_{th})^{-T}\nu}{\norma{(D\Phi_{th})^{-T}\nu}}\circ\Phi_{th}^{-1}.\] 
When there is no ambiguity, we will drop the dependence on $h$, while, for $t=0$, we omit the subscript entirely. Before computing the derivatives of the eigenvalue, we remark that, with this assumption, the function $h_t$ in \autoref{Teorema HAdamart}, in a neighbourhood of $\partial A$, is given by
\[h_t=h\nu\cdot\nu_t,\]
so that, in particular, in said neighbourhood, we have
\[\dot {h_t}|_{t=0}=0,\quad\text{and}\quad\partial_\nu h=0.\]

\subsection{First Derivative}
We start by computing the first derivative of $\lambda_\beta(t)=\lambda_\beta(\Om,A_{th})$.

\begin{lemma}\label{lemma lambda'}
In the notations of \autoref{Teorema HAdamart}, we have that
  \begin{equation}\label{eq lambda'}
     {\lambda_\beta}'(t) = \int_{\partial A_t} \left( a|\nabla^\tau u_t|^2 - k_t u_t^2 \right) h_t\,d\Hn, 
  \end{equation}
where 
\begin{equation}\label{def k_t}
k_t=\lambda_\beta(t)-\beta H_t+\frac{\beta^2}{a}.
\end{equation}
Moreover,   $\dot u_t$ is the unique solution of 
\begin{equation}\label{eq forte u'}
\begin{cases}
-\Delta \dot u_t -\lambda_\beta(t) \dot u_t=\lambda'_\beta(t)u_t & \text{in }\Om, \\[5 pt]
-a\Delta \dot u_t -\lambda_\beta(t) \dot u_t=\lambda'_\beta(t)u_t & \text{in }A_t\setminus \bar\Om, \\[5 pt]
{\dot u_t}^-={\dot u_t}^+ & \text{on } \partial \Om, \\[5 pt]
\partial_{\nu_\Om}{\dot u_t}^- = a \partial_{\nu_\Om} {\dot u_t}^+& \text{on }\partial\Omega,\\[5 pt]
a\partial_{\nu_t} \dot u_t + \beta \dot u_t = k_t u_t  h_t+a\divv^\tau(h_t \nabla^\tau u_t)  & \text{on } \partial A_t, 
\end{cases}
\end{equation}
such that 
\begin{equation}\label{derivata norma 1}
    2\int_{A_t}u_t \dot u_t\,dx+\int_{\partial A_t}u_t^2h_t\, d\Hn=0.
\end{equation}
\end{lemma}
\begin{proof}   
The differentiability follows from \autoref{differentiability lemma}, hence we can use the differentiation formulas in \autoref{Teorema HAdamart}.\medskip

First of all, we notice that $\norma{u_t}_{L^2(A_t)}$ is constant, hence, by \eqref{derivata integrale su At}, we obtain that 
\[  2\int_{A_t}u_t \dot u_t\,dx+\int_{\partial A_t}u_t^2h_t\, d\Hn=0.
\]

 By the variational characterisation,
 \[{\lambda_\beta}(t)=\int_{A_t}\rho|\nabla u_t|^2\,dx+\beta\int_{\partial A_t} u_t^2\,d\Hn,\]
 so that, differentiating with respect to $t$, and using Hadamard formulas \eqref{derivata integrale su At} and \eqref{derivata integrale partial >t}, we obtain 
 \[{\lambda_\beta}'(t)=2\int_{A_t}\rho\nabla u_t \nabla \dot u_t\,dx+2\beta\int_{\partial A_t} u_t\dot u_t\,d\Hn+\int_{\partial A_t}a |\nabla u_t|^2h_t\,d\Hn+\beta\int_{\partial A_t}\left(H_t u_t^2+2u_t\partial_{\nu_t} u_t\right)h_t\,d\Hn.\]
 
Using the weak formulation of \eqref{eq forte autofunz} and \eqref{derivata norma 1}, we obtain 
\[2\int_{A_t}\rho\nabla u_t \nabla \dot u_t\,dx+2\beta\int_{\partial A_t} u_t\dot u_t\,d\Hn=2\lambda_\beta(t)\int_{A_t} u_t\dot{u_t}\,dx=-\lambda_\beta(t) \int_{\partial A_t} u_t^2h_t\,d\Hn.\]
Moreover, by the Robin boundary conditions, $a\partial_{\nu_t} u_t=-\beta u_t$ on $\partial A_t$, we have 
\[2\beta u_t \partial_{\nu_t} u_t=-a\left(\partial_{\nu_t} u_t\right)^2-\dfrac{\beta^2}{a}u_t^2. \]
Hence, 
  \[
{\lambda_\beta}'(t) = \int_{\partial A_t} \left( a|\nabla^\tau u_t|^2 - k_t u_t^2 \right) h_t\,d\Hn.
\]\medskip

We now want to obtain the equation for $\dot u_t$. For all $w\in C^\infty_c (\R^n)$ and for all $t$, the weak equation of $u_t$ ensures that 
\[\int_{A_t}\rho \nabla u_t \nabla w\,dx+\beta \int_{\partial A_t} u_t w\,d\Hn={\lambda_\beta(t)} \int_{A_t}u_t w\,dx.\]
Differentiating this equation, using Hadamard formulas \eqref{derivata integrale su At} and \eqref{derivata integrale partial >t},  and recalling the bilinear form introduced in \autoref{N&T}
\[Q_{\lambda}(v,w)=\int_{A_t}\rho \nabla v \nabla w\,dx+\beta \int_{\partial A_t} v w\,d\Hn-{\lambda_\beta}(t) \int_{A_t}v w\,dx, \]
we have that
\begin{equation}\label{eq u' pt 1}
\begin{split}
Q_{\lambda}(\dot u_t,w)=&{\lambda_\beta}'(t)\int_{A_t}u_tw\,dx+{\lambda_\beta}(t)\int_{\partial A_t} u_t w h_t\,d\Hn-\int_{\partial A_t} a \nabla u_t\nabla w h_t \,d\Hn+\\[10pt]&-\beta \int_{\partial A_t}\left(u_twH_t+\partial_{\nu_t} u_tw+ \partial_{\nu_t} w u_t\right)h_t \,d\Hn\\[10pt]
&={\lambda_\beta}'(t)\int_{A_t}u_tw\,dx+ \int_{\partial A_t} k_t u_t w h_t\,d\Hn-\int_{\partial A_t} a \nabla^\tau u_t\nabla^\tau w h_t \,d\Hn\\[10pt]
&={\lambda_\beta}'(t)\int_{A_t}u_tw\,dx+ \int_{\partial A_t}\left[ k_t u_t  h_t+a\divv^\tau(h_t \nabla^\tau u_t)\right]w\,d\Hn,
\end{split}
\end{equation}
where we used the Robin boundary conditions of $u_t$ and the fact that
\[\int_{\partial A_t} \divv^\tau(F)\,d\Hn=0,\]
if $F$ is tangential to the closed surface $\partial A_t$. We remark that \eqref{eq lambda'} is exactly the compatibility condition \eqref{compatibility}. 
\end{proof}

\begin{oss}\label{oss derivata prima palla}
If $\Om= B_r$ and $A=B_R$ with $R>r$, we have that $u$ is radial, hence
  \begin{equation}
     \ell_1[\lambda_\beta](B_R)(h) = -k_0U(R)^2\int_{\partial B_R} h\,d\Hn, 
  \end{equation}
where $U\colon[0,R]\to\R$ is the radial profile of $u$, and
\[
k_0=\lambda_\beta-\beta H_{\partial B_R}+\frac{\beta^2}{a}=\lambda_\beta-\beta \dfrac{n-1}{R}+\frac{\beta^2}{a},
\]
is constant.
Then, recalling that (see \autoref{lem:perimeter-volume-variations})
\[\ell_1[\mathcal{V}](B_R)(h)=\int_{\partial A}h\,d\Hn,\]
we have that, in the notation of \autoref{def ottimalità locale}, $B_R$ is a critical shape under volume constraint for $\lambda_\beta(B_r,\cdot)$ and the Lagrangian for the associated problem is
\[\mathcal{L}(A)=\lambda_\beta(B_r,A)+\Lambda \mathcal{V}(A),\]
where 
\[\Lambda=k_0U(R)^2.\]
\end{oss}

\begin{prop}\label{oss k0}
    Let $\Om=B_r$, $A=B_R$, and $a\leq1$. Consider $k_0$ the quantity defined in \eqref{def k_t}
 for $t=0$, then $k_0>0$. In particular, the function 
 \[R\in(r,+\infty)\mapsto\lambda_\beta(B_r,B_R)\in\R,\]
 is strictly decreasing.
 \end{prop}
\begin{proof}
The case $a=1$ was proven in \cite[Example 9.1]{BW}. We adapt the same ideas for the two-phase case.
Since $\Om$ and $A$ are concentric balls, by radial symmetry the first eigenfunction $u=u_0$ is radial, thus 
\[
u(x)=U(|x|)\qquad \text{for all }x\in B_R,
\]
where, $U$ is a (bounded) solution to
\[
  \begin{cases}
        -U''(s)-\dfrac{n-1}{s}U'(s)=\lambda_\beta U(s) & \text{in}\quad (0,r)\\[5pt]
           -aU''(s)-a\dfrac{n-1}{s}U'(s)=\lambda_\beta U(s) & \text{in}\quad (r,R)\\[5pt]
        U(r^-)=U(r^+)\\[5pt]
        U'(r^-)=a U'(r^+)\\[5pt]
        a U'(R)+\beta U(R)=0\\[5pt]
        U>0
        \end{cases}
\]
In particular we have that in $(0,r)$
\[U(s)=c s^{-\frac{n-2}{2}}J_{\frac{n-2}{2}}\left(\sqrt{\lambda_\beta}s\right),\]
where $c$ is a positive constant and $J_{m}$ denotes the Bessel function of the first kind of order $m$.\medskip

We define the auxiliary function
\begin{equation}\label{def z}
    z(s)=\rho(s)\frac{U'(s)}{U(s)},\qquad s\in[0,R].
\end{equation}

Then $z$ is continuous on $[0,R]$ and differentiable on $(0,R)\setminus\{r\}$. A direct computation gives
\[
z'(s)=\begin{cases}\dfrac{U''(s)}{U(s)}
      -\left(\dfrac{U'(s)}{U(s)}\right)^2&\text{in }(0,r),\\[10pt]
      a\dfrac{U''(s)}{U(s)}
      -\dfrac{1}{a}\left(a\dfrac{U'(s)}{U(s)}\right)^2&\text{in }(r,R).
    
\end{cases}
\]

Using the equation for $U$, we have that
 $z$ solves
\begin{equation}\label{eq:z}
\begin{cases}z'(s)+z(s)^2+\dfrac{n-1}{s}z(s)+{\lambda_\beta}=0 &\text{in }(0,r),\\[10pt]
z'(s)+\dfrac{1}{a}z(s)^2+\dfrac{n-1}{s}z(s)+{\lambda_\beta}=0&\text{in }(r,R).
\end{cases}
\end{equation}

We note that, using the Robin Boundary condition at $s=R$ and the explicit expression of $U$ near $s=0$, we have
\[
z(0)=0,\qquad z(R)=-\beta,\qquad z'(0)<0,\qquad z'(R)=-k_0.
\]

Let $r^*$ be the minimum of $z$. As $z'(0)<0$, $r^*\in(0,R]$ and $z(r^*)<z(0)=0$. In addition, $r^*\not=r$, indeed, using \eqref{eq:z} we deduce that 
\[z'(r^-)-z'(r^+)=\dfrac{1-a}{a}z(r)^2.\] 

Hence, the left derivative at $r$ is larger than the right derivative, ensuring that $r$ is not a minimum point for $z$.

We are now ready to argue that $r^*=R$. In fact, if we assume by contradiction that this is not the case, then, as $r^*\ne0,r,R$, we deduce that
\[z'(r^*)=0.\]
Differentiating \eqref{eq:z} at $r^*$ yields
\[
z''(r^*)=\frac{n-1}{(r^*)^2}z(r^*)<0,
\]
which contradicts the fact that $r^*$ is a minimum.\medskip

Thus, we deduce that $r^*=R$. In particular, this ensures that $k_0=-z'(R)\geq 0$. Finally, if $k_0=0$, we could repeat the same contradiction argument in $R$, obtaining $z'(R)=0$ and $z''(R)<0$, hence a contradiction.  Thus, we deduce that $k_0>0$. \medskip

In particular, taking as perturbation $h\equiv 1$, we have that 
\[\dfrac{d}{d R} \lambda_\beta(B_r,B_{R})=\dfrac{d}{d t} \lambda_\beta(B_r,B_{R+t})|_{t=0}=-\int_{\partial B_R} k_0 u_0^2\,d\Hn<0,\]
that is $R\mapsto\lambda_\beta(B_r,B_R)$ is strictly decreasing.
\end{proof}

\begin{oss}
    The case $a>1$ is more intricate, since $k_0$ can actually change sign. For instance, fix $a>1$ and $R>0$, then we recall that (see \cite{GS05})
    \[\lim_{\beta\to0^+}\dfrac{\lambda_\beta(B_r,B_R)}{\beta}=\dfrac {P(B_R)}{|B_R|}=\dfrac{n}{R},\]
    thus, we have that 
    \[\lim_{\beta\to 0^+}\dfrac{k_0}{\beta}=\dfrac{n}{R}-H_{\partial B_R}=\dfrac{1}{R}>0.\]
    Moreover, 
    \[\lim_{\beta \to +\infty}\dfrac{k_0}{\beta^2}=\dfrac{1}{a}>0.\]
    Hence, in both cases $k_0>0$.\medskip

    On the other hand, as 
    \[\lambda_\beta(B_r,B_R)<\lambda^D(B_r),\]
    where $\lambda^D(B_r)$ is the first (one-phase) Dirichlet eigenvalue on $B_r$, we have that 
    \[k_0<\lambda^D(B_r)-\beta H_{\partial B_R}+\dfrac{\beta^2}{a}.\]
    Thus taking, for instance, \[\beta>\dfrac{\lambda^D(B_r)}{H_{\partial B_R}},\]
    and 
    \[a>\dfrac{\beta^2}{\beta H_{\partial B_R}-\lambda^D(B_r)},\]
    we have that 
    \[k_0<0.\]\medskip

    However, also note that, if $U''(r^-)<0$, then $k_0>0$. Indeed, arguing as in \autoref{oss k0}, from the condition $U''(r^-)<0$, we have that $r$ cannot be a minimum for the function $z$. Thus, we can reproduce the arguments in \autoref{oss k0} to prove $k_0>0$. This will be interesting later, since the condition $U''(r^-)<0$ will play a key role in proving that $B_R$ is a strictly stable shape under volume constraint (see \autoref{lem: sign U''}). 
\end{oss}

\subsection{Second Derivative}

We start by computing the second derivative of ${\lambda_\beta}(t)$ at $t=0$. In the following, we explicitly write the dependence of $\dot{u}_t|_{t=0}$ on the function $h$ as $\dot{u}[h]$. Recall the definition of the bilinear form
\[Q_\lambda(v,\psi)=\int_A\rho \nabla v\nabla\psi\,dx+\beta\int_{\partial A}v\psi\,d\Hn-\lambda_\beta\int_A v\psi.\]

\begin{lemma}\label{lem:second-variation}
In the notations of \autoref{Teorema HAdamart}, we have that
\begin{align}
\ell_2[\lambda_\beta](A)(h,h)
={}&
-2Q_{\lambda}(\dot u[h],\dot u[h])
-2\ell_1[\lambda_\beta](A)(h)\int_{\partial A}u^2h\,d\Hn
\notag\\
&+
\int_{\partial A}
\left(
H_{\partial A}-\frac{2\beta}{a}
\right)
\left[
a\left|\nabla^{\tau}u\right|^2
-k_0 u^2
\right]h^2\,d\Hn
\notag\\
&+
\beta\int_{\partial A}
\left(
\left|\nabla^{\tau}h\right|^2
-
|\B|^2h^2
\right)u^2\,d\Hn
\notag\\
&-
2a\int_{\partial A}
\B\left(
\nabla^{\tau}u,
\nabla^{\tau}u
\right)h^2\,d\Hn .
\label{eq:second-variation-general}
\end{align}
  
\end{lemma}

\begin{proof}

Set $H=H_{\partial A}$. We start from the first variation formula 
\begin{equation}\label{eq:first-variation-lambda}
\lambda'_\beta(t)
=
\int_{\partial A_t}
\left[
a\left|\nabla^{\tau}u_t\right|^2
-k_t u_t^2
\right]h_t\,d\Hn .
\end{equation}

Recall that on $\partial A$
\[\dot{h_t}|_{t=0}=0,\quad\text{and}\quad\partial_\nu h=0.\]
Then, differentiating \eqref{eq:first-variation-lambda} with respect to $t$, and
using the Hadamard formula for boundary integrals \eqref{derivata integrale partial >t}, at $t=0$, we obtain
\begin{equation}\label{l''eq1}\begin{split}\lambda''_\beta(0)=&2\int_{\partial A} \left[a\nabla^\tau u \cdot(\nabla^\tau u)^\cdot-2k_0u\dot u\right]h\,d\Hn-\int_{\partial A}(\dot{k_0}+\partial_\nu k_0 h)u^2h\,d\Hn\\[10pt]
&+\int_{\partial A}H\left[
a\left|\nabla^{\tau}u\right|^2
-k_0 u^2
\right]h^2\,d\Hn+\int_{\partial A}\left[a \partial_\nu\abs{\nabla^\tau u}^2-2k_0 u\partial\nu h\right]h^2,d\Hn,
\end{split}\end{equation}\medskip

Recall 
\[\nabla^\tau u_t=\nabla u_t-(\nabla u_t\cdot\nu_t)\nu_t,\]
so that
\[(\nabla^\tau u)^\cdot=\nabla\dot{u}-(\nabla \dot u\cdot\nu)\nu-(\nabla u\cdot \dot\nu_t|_{t=0})\nu-(\nabla u\cdot\nu)\dot\nu_t|_{t=0}.\]
By \cite[Proposition 5.4.14]{HP18shape} we know that, for a unit norm extension of the normal $\nu_t$,
\begin{equation}\label{nu'}\dot\nu_t|_{t=0}=-\nabla^\tau h.\end{equation}
Hence,
\[a\nabla^\tau u\cdot (\nabla^\tau u)^\cdot=a\nabla^\tau\cdot u\nabla^\tau \dot u+a\partial_\nu u\nabla^\tau u\cdot\nabla^\tau h=a\nabla^\tau u\cdot\nabla^\tau \dot u-\beta u \nabla^\tau u\cdot\nabla^\tau h,\]
and, substituting back in \eqref{l''eq1}, 
\begin{equation}\label{eq:second-variation-intermediate}\begin{split}
\lambda''_\beta(0)
=&
2\int_{\partial A}
\left[
a\,\nabla^{\tau}u\cdot\nabla^{\tau}\dot u
-k_0 u\dot u
\right]h\,d\Hn
-2\beta\int_{\partial A}
uh\,\nabla^{\tau}u\cdot\nabla^{\tau}h\,d\Hn
\\[10pt]
&
+
\int_{\partial A}
\left(
\beta\dot H-\lambda'_\beta+\beta\partial_\nu H\,h
\right)u^2h\,d\Hn 
\\[10pt]
&+\int_{\partial A}
H\left[
a\left|\nabla^{\tau}u\right|^2
-k_0 u^2
\right]h^2\,d\Hn+
\int_{\partial A}
\left[
a\,
\partial_\nu\abs{\nabla^{\tau}u}^2
-2k_0 u\,\partial_\nu u
\right]h^2\,d\Hn,
\end{split}
\end{equation}
where we also used that 
\[\dot{k_0}=\lambda_\beta'-\beta\dot{H},\quad\text{And}\quad\partial_\nu k_0=-\beta\partial_\nu H.\]

Recall that $\B=D^\tau\nu$ and $D\nu\,\nu=0$. Then, a direct computation shows
\[\partial_\nu \left(\nabla u\right)-\nabla(\partial_\nu u)=D\nu \nabla u=\B\left(\nabla^\tau u,\cdot\right),\]
so that we have \begin{equation}\label{eq:normal-tangential-gradient-recalled}\begin{split}
\partial_\nu
\left|\nabla^{\tau}u\right|^2
=&\partial_\nu\left[\abs{\nabla u}^2-\left(\partial_\nu u\right)^2\right]\\[5pt]
=&2\nabla u\cdot \nabla( \partial_\nu u)-2\B\left(\nabla^\tau u,\nabla^\tau u\right)-2\partial_\nu u\,\partial^2_{\nu\nu}u\\[5pt]
=&2\,\nabla^{\tau}u\cdot
\nabla^{\tau}(\partial_\nu u)
-
2\B\left(
\nabla^{\tau}u,
\nabla^{\tau}u
\right).
\end{split}\end{equation}
By the Robin boundary conditions, \eqref{eq:normal-tangential-gradient-recalled}, gives
\begin{equation}\label{eq:normal-tangential-gradient-boundary}
a\partial_\nu
\left|\nabla^{\tau}u\right|^2
=
-2\beta
\left|\nabla^{\tau}u\right|^2
-
2a\B\left(
\nabla^{\tau}u,
\nabla^{\tau}u
\right).
\end{equation}
By \eqref{nu'} we get
\begin{equation}\label{eq:Hdot-recalled}
\dot H
= \divv(\dot\nu_t)|_{t=0}=
-\Delta_{\tau}h.
\end{equation}
While, using the formula
\[0=\dfrac{\Delta(\abs{\nabla d}^2)}{2}=\nabla(\Delta d)\nabla d+\abs{D^2 d}^2=\partial_\nu H+\abs{\B}^2,\]
we have
\begin{equation}\label{nuH}\partial_\nu H=-\abs{\B}^2.\end{equation}

Substituting \eqref{eq:normal-tangential-gradient-boundary}, 
\eqref{eq:Hdot-recalled}, and \eqref{nuH} into \eqref{eq:second-variation-intermediate} we have
\begin{equation}\label{eq:second-variation-intermediate2}\begin{split}
\lambda''_\beta(0)
=&
2\int_{\partial A}
\left[
a\,\nabla^{\tau}u\cdot\nabla^{\tau}\dot u
-k_0 u\dot u
\right]h\,d\Hn
-2\beta\int_{\partial A}
uh\,\nabla^{\tau}u\cdot\nabla^{\tau}h\,d\Hn
\\[10pt]
&
-
\int_{\partial A}
\left(
\beta\Delta_\tau h+\lambda'_\beta+\beta\abs{\B}^2\,h
\right)u^2h\,d\Hn 
\\[10pt]
&+\int_{\partial A}
\left(H-\dfrac{2\beta}{a}\right)\left[
a\left|\nabla^{\tau}u\right|^2
-k_0 u^2
\right]h^2\,d\Hn-
2a\int_{\partial A}
\B(\nabla^\tau u,\nabla^\tau u)h^2\,d\Hn.
\end{split}\end{equation}\medskip

Integrating by parts, we have
\begin{equation}\label{parts1}\int_{\partial A}\nabla^\tau u\cdot \nabla^\tau \dot u\,h\,d\Hn=-\int_{\partial A}\dot u \divv^\tau(h\nabla^\tau u)\,d\Hn,\end{equation}
and
\begin{equation}\label{parts2}
-\int_{\partial A}u^2h\Delta_\tau h\,d\Hn-2\int_{\partial A}uh\nabla^\tau u\cdot \nabla^\tau h\,d\Hn=\int_{\partial A} u^2\abs{\nabla^\tau h}^2\,d\Hn.
\end{equation}

On the other hand, using the equation \eqref{eq forte u'}--\eqref{derivata norma 1}  satisfied by
$\dot u$, gives
\begin{equation}\label{2Q}\begin{split}2Q_\lambda(\dot u,\dot u)&=2\lambda_\beta'\int_{A}u\,\dot u\,dx+ 2\int_{\partial A}\left[ k_0 u  h+a\divv^\tau(h \nabla^\tau u)\right]\dot u\,d\Hn\\[10 pt]&=-\lambda_\beta'\int_{\partial A}u^2 h\,d\Hn+ 2\int_{\partial A}\left[ k_0 u  h+a\divv^\tau(h \nabla^\tau u)\right]\dot u\,d\Hn\end{split}\end{equation}
Substituting \eqref{parts1},\eqref{parts2}, and \eqref{2Q} into \eqref{eq:second-variation-intermediate2}, we finally get
\[\begin{split}
\lambda''_\beta(0)
=&
-2Q_{\lambda}(\dot u,\dot u)
-2\lambda_\beta'\int_{\partial A}u^2h\,d\Hn
+
\int_{\partial A}
\left(
H-\frac{2\beta}{a}
\right)
\left[
a\left|\nabla^{\tau}u\right|^2
-k_0 u^2
\right]h^2\,d\Hn
\\[10pt]
&+
\beta\int_{\partial A}
\left(
\left|\nabla^{\tau}h\right|^2
-
|\B|^2h^2
\right)u^2\,d\Hn
-
2a\int_{\partial A}
\B\left(
\nabla^{\tau}u,
\nabla^{\tau}u
\right)h^2\,d\Hn .
\end{split}\]
\end{proof}

\begin{oss}\label{oss lambda'' palla}
      In particular, if $\Om=B_r$  and $A=B_R$,
\[
H_{\partial B_R}=\frac{n-1}{R},
\qquad
|\B|^2=\frac{n-1}{R^2}
\]
hence
\[\begin{split}
\ell_2[\lambda_\beta](B_R)(h,h)
=&
-2Q_{\lambda}(\dot u[h],\dot u[h])
-2U^2(R)\ell_1[\la_\beta](B_R)(h)\int_{\partial B_R}h\,d\Hn\\[10pt]
&
-k_0 U^2(R)\left(H_{\partial B_R}-\dfrac{2\beta}{a}\right)
\int_{\partial B_R}
h^2\,d\Hn
+
\beta U^2(R)\int_{\partial B_R}
\left(
\left|\nabla^{\tau}h\right|^2
-
\dfrac{n-1}{R^2}h^2
\right)\,d\Hn,
\end{split}\]
where $U(\abs{x})=u(x)$, is the radial profile of the eigenfunction.\medskip

Using the second variation formulas for perimeter and volume on the ball (see \autoref{lem:perimeter-volume-variations}), we have
\[\ell_2[\mathcal{V}](B_R)(h,h)=H_{\partial B_R} \int_{\partial B_R} h^2\,d\Hn,\]
and
\[
\ell_2[P-H_{\partial B_R}  \mathcal{V}](B_R)(h,h)
=
\int_{\partial B_R}
\left(
\left|\nabla^{\tau}h\right|^2
-
\frac{n-1}{R^2}h^2
\right)\,d\Hn ,
\]
so that we can rewrite the shape Hessian of the eigenvalue as
\[\begin{split}
\ell_2[\lambda_\beta](B_R)(h,h)
=&
-2Q_{\lambda}(\dot u[h],\dot u[h])
-2U^2(R)\ell_1[\la_\beta](B_R)(h)\int_{\partial B_R}h\,d\Hn+\dfrac{2\beta k_0 U(R)^2}{a}
\int_{\partial B_R}
h^2\,d\Hn\\[10pt]
&
-k_0 U^2(R)\ell_2[\mathcal{V}](B_R)(h,h)
+
\beta U^2(R)\ell_2\left[P-H_{\partial B_R}\mathcal{V}\right](B_R)(h,h).
\end{split}\]
Recalling that, by \autoref{oss derivata prima palla}, the Lagrangian for the problem under volume constraint is
\[\mathcal{L}(A)=\lambda_\beta(B_r,A)+\Lambda\mathcal{V}(A),\]
with $\Lambda=k_0 U^2(R)$, we finally have
\begin{equation}
\label{Lag''}\begin{split}\ell_2[\mathcal{L}](B_R)(h,h)=&-2Q_\lambda(\dot u[h],\dot u[h])-2U^2(R)\ell_1[\la_\beta](B_R)(h)\int_{\partial B_R} h\,d\Hn+\dfrac{2\beta \Lambda}{a}\int_{\partial B_R} h^2\,d\Hn\\[10pt]&+\beta U^2(R)\ell_2\left[P-H_{\partial B_R}\mathcal{V}\right](B_R)(h,h).\end{split}\end{equation}
Finally, notice that 
\[P(A)-H_{\partial B_R}\mathcal{V}(A),\]
is exactly the Lagrangian for the minimisation of the perimeter under the volume constraint \[\mathcal{V}(A)=\mathcal{V}(B_R).\]
\end{oss}

\begin{oss}
  We remark that, given the chosen extension of $h$ and $\nu$, we have that, if $s\norma{h}_\infty$ is sufficiently small, 
  \[\Phi_t(\Phi_s(\sigma))=\Phi_{s+t}(\sigma),\]
  for every $\sigma\in\partial A$. In particular, then, 
   \[\partial A_{t+s}=\set{\sigma +s h(\sigma)\nu(\sigma)|\,\sigma\in\partial A_t},\]
  so that, computing $\lambda''_\beta(t)$ along the path $t\mapsto\Phi_t(A)$, is equivalent to compute
  \[\dfrac{d^2}{ds^2}\lambda_\beta\left(\Om,(Id+s\xi)(A_t)\right){\Large|}_{s=0},\]
  for $\xi=h\nu$ on $\partial A_t$. Finally, by the structure theorem (\autoref{Teorema struttura}), 
  it is then sufficient to compute the first and second derivatives, at zero, of normal deformations of $A_t$. Indeed, in the notation of \autoref{Teorema struttura}, we have that 
   \[\lambda''_\beta(t)=\ell_2[\lambda_\beta](A_t)(h\nu\cdot\nu_t, h\nu\cdot\nu_t ) 
  + \ell_1[\lambda_\beta](A_t)\left(h^2\left[\mathbf{B}_t(\gamma_t, \gamma_t)+2\gamma_t\cdot\nabla^{\tau_t}\alpha_t \right]
  +2 h\alpha_t \gamma_t\cdot\nabla^{\tau_t} h\right),\]
  where $\alpha_t=\nu\cdot\nu_t$ and $\gamma_t=\nu-\alpha_t\nu_t$.
  Thus, by \autoref{lem:second-variation} we immediately obtain the following
\end{oss}

\begin{prop}\label{prop lambda''} In the notation of \autoref{Teorema HAdamart}, if in addition $A$ has $C^3$ boundary, we have that
    \[\begin{split}\displaystyle
    \lambda''_\beta(t)=&-2\left[\int_{A_t}\rho \abs{\nabla \dot{u_t}}^2\,dx+\beta\int_{\partial A_t} \dot{u_t}^2\,d\Hn-\lambda_\beta(t)\int_{A_t} \dot{u_t}^2\,dx-2\lambda'_\beta(t)\int_{ A_t} u_t\dot{u_t}\,dx\right]\\[10pt]
        &+\int_{\partial A_t} \left[a\abs{\nabla^{\tau_t} u_t}^2-k_tu_t^2\right]\left[\mathbf{B}_t(\gamma_t, \gamma_t)+2\gamma_t\cdot\nabla^{\tau_t}\alpha_t + \left(H_t-\dfrac{2\beta}{a}\right)\alpha_t^2\right]h^2\,d\Hn\\[10pt]&-2a\int_{\partial A_t} \mathbf{B}_t\left(\nabla^{\tau_t}u_t,\nabla^{\tau_t}u_t\right)h^2\alpha_t^2\,d\Hn+\beta\int_{\partial A_t}\left(\abs{\nabla^{\tau_t}\alpha_t}^2-\abs{\mathbf{B}_t}^2\alpha_t^2\right)u^2_th^2\,d\Hn\\[10pt]
        &+2\int_{\partial A_t} \left[a\abs{\nabla^{\tau_t} u_t}^2-k_tu_t^2\right]h\alpha_t \gamma_t\cdot \nabla^{\tau_t} h\,d\Hn+\beta\int_{\partial A_t}(2\alpha_t h\nabla^{\tau_t}h\cdot\nabla^{\tau_t}\alpha_t+\alpha_t^2\abs{\nabla^{\tau_t}h}^2)u_t^2\,d\Hn.
    \end{split}\]
\end{prop}

\section{Proof of the main theorems}\label{PMT}
In this section, we prove \autoref{teor: main a<1} and \autoref{teor: main a>1}. In order to prove these results, we use \autoref{DL}; hence, we need to prove that the shape functional $\lambda_\beta(\Om,\cdot)$ satisfies assumptions \ICref,\Csref \, and characterise the assumptions under which $B_R$ is a strictly stable shape under volume constraint. We start with the following lemma proving \Csref.

\begin{lemma}\label{lem:Hs-condition}
    Let $\Om= B_r$ and $A=B_R$ with $R>r$. Then the shape functional $\lambda_\beta(B_r,\cdot)$ satisfies assumption \Csref\, at $B_R$ with $s=1$.
\end{lemma}
\begin{proof}
By \autoref{oss lambda'' palla}, we have that
\[\begin{split}
\ell_2[\lambda_\beta](B_R)(h,h)
=&
-2Q_{\lambda}(\dot u[h],\dot u[h])
-2U^2(R)\ell_1[\la_\beta](B_R)(h)\int_{\partial B_R}h\,d\Hn\\[10pt]
&
-k_0 U^2(R)\left(H_{\partial B_R}-\dfrac{2\beta}{a}\right)
\int_{\partial B_R}
h^2\,d\Hn
+
\beta U^2(R)\int_{\partial B_R}
\left(
\left|\nabla^{\tau}h\right|^2
-
\dfrac{n-1}{R^2}h^2
\right)\,d\Hn,
\end{split}\]
where we recall that in the symmetric case
\[\ell_1[\lambda_\beta](B_R)(h)=-k_0U^2(R)\int_{\partial B_R}h\,d\Hn,\]
and $\dot{u}=\dot{u}[h]$ satisfies
\[\begin{cases}
    -\Delta \dot u -\lambda_\beta \dot u=\lambda'_\beta u &\text{in }B_r,\\[5pt]
    -a\Delta \dot u -\lambda_\beta \dot u=\lambda'_\beta u &\text{in }B_R\setminus\bar{B_r},\\[5pt]
    \dot u^- = \dot u^+ &\text{on }\partial B_r,\\[5pt]
    \partial_\nu \dot u^- = a\partial\nu \dot u^+ &\text{on }\partial B_r,\\[5pt]
    a\partial_\nu \dot u+\beta \dot u=k_0U(R) h &\text{on }\partial B_R,\\[10pt]
    2\dint_{B_R} u\dot u+U(R)^2\int_{\partial B_R}h\,d\Hn=0. 
\end{cases}\]
In particular, $\ell_1[\lambda_\beta](B_R)$ is continuous with respect to the $L^2(\partial B_R)-$norm of $h$, and, by \autoref{usefullemma}, we have that \[\norma{\dot u[h]}_{H^1(B_R)}\le c \norma{h}_{L^2(\partial B_R)},\]
for some constant $c>0$. 
Hence, $Q_\lambda(\dot u[h],\dot u[h])$ is a continuous bilinear form on $L^2(\partial B_R)$. 

Finally, setting 
\[\begin{split}\ell_r(h,h)=&-2Q_{\lambda}(\dot u[h],\dot u[h])
-2U^2(R)\ell_1[\la_\beta](B_R)(h)\int_{\partial B_R}h\,d\Hn\\[10pt]
&
-k_0 U^2(R)\left(H_{\partial B_R}-\dfrac{2\beta}{a}\right)
\int_{\partial B_R}
h^2\,d\Hn
-
\dfrac {\beta U^2(R)(n-1)}{R^2}\int_{\partial B_R}
h^2
\,d\Hn,
\end{split}\]
and
\[\ell_m(h,h)=\beta U^2(R)\int_{\partial B_R}
\left|\nabla^{\tau}h\right|^2\,d\Hn,\]
we immediately have that $\ell_r$ is continuous in $L^2(\partial B_R)$ while $\ell_m$ is equivalent to the $H^1(\partial B_R)$ semi-norm. That is $\lambda_\beta$ satisfies assumption \Csref\, at $B_R$ for $s=1$ and $s_1=0$. 
\end{proof}
\begin{oss}
    Let us notice that, for general sets $\Om$ and $A$, the estimate 
    \[\norma{\dot{u}[h]}_{H^1(A)}\le C\norma{h}_{L^2(\partial A)},\]
    does not hold. Indeed the right-hand side of the Robin boundary condition on $\partial A$ contains the term \[a\divv^\tau(h \nabla^\tau u),\] so that $\dot{u}[h]$ is continuous with respect to the $H^1(\partial A)-$norm of $h$ but, in general, not the $L^2(\partial A)-$norm. 
\end{oss}

The characterisation of the assumptions under which $B_R$ is a strictly stable shape under volume constraint for $\lambda_\beta(B_r,\cdot)$, and the proof of the improved continuity of the second variation (assumption \ICref), is more intricate. Thus, we defer the proofs of the following two propositions to \autoref{coercivity} and   \autoref{improved}, respectively.

\begin{prop}\label{lem:corcivity condition}
 Let $R>r>0$ and let \[\mathcal{T}(\partial B_R) = \Set{ h \in H^1(\partial B_R) \colon\, \int_{\partial B_R} h\,d\Hn=0  }.\]
   If 
   \[(1-a)\lambda_\beta(B_r,B_R)<(1-a)\lambda^N(B_r)\quad \text{and}\quad k_0>0,\]
    then there exists $c>0$ such that, for all $h\in \mathcal{T}(\partial B_R)\setminus\set{0}$,
    \[\ell_2[\lambda_\beta+\Lambda \mathcal{V}](B_R)(h,h)> c\norma{h}_{H^1(\partial B_R)}^2.\]
    If either $k_0<0$ or \[(1-a)\lambda_\beta(B_r,B_R)>(1-a)\lambda^N(B_r),\]
    then there exists $h_1\in\mathcal{T}(\partial B_R)\setminus\set{0}$ such that  
    \[\ell_2[\lambda_\beta+\Lambda \mathcal{V}](B_R)(h_1,h_1)<0,\]
    in particular, for any vector $\mathbf{v}\ne0$ we can choose 
    \[h_1(x)=\dfrac{\mathbf{v}\cdot x}{\abs{x}}.\]
\end{prop}

\begin{prop}\label{lem:Ic-condition}
    Let $\Om,A\subset\R^n$ be bounded opens sets with $\Om\ssubset A$ and assume in addition that $\Om$ has $C^{2,\gamma}$ boundary and that $A$ is connected and with $C^3$ boundary. The shape functional $\lambda_\beta(\Om,\cdot)$ satisfies condition \ICref\, at $A$ with $s=1$ and $X=C^{2,\gamma}(\partial A)$.
\end{prop}
We are now ready to prove \autoref{teor: main a<1} and \autoref{teor: main a>1}.
\begin{proof}[Proof of \autoref{teor: main a<1} and \autoref{teor: main a>1}]
Fix $\mathbf{v}$ to be a unit vector.
By the structure theorem (\autoref{Teorema struttura}), we have that the first and second derivatives at $t=0$ of the function  
\[j(t)=\lambda_\beta(B_r,B_R+t\mathbf{v})\]
are exactly 
\[j'(0)=\ell_1[\lambda_\beta](B_R)(\mathbf{v}\cdot\nu)\]
and
\[j''(0)=\ell_2[\lambda_\beta](B_R)(\mathbf{v}\cdot\nu,\mathbf{v}\cdot\nu)+\ell_1[\lambda_\beta](B_R)(\B(\mathbf{v},\mathbf{v})-2\nabla^\tau (\mathbf{v}\cdot\nu)\cdot\mathbf{v}).\]
Let $h=\mathbf{v}\cdot\nu$, then 
\[\int_{\partial B_R} h\,d\Hn=0,\]
and
\[Z=\B(\mathbf{v},\mathbf{v})-2\nabla^\tau (\mathbf{v}\cdot\nu)\cdot\mathbf{v}=-\B(\mathbf{v},\mathbf{v})=\dfrac{h^2-1}{R}.\]
Since $B_R$ is a critical shape under volume constraint we have that $j'(0)=0$, while
\[\begin{split}j''(0)&=\ell_2[\lambda_\beta](B_R)(h,h)+\ell_1[\la_\beta](B_R)(Z)=\ell_2[\lambda_\beta](B_R)(h,h)-\Lambda\ell_1[\mathcal{V}](B_R)(Z)\\[5pt]&=\ell_2[\lambda_\beta+\Lambda\mathcal{V}](B_R)(h,h)-\Lambda\left(\ell_2[\mathcal{V}](B_R)(h,h)+\ell_1[\mathcal{V}](B_R)(Z)\right)\\[5pt]&=\ell_2[\lambda_\beta+\Lambda\mathcal{V}](B_R)(h,h)-\Lambda\dfrac{d^2}{dt^2}\left(\mathcal{V}(B+t\mathbf{v})\right)|_{t=0}=\ell_2[\lambda_\beta+\Lambda\mathcal{V}](B_R)(h,h).
\end{split}\]
Then by \autoref{lem:corcivity condition}, if either $k_0<0$ or \[(1-a)\lambda_\beta(B_r,B_R)>(1-a)\lambda^N(B_r),\]
we have that  
\[\ell_2[\lambda_\beta+\Lambda\mathcal{V}](B_R)(h,h)<0.\]
Hence $j''(0)<0$, and there exists $\eta>0$ such that 
\[\lambda_\beta(B_r,B_R+t\mathbf{v})<\lambda_\beta(B_r,B_R)\]
for every $t\in(0,\eta)$.
\medskip

 On the other hand, if \[(1-a)\lambda_\beta(B_r,B_R)<(1-a)\lambda^N(B_r)\quad \text{and}\quad k_0>0,\]
 by \autoref{lem:Hs-condition}, \autoref{lem:corcivity condition}, and \autoref{lem:Ic-condition}  the shape functional $\lambda_\beta(B_r,\cdot)$ satisfies all the assumptions of \autoref{DL}.
Thus, recalling that $k_0>0$ when $a<1$ (\autoref{oss k0}), we deduce the assertions.
\end{proof}

\begin{oss}
     We remark that the constant $\eta$ in \autoref{teor: main a<1} and \autoref{teor: main a>1} can be chosen independently of the vector $\mathbf{v}$. Indeed, by \autoref{lem:Ic-condition}, $\la_\beta(B_r,\cdot)$ satisfy \ICref{}, while, as we will see in the proof \autoref{prop: break e segno f} (see in particular equation \eqref{remketa}), 
     $\ell_2[\lambda_\beta\mathcal{V}](B_R)(h_1,h_1)$
     depends on $\mathbf{v}$ only through its modulus.
\end{oss}

\subsection{Coercivity of the Lagrangian}\label{coercivity}
In this section, we prove \autoref{lem:corcivity condition}.
To study the coercivity of the bilinear form $\ell_2[\lambda_\beta+\Lambda\mathcal{V}](B_R)$, we start by introducing an associated eigenvalue problem of Steklov-type. Then, in \autoref{prop: break e segno f} we characterise the coercivity of the bilinear form in terms of the principal modified Steklov eigenvalue $\sigma_\beta$. Finally, in \autoref{prop: sign f} we relate the previous condition to an inequality between
$\lambda_\beta(B_r,B_R)$ and the principal Neumann eigenvalue, $\lambda^N(B_r)$, of the ball $B_r$.

\subsubsection*{A modified Steklov eigenvalue problem}\label{section steklov}

We now introduce an auxiliary Steklov-type spectral problem. As mentioned, this problem is
related to the bilinear form $\ell_2[\lambda_\beta+\Lambda\mathcal{V}](B_R)$. To be more precise, it is associated with the quadratic form $Q_\lambda$, which appears naturally in the
formula for the second variation of the first Robin eigenvalue. The
purpose of this construction is to have a spectral decomposition on
 $\partial A$, which will later allow us to characterise the coercivity condition for the second derivative of the eigenvalue. \medskip 

Recall that $Q_\lambda$ is non-negative, and $Q_\lambda(v,w)=0$ for every $w\in H^1(A)$, if and only if $v$ is a first Robin eigenfunction. We define the Steklov-type eigenvalue problem associated with $Q_\lambda$ as
follows: find $\sigma\in\mathbb R$ and a non-trivial
$\phi\in H^1(A)$ such that
\begin{equation}\label{eq:steklov_Qlambda_weak}
Q_\lambda(\phi,v)
=
\sigma
\int_{\partial A}\phi v\,d\mathcal H^{n-1}
\qquad
\forall v\in H^1(A).
\end{equation}
Equivalently, in strong form, sufficiently regular eigenfunctions solve
\begin{equation}\label{eq:steklov_Qlambda_strong}
\begin{cases}
-\Delta \phi - \lambda_\beta \phi=0, & \text{in } \Omega,\\[5pt]
-a\Delta \phi - \lambda_\beta \phi=0, & \text{in } A\setminus\bar\Om,\\[5pt]
\phi^-=\phi^+, & \text{on } \partial\Omega,\\[5pt]
\partial_{\nu_\Omega}\phi^-=
a\partial_{\nu_\Omega}\phi^+, & \text{on } \partial\Omega,\\[5pt]
a\partial_{\nu_A}\phi+\beta \phi
=
\sigma\phi, & \text{on } \partial A.
\end{cases}
\end{equation}
We have the following lemma.
\begin{lemma}
The Steklov-type problem \eqref{eq:steklov_Qlambda_weak} admits a discrete
spectrum
\[
0=\sigma_{0}(\Om,A)<\sigma_{1}(\Om,A)\leq\cdots\le\sigma_{k}(\Om,A)\le\cdots\to+\infty,
\]
where the eigenvalues are repeated according to their multiplicity.

Moreover, there exists a sequence of eigenfunctions
$\{\phi_j\}_{j\geq0}\subset H^1(A)$ such that their traces form a complete orthonormal
system of $L^2(\partial A)$. The eigenspace corresponding to $\sigma_{0}=0$ is
 one-dimensional and is generated by $u$. 

\end{lemma}
\begin{proof} Let $\lambda=\lambda_\beta(\Om,A)$, fix $\mu>0$ and set \[ Q_{\lambda,\mu}(v,w) := Q_\lambda(v,w) + \mu\int_{\partial A}vw\,d\mathcal H^{n-1}. \] We claim that $Q_{\lambda,\mu}$ is coercive, i.e. there exists $c_\mu>0$ such that 
\begin{equation}\label{eq:Q_lambda_mu_coercive} Q_{\lambda,\mu}(v,v) \geq c_\mu\|v\|_{H^1(A)}^2 \qquad \forall v\in H^1(A). \end{equation}
Suppose by contradiction that \eqref{eq:Q_lambda_mu_coercive} is false. Then there exists a sequence $\{v_k\}_{k\in\mathbb N}\subset H^1(A)$ such that \[ \|v_k\|_{H^1(A)}=1 \] and \[ Q_{\lambda,\mu}(v_k,v_k)\to 0. \] 
Since $Q_\lambda$ is non-negative, this implies \[ Q_\lambda(v_k,v_k)\to 0, \qquad \int_{\partial A}v_k^2\,d\mathcal H^{n-1}\to 0. \] Up to a subsequence, then, there exists $v\in H^1(A)$ such that \[ v_k\rightharpoonup v \quad\text{weakly in }H^1(A). \] 
Hence, $v_k$ converges to $v$ strongly in $L^2(A)$ and $L^2(\partial A)$; and by weak lower semicontinuity of $Q_{\lambda,\mu}$, we have that \[ 0\le Q_\lambda(v,v) \leq \liminf_{k\to\infty}Q_\lambda(v_k,v_k) = 0, \]
and
\[\int_{\partial A} v^2\,d\Hn=\lim_{k\to\infty}\int_{\partial A} v_k^2\,d\Hn=0.\]
Then, \[ v\in \operatorname{span}\{u\}, \] 
with $v_{\mid\partial A}=0$. However, since $u$ is positive in $A$ and $u_{\mid\partial A}\not\equiv 0$, this implies 
\[ v=0. \] 

Consider now the bilinear form \[ Q(v,w) = \int_A \rho \nabla v\cdot\nabla w\,dx + \beta\int_{\partial A}vw\,d\mathcal H^{n-1}, \]

We have
\[ Q(v_k,v_k) = Q_\lambda(v_k,v_k) + \lambda\int_A v_k^2\,dx\to0. \] 
then, recalling that $Q$ is equivalent to the standard scalar product on $H^1(A)$, we have
 \[ \|v_k\|_{H^1(A)}\to0, \] which contradicts the normalisation $\|v_k\|_{H^1(A)}=1$. 
Thus \eqref{eq:Q_lambda_mu_coercive} holds and we can now apply the Lax--Milgram theorem. \medskip
 
For every $f\in L^2(\partial A)$ consider the continuous linear functional  \[ L_f(w):= \int_{\partial A}fw\,d\mathcal H^{n-1}, \qquad w\in H^1(A) \]
and let $v_f\in H^1(A)$ be the unique solution of 
\begin{equation}\label{eq:shifted_problem} Q_{\lambda,\mu}(v_f,w) = \int_{\partial A}fw\,d\mathcal H^{n-1}, \end{equation} 
for any $w\in H^1(A)$. Notice, in addition, that $v_f$ satisfies the estimate \begin{equation}\label{vfcont} \|v_f\|_{H^1(A)}\le C_\mu \dfrac{Q_{\lambda,\mu}(v_f,v_f)}{\|v_f\|_{H^1(A)}} \leq C_\mu \dfrac{\abs{L_f(v_f)}}{\|v_f\|_{H^1(A)}} \leq C_\mu \|f\|_{L^2(\partial A)}. \end{equation}
We consider the continuous linear map which associates each $f$ in $L^2(\partial A)$ to the trace of the function $v_f$, namely, we define
\[ T_\mu:f\in L^2(\partial A)\mapsto v_{f}|_{\partial A}\in L^2(\partial A). \]  We notice that the map $T_\mu$ is compact. Indeed, by \eqref{vfcont}, the map \[ f\mapsto v_f \] is continuous from $L^2(\partial A)$ into $H^1(A)$, while the trace embedding 
 \[ H^1(A)\to L^2(\partial A) \] is compact. 
 Moreover, by the symmetry of $Q_{\lambda,\mu}$, the operator $T_\mu$ is self-adjoint in $L^2(\partial A)$. Indeed, if $f,g\in L^2(\partial A)$ and $v_f,v_g$ are the corresponding solutions of 
\eqref{eq:shifted_problem}, then \[ \int_{\partial A} f\,T_\mu g\,d\mathcal H^{n-1} = \int_{\partial A} f v_g\,d\mathcal H^{n-1} = Q_{\lambda,\mu}(v_f,v_g) = \int_{\partial A} g v_f\,d\mathcal H^{n-1} = \int_{\partial A} T_\mu f\,g\,d\mathcal H^{n-1}. \]
Finally, we observe that $T_\mu$ is injective. Indeed, if $T_\mu f=0$, by definition $v_{f\mid\partial A}=0$, so that by  \eqref{eq:shifted_problem},
we get \[ Q_{\lambda,\mu}(v_f,v_f) = \int_{\partial A}fv_f\,d\mathcal H^{n-1} = 0. \] 
By coercivity, then, $v_f=0$. Hence  \eqref{eq:shifted_problem} reads
\[ \int_{\partial A}fw\,d\mathcal H^{n-1}=Q(v_f,w)=0,\] 
for every $w\in H^1(A)$.
Hence, by density, we conclude that $f=0$. \medskip

By the spectral theorem for compact self-adjoint operators, there exists a complete orthonormal system $\{\eta_j\}_{j\geq0}$ of $L^2(\partial A)$ made of eigenfunctions of $T_\mu$, with 
corresponding positive eigenvalues $\{\tau_j\}_{j\geq0}$. Thus
\[ T_\mu\eta_j=\tau_j\eta_j, \qquad \tau_j>0, \qquad \tau_j\to0. \] 
For every $j\ge0$, let $\phi_j\in H^1(A)$ be the solution of \eqref{eq:shifted_problem} with $f=\tau_j^{-1}\eta_j$, that is 
\[Q_{\lambda,\mu}(\phi_j,w)=\dfrac{1}{\tau_j}\int_{\partial A}\eta_jw\,d\Hn.\]
By definition of $T_\mu$ \[ \phi_j|_{\partial A} = \tau_j^{-1}T_\mu\eta_j = \eta_j. \] 
Then, $\phi_j\ne0$, and
\[ Q_{\lambda,\mu}(\phi_j,w) = \frac1{\tau_j} \int_{\partial A}\phi_j w\,d\mathcal H^{n-1} \]
for all $w\in H^1(A)$. That is  \[ Q_\lambda(\phi_j,w) = \left(\frac1{\tau_j}-\mu\right) \int_{\partial A}\phi_j w\,d\mathcal H^{n-1} \]
for all $w\in H^1(A)$. Therefore $\phi_j$ is an eigenfunction of \eqref{eq:steklov_Qlambda_weak} with eigenvalue \begin{equation}\label{Sspecj}\sigma_j=\frac1{\tau_j}-\mu. \end{equation} Since $\tau_j\to0$, the corresponding Steklov eigenvalues satisfy \[ \sigma_j\to+\infty. \] 
Thus the Steklov eigenvalue problem admits a discrete spectrum. Notice that every eigenvalue $\sigma$ of \eqref{eq:steklov_Qlambda_weak} is of the type \eqref{Sspecj} for some $j\ge0$. Indeed, if $\phi$ is an eigenfunction associated to $\sigma$ then $\phi=v_f$ for $f=(\sigma+\mu)\phi|_{\partial A}$, hence letting $\eta=\phi|_{\partial A}$, $\eta\ne0$, or else $\phi$ would be zero, and 
\[T_\mu \eta = (\sigma+\mu)^{-1} \eta,\]
that is 
\[(\sigma+\mu)^{-1}=\tau_j\]
for some $j\ge0$.\medskip

We now identify the bottom of the spectrum. 

Let $\sigma_{0}$ be the smallest eigenvalue of \eqref{eq:steklov_Qlambda_weak} and let $\phi_0$ be an associated $L^2(\partial A)$-normalised eigenfunction. Then we have  \[ 0\le Q_\lambda(\phi_0,\phi_0) = \sigma_{0} \int_{\partial A}\phi_0^2\,d\mathcal H^{n-1}=\sigma_{0}. \]  
Therefore no negative Steklov eigenvalue occurs. On the other hand, if $u>0$ is the first $L^2(A)$-normalised Robin eigenfunction, by definition
\[ Q_\lambda(u,v)=0 \qquad \forall v\in H^1(A). \]
That is, $u$ is an eigenfunction of \eqref{eq:steklov_Qlambda_weak} with eigenvalue $\sigma=0$. Hence $\sigma_{0}=0$. 
Moreover, $\sigma_{0}$ is simple. Indeed, if $\phi$ is an eigenfunction with eigenvalue $\sigma_{0}=0$, then \[ Q_\lambda(\phi,w)=0, \]
for every $w\in H^1(A)$, that is, $\phi$ is a first Robin eigenfunction and therefore is proportional to $u$. 
\end{proof}

In the following, we will denote the principal eigenvalue $\sigma_{1}(\Om,A)$ by $\sigma_{\beta}(\Om,A)$ and drop the dependence of the sets when no ambiguity arises.

\begin{prop}\label{steklov palla}
    
Let $\Omega=B_r$ and $A=B_R$. For every $k\in\N_0$, let   $\set{Y_{k,l}}_{l=1}^{d_k}$ be an orthonormal basis for the space of spherical harmonics of degree $k$ on $\partial B_R$ and let $V_k$ be the bounded solution to
\begin{equation}\label{vk}\begin{cases}
    
        -V_k''(s)-\dfrac{n-1}{s}V_k'(s)+\dfrac{k(n-2+k)}{s^2}V_k(s) =\lambda_\beta V_k & \text{in}\quad (0,r)\\[10pt]
            -aV_k''(s)-a\dfrac{n-1}{s}V_k'(s)+a\dfrac{k(n-2+k)}{s^2}V_k(s) =\lambda_\beta V_k & \text{in}\quad (r,R)\\[5pt]
        V_k(r^-)=V_k(r^+)\\[5pt]
        V_k'(r^-)=a V_k'(r^+)\\[5pt]
      V_k(R)=1.
\end{cases}\end{equation} 
Then, the family of functions
\[\phi_{k,l}(x)=V_k\left(\abs{x}\right) Y_{k,l}\left(\dfrac{x}{\abs{x}}\right),\]
forms a complete $L^2(\partial B_R)-$orthonormal system of eigenfunctions of problem \eqref{eq:steklov_Qlambda_strong}. Moreover, $V_k>0$ in $(0,R]$, the
 eigenvalue corresponding to $\phi_{k,l}$ is \[\sigma_{k,l}=aV'_k(R)+\beta,\] 
has multiplicity $d_k$ and is monotone in $k$. In particular, the eigenspace associated with the principal eigenvalue $\sigma_\beta$ is generated by the system
\[\Set{V_1\left(\abs{x}\right)\dfrac{x_l}{\abs{x}}\colon\, l=1,\dots n}.\]
\end{prop}

\begin{proof}
   Following \cite[Lemma 8.2]{BW}, it is clear that the family of functions $\set{\phi_{k,l}}$  is an orthonormal system of eigenfunctions with corresponding eigenvalues
   \[\sigma_{k,l}=aV'_k(R)+\beta.\]
   
   Assume by contradiction that the system is not complete, and let $\phi$ be an eigenfunction in the $L^2(\partial B_R)-$orthogonal complement of the span of the family. Then
   \[\int_{\partial B_R}\phi Y_{k,l}\,d\Hn=0,\]
   for every $k\in \N_0$ and $1\le l\le d_k$. Since the spherical harmonics form a complete orthonormal system, necessarily $\phi|_{\partial B_R}=0$, which is in contradiction with $\phi$ being an eigenfunction. Hence, the system is complete.\medskip

    Let us briefly remark on the existence and uniqueness of the solution, $V_k$, to \eqref{vk} and on its positivity. It is well known from the theory of Bessel functions that the general solution of the equation 
    \[-v''(s)-\dfrac{n-1}{s}v'(s)+\dfrac{k(n-2+k)}{s^2} v(s)=\eta v(s)\]
    is given by the linear span of the  two independent solutions
    \[v_{k,i}(s;\eta)=s^{-\frac{n-2}{2}} \zeta_{k,i}\left(\sqrt{\eta} s\right),\]
    where $i=1,2$ and $\zeta_{k,i}$ is a Bessel function of the first kind ($i=1$) or second kind ($i=2$), of order $(n-2)/2+k$. Moreover, in any neighbourhood of $s=0$, only $v_{k,1}$ is bounded (and corresponds to the radial part of an $H^1$ function). Notice that this also entails that bounded solutions must satisfy $v'(0)=0$.\medskip
    
    The existence and uniqueness of a bounded solution to \eqref{vk}, then, is equivalent to the existence and uniqueness of a solution to the linear system of equations
    \[\begin{cases}
    c_1 v_{k,1}(r;\lambda_\beta)-&c_2v_{k,1}(r;\lambda_\beta/a)-c_3v_{k,2}(r;\lambda_\beta/a)=0\\[5pt
    ]
       c_1 v'_{k,1}(r;\lambda_\beta)-&c_2av'_{k,1}(r;\lambda_\beta/a)-c_3av'_{k,2}(r;\lambda_\beta/a)=0\\[5pt
    ]   &c_2v_{k,1}(R;\lambda_\beta/a)+c_3v_{k,2}(R;\lambda_\beta/a)=1.\\[5pt
    ]
    \end{cases}\]
    We then conclude by showing that the associated homogeneous system admits only the trivial solution. Indeed, the existence of a non-trivial solution to 
     \[\begin{cases}
    c_1 v_{k,1}(r;\lambda_\beta)-&c_2v_{k,1}(r;\lambda_\beta/a)-c_3v_{k,2}(r;\lambda_\beta/a)=0\\[5pt
    ]
       c_1 v'_{k,1}(r;\lambda_\beta)-&c_2av'_{k,1}(r;\lambda_\beta/a)-c_3av'_{k,2}(r;\lambda_\beta/a)=0\\[5pt
    ]   &c_2v_{k,1}(R;\lambda_\beta/a)+c_3v_{k,2}(R;\lambda_\beta/a)=0,\\[5pt
    ]
    \end{cases}\]
    entails that there exists $\tilde{V}\ne0$ solution to 
    \[\begin{cases}
    
        -\tilde{V}''(s)-\dfrac{n-1}{s}\tilde{V}'(s)+\dfrac{k(n-2+k)}{s^2}\tilde{V}(s) =\lambda_\beta \tilde{V}(s) & \text{in}\quad (0,r)\\[10pt]
            -a\tilde{V}''(s)-a\dfrac{n-1}{s}\tilde{V}'(s)+a\dfrac{k(n-2+k)}{s^2}\tilde{V}(s) =\lambda_\beta \tilde{V}(s) & \text{in}\quad (r,R)\\[5pt]
        \tilde{V}(r^-)=\tilde{V}(r^+)\\[5pt]
        \tilde{V}'(r^-)=a \tilde{V}'(r^+)\\[5pt]
      \tilde{V}(R)=0.
\end{cases}\]
Hence, the function $\tilde{V}(\abs{x})Y_{k,l}\left(\frac{x}{\abs{x}}\right)$ would be a Dirichlet eigenfunction with eigenvalue 
\[\lambda_\beta=\lambda_\beta(B_r,B_R)<\lambda^D(B_r,B_R),\]
which is a contradiction.\medskip

By the same argument, the functions $V_k$ solutions to \eqref{vk} are necessarily strictly positive on $(0,R]$. Indeed, arguing by contradiction we have $V_k(\bar s)=0$ for some $\bar{s}\in(0,R]$, then $V_k(\abs{x})Y_{k,l}\left(\frac{x}{\abs{x}}\right)$ is a Dirichlet eigenfunction in $B_{\bar s}$ of eigenvalue $\lambda_\beta$. Where the Dirichlet problem is the usual one-phase problem if $\bar{s}\le r$ and the two-phase one is $\bar{s}>r$. In both case, the inequality $\lambda_\beta<\lambda^D(B_r,B_R)$ contradicts the monotonicity of Dirichlet eigenvalues with respect to domain inclusion.\medskip

 Finally, let us prove that $\sigma_{k,l}$ is monotone in $k$. Let 
 $S_k=k(n-2+k),$
 and notice that, by \eqref{vk}, in $(0,R)\setminus\set{r}$ 
\begin{equation}\label{Vkcontracted}\rho(s)(s^{n-1} V_k'(s))'=\left(\dfrac{\rho(s)S_k}{s^2}-\lambda_\beta\right)s^{n-1}V_k(s).\end{equation}
 Let $j\in\N$ and Consider the weighted Wronskian of $V_k$ and $V_{k+j}$
 \[W(s):=\rho(s)s^{n-1}\left(V_k'(s)V_{k+j}(s)-V_k(s)V_{k+j}'(s)\right).\]
 Then, by the transmission condition $W$ is continuous in $[0,R]$, and by \eqref{Vkcontracted}, in $(0,R)\setminus\set{r}$ we have
 \[W'(s)=\rho(s)s^{n-3}(S_{k}-S_{k+j})V_k(s)V_{k+j}(s)<0.\]
 Then, as $W(0)=0$ we deduce that 
 \[0=W(0)>W(R)=aR^{n-1}(V'_k(R)-V_{k+j}'(R)),\]
 that is 
 \[\sigma_{k+j,i}>\sigma_{k,l}\]
 for every $1\le l\le d_k$ and $1\le i\le d_{k+j}$. In particular, for every $k\in\N_0$, the eigenvalue $\sigma_{k,l}$ has exactly multiplicity $d_k$.
\end{proof}

\subsubsection*{Second-order necessary condition}
For the remainder of the section, we fix $\Om= B_r$ and $A=B_R$ with $R>r$ and 
\[H=\dfrac{n-1}{R}.\]

Given the characterisation of the eigenfunctions of the modified Steklov problem in \autoref{steklov palla}, we obtain the following proposition.

\begin{prop}\label{oss dot u palla}
In the notation of \autoref{steklov palla}, let $\set{\phi_{k,l}}$ be the orthonormal system of eigenfunctions of the modified Steklov problem \eqref{eq:steklov_Qlambda_strong}. If $h\in \mathcal{T}(\partial B_R)$, then
\[h=\sum_{k\ge1}\sum_{l=1}^{d_k} h_{k,l} \phi_{k,l}|_{\partial B_R}\]
for some coefficients $h_{k,l}$, and,
\[\dot u[h] = k_0 U(R)\sum_{k\ge 1} \sum_{l=1}^{d_k}\dfrac{h_{k,l}}{\sigma_{k,l}} \phi_{k,l}.\]
\end{prop}
\begin{proof}
  $\dot u=\dot u[h]$ is the unique solution of 
\[
\begin{cases}
-\Delta \dot u -\lambda_\beta\dot u=\lambda'_\beta u & \text{in }B_r, \\[5 pt]
-a\Delta \dot u -\lambda_\beta\dot u=\lambda'_\beta u & \text{in }B_R\setminus B_r, \\[5 pt]
{\dot u}^-={\dot u}^+ & \text{on } \partial B_r, \\[6 pt]
\partial_{\nu_{B_r}}{\dot u}^- = a \partial_{\nu_{B_r}} {\dot u}^+& \text{on }\partial B_r,\\[5 pt]
a\partial_{\nu_{B_R}} \dot u + \beta \dot u = k_0 U(R)  h  & \text{on } \partial B_R, 
\end{cases}
\]
such that 
\[
    \int_{B_R}u \dot u\,dx=-\dfrac{U(R)^2}{2}\int_{\partial B_R}h\, d\Hn=0.
\]
In particular, if 
\[\int_{\partial B_R} h\,d\Hn=0,\]
$\la'_\beta=0$ and $\dot u$ is orthogonal to $u$ and the equation reduces to 
\[
\begin{cases}
-\Delta \dot u -\lambda_\beta\dot u=0 & \text{in }B_r, \\[5 pt]
-a\Delta \dot u -\lambda_\beta\dot u=0 & \text{in }B_R\setminus B_r, \\[5 pt]
{\dot u}^-={\dot u}^+ & \text{on } \partial B_r, \\[6 pt]
\partial_{\nu_{B_r}}{\dot u}^- = a \partial_{\nu_{B_r}} {\dot u}^+& \text{on }\partial B_r,\\[5 pt]
a\partial_{\nu_{B_R}} \dot u + \beta \dot u = k_0 U(R)  h  & \text{on } \partial B_R. 
\end{cases}
\]
Let $\set{\phi_{k,l}}$ be the complete orthonormal system of eigenfunctions of the modified Steklov problem \eqref{eq:steklov_Qlambda_strong} defined in \autoref{steklov palla}. Then,
 for every $k\ge1$ and any $l$, $\sigma_{k,l}>0$, and, by \autoref{steklov palla} 
\[\int_{\partial B_R} \phi_{k,l}\,d\Hn=\int_{B_R}u\phi_{k,l}\,dx=0.\]
Then, if $h=\phi_{k,l}|_{\partial B_R}$, necessarily,
\[\dot u = \dfrac{k_0 U(R)}{\sigma_{k,l}} \phi_{k,l}.\]
Hence, as the sequence of the traces of the eigenfunctions forms a complete orthonormal system in $L^2(\partial B_R)$, $h$ can be written as
\[h=\sum_{k\ge1}\sum_{l=1}^{d_k} h_{k,l} \phi_{k,l}|_{\partial B_R}\]
and, by linearity and continuity of the boundary value problem solved by $\dot u$,
\[\dot u = k_0 U(R)\sum_{k\ge 1}\sum_{l=1}^{d_k} \dfrac{h_{k,l}}{\sigma_{k,l}} \phi_{k,l}.\]
\end{proof}

We can finally prove a first characterisation of the coercivity of the Lagrangian.

\begin{prop}\label{prop: break e segno f}
  Consider the quantity 
    \[f(a,\beta,r,R):=\sigma_\beta(B_r,B_R)-\beta-a\left(\dfrac{\lambda_\beta(B_r,B_R)}{\beta}-H_{\partial B_R}\right).\]
    If $k_0>0$ and $f(a,\beta,r,R)>0$, there exists $c>0$ such that for all $h\in \mathcal{T}(\partial B_R)\setminus\set{0}$
    \[\ell_2[\lambda_\beta+\Lambda\mathcal{V}](B_R)(h,h)> c\norma{h}_{H^1(\partial B_R)}^2.\]
    If either $k_0<0$ or $f(a,\beta,r,R)<0$, then for any vector $\mathbf{v}\ne0$ setting
    \[h_1=\dfrac{\mathbf{v}\cdot x}{\abs{x}},\]
    we have
    \[\ell_2[\lambda_\beta+\Lambda \mathcal{V}](B_R)(h_1,h_1)<0.\]
\end{prop}
\begin{proof}
By \autoref{oss lambda'' palla} we have that, for every $h\in\mathcal{T}(\partial B_R)$ 
\[\begin{split}  \ell_2[{\lambda_\beta}+\Lambda\mathcal{V}](B_R)(h,h)=&-2Q_\lambda(\dot{u}[h],\dot{u}[h])+\dfrac{2\beta \Lambda}{a}\int_{\partial B_R} h^2\,d\Hn\\[10pt]&+\beta U^2(R)\ell_2[P-H \mathcal{V}](B_R)(h,h).\end{split}\]
 We recall that (see, for instance, \cite{DL19,HP18shape})  that there exists $c>0$ such that for all $h\in \mathcal{T}(\partial B_R)\setminus\set{0}$ such that 
 \[\int_{\partial B_R} hx_i\,d\Hn=0\]
 for every $i=1,\dots,n$
\begin{equation}\label{Pcoercive}\ell_2[P-H\mathcal{V}](B_R)(h,h)\geq c\norma{h}_{H^1(\partial B_R)}^2.\end{equation}
While, by direct computations, if 
\[h=\dfrac{\mathbf{v}\cdot x}{\abs{x}},\]
for some vector $\mathbf{v}\ne0$, we have
\[\ell_2[P-H\mathcal{V}](B_R)(h,h)=0.\]
Thus, we are left to analyse the quantity 
\[T[h]=-2Q_\lambda(\dot{u}[h],\dot{u}[h])+\dfrac{2\beta \Lambda}{a}\int_{\partial B_R} h^2\,d\Hn\]

By \autoref{oss dot u palla}, we have that every $h\in\mathcal{T}(\partial B_R)$ can be written as
\[h=\sum_{k\ge1}\sum_{l=1}^{d_k} h_{k,l} \phi_{k,l}|_{\partial B_R}\]
and
\[\dot{u}[h]=k_0U(R)\sum_{k\ge 1}\sum_{l=1}^{d_k} \dfrac{h_{k,l}}{\sigma_{k,l}}\phi_{k,l},\]
where $\set{\phi_{k,l}}$ is the complete orthonormal system of eigenfunctions of \eqref{eq:steklov_Qlambda_strong} defined in \autoref{steklov palla}.\medskip

Hence, as
\[Q_\lambda(\phi_{k,l},w)=\sigma_{k,l}\int_{\partial B_R}\phi_{k,l} w\,d\Hn,\]
for every $w\in H^1(A)$, we have
\[Q_\lambda(\dot{u}[h],\dot{u}[h])=k_0^2 U(R)^2\sum_{k\ge1}\sum_{l=1}^{d_k}\dfrac{h_{k,l}^2}{\sigma_{k,l}}.\]
Therefore,  recalling that $\Lambda=k_0U(R)^2$, we have
    \[\begin{split}T[h]=&-2k_0^2U(R)^2\sum_{k\ge1}\sum_{l=1}^{d_k}\dfrac{h_{k,l}^2}{\sigma_{k,l}}+\dfrac{2\beta k_0 U(R)^2}{a}\sum_{k\ge1}\sum_{l=1}^{d_k}h_{k,l}^2\\[10pt]&=2k_0^2U(R)^2\sum_{k\geq 1}\sum_{l=1}^{d_k} h_{k,l}^2\left(\frac{\beta}{a k_0}-\frac{1}{\sigma_{k,l}}\right)\ge2k_0^2U(R)^2\left(\dfrac{\beta}{ak_0}-\dfrac{1}{\sigma_{\beta}}\right)\norma{h}^2_{L^2(\partial B_R)},\end{split}\]
    with equality if $h$ is in the eigenspace associated to $\sigma_\beta$, that is (by \autoref{steklov palla}) 
    \[h=\dfrac{\mathbf{v}\cdot x}{\abs{x}},\]
    for some vector $\mathbf{v}\ne0$.\medskip
    
Hence, there are three possible cases:

\begin{itemize}

\item[Case 1:] $k_0>0$ and \[\dfrac{\beta}{ak_0}-\dfrac{1}{\sigma_{\beta}}>0\] 

that is, $k_0>0$ and
\[ \sigma_\beta> k_0\frac{a}{\beta}\]
which we rewrite as $k_0>0$ and 
\[\sigma_\beta-\beta> a \left(\frac{{\lambda_\beta}}{\beta}-H\right).\]
In this case, we have that  for all $h\in \mathcal{T}(\partial B_R)\setminus\set{0}$
\[T[h]> 0,\]
and, in particular, if $h$ is in the eigenspace associated to $\sigma_\beta$, then $h$ is also a spherical harmonic of eigenvalue $(n-1)R^{-2}$, so that, in said eigenspace 
\[T[h]=C\norma{h}^2_{L^2(\partial B_R)}=C' \norma{h}^2_{H^1(\partial B_R)}.\]
Hence, by \eqref{Pcoercive}, we have that there exists $c>0$ such that
\[\ell_2[{\lambda_\beta}+\Lambda \mathcal{V}](B_R)(h,h)\geq c\norma{h}^2_{H^{1}(\partial B_R)},\]
for every $h\in\mathcal{T}(\partial B_R)\setminus\set{0}$, and $B_R$ is a strictly stable shape under volume constraint.

\item[Case 2:] $k_0<0$ or 

\begin{equation}\label{condizione breacking} k_0>0\quad\text{and}\quad
    \sigma_\beta-\beta<a \left(\frac{{\lambda_\beta}}{\beta}-H\right).
\end{equation}

Then in this case, for every function $h$ in the eigenspace of $\sigma_\beta$, we have
\[T[h]<0\]
and
\[\ell_2[P-H\mathcal{V}](B_R)(h,h)=0.\]
Hence, 
\begin{equation}\label{remketa}\ell_2[{\lambda_\beta}+\Lambda\mathcal{V}](B_R)(h,h)=2k_0^2U(R)^2\left(\dfrac{\beta}{ak_0}-\dfrac{1}{\sigma_{\beta}}\right)\norma{h}^2_{L^2(\partial B_R)}<0.\end{equation}

\item[Case 3:] $k_0=0$ or 
\[\sigma_\beta-\beta=a \left(\frac{{\lambda_\beta}}{\beta}-H\right).\]
In this case 
\[\ell_2[\lambda_\beta+\Lambda\mathcal{V}](B_R)(h,h)\ge0,\]
for every $h\in \mathcal{T}(\partial B_R)$, but
\[\ell_2[\lambda_\beta+\Lambda\mathcal{V}](B_R)(h_1,h_1)=0,\]
for all functions $h_1$ of the type
\[h=\dfrac{\mathbf{v}\cdot x}{\abs{x}}.\]    
\end{itemize}
\end{proof}

\begin{oss}
     Let $\Omega,A\subset \mathbb{R}^n$ be bounded, open sets with $C^{2,\gamma}$ boundary, such that $A$ is connected and $\Om\ssubset A$. Assume in addition that $A$ is a critical shape under volume constraint for the functional $\lambda_\beta(\Om,\cdot)$. Then we can express $\dot{u}[h]$, for zero mean deformation, using the eigenfunctions, $\set{\phi_k}$, of the modified Steklov eigenvalue problem \eqref{eq:steklov_Qlambda_strong}. Indeed, recall that $\dot{u}=\dot{u}[h]$ is the unique solution of 
    \[
    \begin{cases}
        -\Delta \dot u -\lambda_\beta \dot u=\lambda'_\beta u & \text{in }\Om, \\[5 pt]
-a\Delta \dot u -\lambda_\beta \dot u=\lambda'_\beta u & \text{in }A\setminus \bar\Om, \\[5 pt]
{\dot u}^-={\dot u}^+ & \text{on } \partial \Om, \\[5 pt]
\partial_{\nu_\Om}{\dot u}^- = a \partial_{\nu_\Om} {\dot u}^+& \text{on }\partial\Omega,\\[5 pt]
a\partial_{\nu} \dot u + \beta \dot u = k_0 u  h+a\divv^\tau(h \nabla^\tau u)  & \text{on } \partial A,\\[10pt]
 \dint_{A}u \dot u\,dx=-\dfrac{1}{2}\dint_{\partial A}u^2h\, d\Hn
\end{cases}\]
 where,
    \[k_0 = \la_\beta-\beta H_{\partial A}+\dfrac{\beta^2}{a}.\] 
Let 
\[g=g[h]= k_0 u  h+a\divv^\tau(h \nabla^\tau u).\]
Using integration by parts on $\partial A$ we know that
\[\int_{\partial A}g u\,d\Hn=\int_{\partial A}\left[a\abs{\nabla^\tau u}^2-k_0u^2\right]h\,d\Hn=\lambda_\beta',\]
so that, if $h$ has zero mean, as $A$ is a critical shape under volume constraint, we have that $\lambda_\beta'=0$ and 
\[\int_{\partial A}gu\,d\Hn=0.\]
Thus recalling that the Robin eigenfunction $u$ is proportional to $\phi_0$, we have that the $L^2(\partial A)$ function $g$ can be written as
\[g=\sum_{k\ge1}g_k[h]\phi_k|_{\partial A},\]
where 
\[g_k=g_k[h]=\int_{\partial A}\left[k_0uh+a\divv^\tau(h\nabla^\tau u)\right]\phi_k\,d\Hn.\]
Recalling that the eigenvalues $\sigma_k$ are strictly positive for $k\ge1$ we can define 
\[v[h]=\sum_{k\ge1}\dfrac{g_k}{\sigma_k}\phi_k.\]
Then $v[h]\in H^1(A)$ and it is a solution to 
\[
    \begin{cases}
        -\Delta \dot u -\lambda_\beta \dot u=\lambda'_\beta u & \text{in }\Om, \\[5 pt]
-a\Delta \dot u -\lambda_\beta \dot u=\lambda'_\beta u & \text{in }A\setminus \bar\Om, \\[5 pt]
{\dot u}^-={\dot u}^+ & \text{on } \partial \Om, \\[5 pt]
\partial_{\nu_\Om}{\dot u}^- = a \partial_{\nu_\Om} {\dot u}^+& \text{on }\partial\Omega,\\[5 pt]
a\partial_{\nu} \dot u + \beta \dot u = k_0 u  h+a\divv^\tau(h \nabla^\tau u)  & \text{on } \partial A.
\end{cases}\]
Hence, letting
\[c[h]=-\dfrac{1}{2}\int_{\partial A}u^2 h\,d\Hn-\int_A v[h]u\,dx,\]
we have that 
\[\dot{u}[h]=v[h]+c[h]u.\]
Moreover
\[Q_\lambda(\dot{u}[h],\dot{u}[h])=\sum_{k\ge1}\dfrac{g_k^2}{\sigma_k}\le \dfrac{\norma{g}_{L^2(\partial A)}^2}{\sigma_\beta} ,\]
with equality if and only if $g$ lies in the eigenspace associated to the principal eigenvalue $\sigma_\beta$.
\end{oss}

\subsubsection*{On the symmetry-breaking condition}

We aim to better understand the symmetry-breaking condition 
\[f(\beta,a,r, R)=\sigma_\beta-\beta-a \left(\frac{{\lambda_\beta}}{\beta}-\dfrac{n-1}{R}\right).\]\medskip
 
Recall that the radial profile, $U$, of the eigenfunction is a bounded solution of the following equation:
    \begin{equation}    \label{Eq U}
        \begin{cases}
        -U''(s)-\dfrac{n-1}{s}U'(s)=\lambda_\beta U(s) & \text{in}\quad (0,r)\\[5pt]
           -aU''(s)-a\dfrac{n-1}{s}U'(s)=\lambda_\beta U(s) & \text{in}\quad (r,R)\\[5pt]
        U(r^-)=U(r^+)\\[5pt]
        U'(r^-)=a U'(r^+)\\[5pt]
        a U'(R)+\beta U(R)=0\\[5pt]
        U>0
        \end{cases}
    \end{equation}
while, by \autoref{steklov palla}, the eigenspace associated to $\sigma_\beta$ is generated by the functions
\[\phi_{1,l}(x):= V_1(\abs{x})\dfrac{x_l}{\abs{x}},\]
where $V=V_1$ is the bounded solution to the equation

    \begin{equation}    \label{Eq V}
        \begin{cases}
       - V''(s)-\dfrac{n-1}{s}V'(s)+\dfrac{n-1}{s^2}V(s)=\lambda_\beta V(s)& \text{in}\quad (0,r)\\[10pt]
         -aV''(s)-a\dfrac{n-1}{s}V'(s)+a\dfrac{n-1}{s^2}V(s)=\lambda_\beta V(s) & \text{in}\quad (r,R)\\[5pt]
        V(r^-)=V(r^+)\\[5pt]
        V'(r^-)=a V'(r^+)\\[5pt]
      V(R)=1.
        \end{cases}
    \end{equation}
We also recall that 
\begin{equation}\label{sigma come V'}
    \sigma_\beta-\beta=a{V'(R)}.
\end{equation}

We start with the following lemma, which allows us to rewrite the symmetry-breaking condition only in terms of $a$ and $U''(r-)$. 

\begin{lemma}\label{lem: sign f}
We have that
    \[f(a,\beta,r,R)=\left(\dfrac{r}{R}\right)^{n-1}\dfrac{(a-1)}{\beta U(R)}V(r)U''(r^-).\]
    In particular, the symmetry-breaking condition
    \[f(a,\beta,r,R)<0\]
    is equivalent to 
    \[(1-a)U''(r^-)>0.\]
\end{lemma}

\begin{proof}
We consider $U'$ the derivative of the function $U$. Since $U$ solves \eqref{Eq U}, $U'$ is a bounded solution of the following equation 
    \begin{equation}    \label{Eq U'}
        \begin{cases}
        -(U')''(s)-\dfrac{n-1}{s}(U')'(s)+\dfrac{n-1}{s^2}U'(s)=\lambda_\beta U'(s) & \text{in}\quad (0,r)\\[10pt]
         -a(U')''(s)-a\dfrac{n-1}{s}(U')'(s)+a\dfrac{n-1}{s^2}U'(s)=\lambda_\beta U'(s) & \text{in}\quad (r,R)\\[5pt]
        U'(r^-)=a U'(r^+)\\[5pt]
        (U')'(r^-)=a (U')'(r^+)\\[10pt]
       (U')'(R)+\left(H-\dfrac{\lambda_\beta}{\beta}\right)U'(R)=0
        \end{cases}
    \end{equation}

For all $s\in(r,R)$, we consider the weighted Wronskian of $V$ and $U'$, defined as 
\[W(s):=s^{n-1}\left(U''(s)V(s)-U'(s) V'(s)\right).\]
By direct computation, using \eqref{Eq U'} and \eqref{Eq V}, we obtain 
\[W'(s)= 0\quad \text{for all}\quad s\in(r,R).\]

Thus, $W$ is constant and, in particular 
\begin{equation}\label{Wronskiano R=r}
    W(R)=W(r^+).
\end{equation}
Using the boundary conditions in \eqref{Eq U'} and \eqref{Eq V}, and by \eqref{sigma come V'}, we have that 

\begin{equation}\label{W(R)}
    W(R)=R^{n-1}U'(R)\left(\dfrac{\lambda_\beta}{\beta}-H-\dfrac{\sigma_\beta-\beta}{a}\right)=\dfrac{\beta U(R)R^{n-1}}{a^2}f(a,\beta,r,R).
\end{equation}

We now want to rewrite $W(r^+)$.
We start by noticing that $U'$ and $V$ are bounded solutions to the equation
\begin{equation}\label{u'andveq}- v''(s)-\dfrac{n-1}{s}v'(s)+\dfrac{n-1}{s^2}v(s)=\lambda_\beta v(s)\end{equation}
in $(0,r)$. As discussed in the proof of \autoref{steklov palla}, the space of bounded solutions in a neighbourhood of zero of equation \eqref{u'andveq} is one-dimensional; thus, there exists $c\in \R$ such that 
\[V(s)=cU'(s)\,\quad \text{ for all }\quad s\in(0,r).\]

Using the transmission conditions in \eqref{Eq U'} and \eqref{Eq V}, we have that 
\begin{equation}\label{Wr+}W(r^+)=r^{n-1}\left(\dfrac{c}{a}U''(r^-)U'(r^-)-\dfrac{c}{a^2}U''(r^-)U'(r^-) \right)=\dfrac{r^{n-1}}{a^2}(a-1)V(r)U''(r^-).\end{equation}

Hence, by \eqref{Wronskiano R=r}, \eqref{W(R)}, and \eqref{Wr+} we have that 
\[f(a,\beta,r,R)=\left(\dfrac{r}{R}\right)^{n-1}\dfrac{(a-1)}{\beta U(R)}V(r)U''(r^-).\]

Since by definition $U>0$, and by \autoref{steklov palla}, $V>0$, we have that the symmetry-breaking condition is equivalent to
\[(1-a)U''(r^-)>0.\]
\end{proof}

We conclude the characterisation of the symmetry-breaking condition with the following lemma about the sign of $U''(r^-)$. 

\begin{lemma}\label{lem: sign U''}
    We have that
    \[U''(r^-)>0\]
    if and only if 
    $$\lambda_\beta>\lambda^N(B_r),$$
    the principal Neumann eigenvalue of the Laplacian in $B_r$.
\end{lemma}
\begin{proof}
Recall that in $(0,r)$ 
    \[U(s)=cs^{-\frac{n-2}{2}}J_{\frac{n-2}{2}}\left(\sqrt{\lambda_\beta} s\right),\]
    for some positive constant $c$.
   By the recurrence formula (see for instance \cite[Equation 9.1.30]{ASHandbook}) 
    \[\dfrac{d}{dz}\left(s^{-m}J_m(s)\right)=-s^{-m}J_{m+1}(s),\]
    we deduce that 
    \[U'(s)=-c\sqrt{\lambda_\beta}s^{-\frac{n-2}{2}}J_{\frac{n}{2}}\left(\sqrt{\lambda_\beta}\,s\right).\]
    Hence,
    \[U''(r^-)=-c\sqrt{\lambda_\beta}\,r^{-\frac{n-2}{2}}\left(\sqrt{\lambda_\beta}\,J'_\frac{n}{2}\left(\sqrt{\lambda_\beta}\,r\right)-\dfrac{n-2}{2r}J_\frac{n}{2}\left(\sqrt{\lambda_\beta}\,r\right)\right).\]
Then $U''(r^-)>0$
    if and only if 
    \[\sqrt{\lambda_\beta}\,sJ'_\frac{n}{2}\left(\sqrt{\lambda_\beta}\,s\right)-\dfrac{n-2}{2}J_\frac{n}{2}\left(\sqrt{\lambda_\beta}\,s\right)<0.\]\medskip
    
    Let us consider the function 
    \[\psi(s)= sJ'_{\frac{n}{2}}(s)-\dfrac{n-2}{2}J_{\frac{n}{2}}(s),\]
    and let $j_{m,1}$ denote the first positive zero of $J_m$.  By the properties of Bessel functions, we have that 
    \[\psi(j_{n/2,1})=j_{n/2,1}J'_{\frac{n}{2}}(j_{n/2,1})<0,\]
    while, using the power series representation (see for instance \cite[formula 9.1.10]{ASHandbook}), we have
    \[\lim_{s\to0^+}s^{-\frac{n}{2}}\psi(s)>0,\]
    so that $\psi$ admits at least one zero in the interval $(0,j_{{n/2},1})$.
    By Dixon’s theorem (see \cite[p.480]{TBF} and \cite[Equation 2.1]{BPS18}) the zeroes of $\psi$ and $J_{\frac{n}{2}}$ are interlaced, hence, $\psi$ admits a unique positive zero in the interval $(0,j_{n/2,1})$. Moreover, as it is well known, the first positive zero of $\psi$ coincides with \[\sqrt{\lambda^N(B_1)}=\dfrac{\sqrt{\lambda^N(B_r)}}{r}.\]
    In particular, then, in the interval $(0,j_{n/2,1})$ we have that $\psi(s)<0$ if and only if $s\in\left(\sqrt{\lambda^N(B_1)},j_{n/2,1}\right)$.\medskip
    
    Recall that the square of first positive zero of $J_{\frac{n-2}{2}}$ coincides with the first Dirichlet eigenvalue of the unit ball
    \[\lambda^D(B_1)=\dfrac{\sqrt{\lambda^D(B_r)}}{r},\]
    and that, by the interlacing properties of the zeroes of the Bessel functions
    (see for instance \cite[formula 9.5.2]{ASHandbook}),  
    $j_{\frac{n-2}{2},1}<j_{n/2,1}$. Hence, as
    \[\lambda_\beta<\lambda^D(B_r),\]
    we finally have that $U''(r^-)>0$ if and only if 
    \[\lambda_\beta>\lambda^N(B_r).\]
\end{proof}

For clarity, we summarise \autoref{lem: sign f} and \autoref{lem: sign U''} in the following proposition.

\begin{prop}\label{prop: sign f}
    We have that $f(a,\beta,r,R)>0$ if and only if 
    \[(1-a)\lambda_\beta(B_r,B_R)<(1-a)\lambda^N(B_r).\]
\end{prop}

Finally, we have the following. 
\begin{proof}[Proof of \autoref{lem:corcivity condition}]
The conclusion follows immediately by putting together
\autoref{prop: break e segno f}, \autoref{prop: sign f} . 
\end{proof}

\subsection{Improved continuity of the second derivative}\label{improved}
In this section, we prove \autoref{lem:Ic-condition}.
We point out that, to the best of our knowledge, this improved continuity
property does not seem to have been explicitly proved in the literature even
in the one-phase case (i.e. $a=1$).\medskip

 Let $\Om\subset\R^n$ be a bounded open set with $C^{2,\gamma}$ boundary, let $A\subset\R^n$ be an open connected bounded set with $C^3$ boundary such that $\Om\ssubset A$. Finally, fix $h\in C^{2,\gamma}(\partial A)$ and let $\Phi_t=\Phi_{th}$ be as in \autoref{SDcomputed}. Up to assuming $\norma{h}_{C^{2,\gamma}(\partial A)}<\eta$ sufficiently small we can always assume $t\in[0,1]$.\medskip 

We denote by $J_t$ and $J_t^\tau$ the volume and tangential Jacobians,
respectively. If $f$ is a function on $A_t$ or on $\partial A_t$, we write
\[
\widehat f:=f\circ \Phi_t .
\]
In particular,
\[
\widehat{\nabla \dot{ u_t}}:=(\nabla \dot u_t)\circ \Phi_t,
\qquad
\widehat{\nabla^{\tau_t}u_t}:=(\nabla^{\tau_t}u_t)\circ \Phi_t .
\]\medskip

By our assumptions, we have that for every $t\in[0,1]$ 
\[\norma{x-\Phi_t}_{C^{2,\gamma}(\R^n,\R^n)}\le C \norma{h}_{C^{2,\gamma}(\partial A)},\]
for some positive constant $C$. In particular, this entails that, up to changing the constant $C$,
\begin{align}
\|J_t-1\|_{C^{1,\gamma}(A)}
&\leq C\norma{h}_{C^{2,\gamma}(\partial A)},&
\|J_t^\tau-1\|_{C^{1,\gamma}(\partial A)}
&\leq C\norma{h}_{C^{2,\gamma}(\partial A)}, \label{eq:jac-est-forte}\\
\|\widehat{\nu_t}-\nu\|_{C^{1,\gamma}(\partial A)}
&\leq C\norma{h}_{C^{2,\gamma}(\partial A)},&
\|\widehat H_t-H\|_{C^{0,\gamma}(\partial A)}
&\leq C\norma{h}_{C^{2,\gamma}(\partial A)}, \label{eq:curv-est-forte}\\
\|\widehat\alpha_t-1\|_{C^{1,\gamma}(\partial A)}
&\leq C\norma{h}_{C^{2,\gamma}(\partial A)},&
\|\widehat\gamma_t\|_{C^{1,\gamma}(\partial A)}
&\leq C\norma{h}_{C^{2,\gamma}(\partial A)}, \label{eq:alpha,gamma-est-forte}
\end{align}
where we recall that $\alpha_t=\nu\cdot\nu_t$ and $\gamma_t=\nu-\alpha_t\nu_t$. Finally, recall that, by \autoref{differentiability lemma}, 
\begin{equation}\label{eq:u-small}
\|\widehat u_t-u\|_{C^{2,\gamma}(\bar \Om)}+
\|\widehat u_t-u\|_{C^{2,\gamma}(\overline A\setminus\bar\Om)}
+
\|\widehat k_t-k\|_{L^\infty(\partial A)}
+
|\lambda_\beta(t)-\lambda_\beta(0)|
\leq C \norma{h}_{C^{2,\gamma}(\partial A)}. \end{equation}

In the following, we will denote by $C$ any constant depending solely on the data of the problem, that is $\Om, A, \beta,$ and $a$.\medskip

We now prove the following lemma
\begin{lemma}\label{lemmIC}
The following estimates hold
\begin{equation}\label{eq:uniform-udot}
\|\dot u\|_{H^1(A)}
+
|\lambda'_\beta(0)|
\leq
C\|h\|_{H^1(\partial A)} .
\end{equation}
Moreover, there exists $\eta>0$ such that, if $\norma{h}_{C^{2,\gamma}(\partial A)}\le \eta$, then
\begin{equation}
\|\hdu-\dot u\|_{H^1(A)}
+
|\lambda'_\beta(t)-\lambda'_\beta(0)|
\leq
C\|h\|_{C^{2,\gamma}(\partial A)}
\|h\|_{H^1(\partial A)}. \label{eq:udot-small}
\end{equation}

\end{lemma}
\begin{proof}
    Recall that, by \autoref{lemma lambda'} 

    \[\la'_\beta(t)=\int_{\partial A_t} \left(a\abs{\nabla^{\tau_t} u_t}^2-k_tu^2_t\right)h_t\,d\Hn, \]

and that $\dot{u}_t$ is the unique weak solution to
    \[
    \begin{cases}
        -\Delta \dot u_t -\lambda_\beta(t) \dot u_t=\lambda'_\beta(t)u_t & \text{in }\Om, \\[5 pt]
-a\Delta \dot u_t -\lambda_\beta(t) \dot u_t=\lambda'_\beta(t)u_t & \text{in }A_t\setminus \bar\Om, \\[5 pt]
{\dot u_t}^-={\dot u_t}^+ & \text{on } \partial \Om, \\[5 pt]
\partial_{\nu_\Om}{\dot u_t}^- = a \partial_{\nu_\Om} {\dot u_t}^+& \text{on }\partial\Omega,\\[5 pt]
a\partial_{\nu_t} \dot u_t + \beta \dot u_t = k_t u_t  h_t+a\divv^\tau(h_t \nabla^\tau u_t)  & \text{on } \partial A_t,\\[10pt]
 \dint_{A_t}u_t \dot u_t\,dx=-\dfrac{1}{2}\dint_{\partial A_t}u_t^2h_t\, d\Hn
\end{cases}\]
 where,
    \[k_t = \la_\beta(t)-\beta H_t+\dfrac{\beta^2}{a}.\] 
Let 
\[g_t= k_t u_t  h_t+a\divv^\tau(h_t \nabla^\tau u_t).\]

We start by proving \eqref{eq:uniform-udot}. By the regularity of $u=u_0$, we immediately have that
\[\begin{split}\abs{\la'_\beta(0)}&\le\left(\norma{\abs{\nabla u}}_{L^\infty(\partial A)}^2+\norma{k_0}_{L^\infty(\partial A)} \norma{u}_{L^\infty(\partial A)}^2\right)\dint_{\partial A} \abs{h}\,d\Hn\\[10pt]&\le C\norma{h}_{L^2(\partial A)}.\end{split}\]
Similarly, by \autoref{usefullemma},
\[\norma{\dot{u}}_{H^1(A)}\le C\left(\norma{g_0}_{L^2(\partial A)}+\norma*{u^2h}_{L^1(\partial A)}\right)\le C \norma{h}_{H^1(\partial A)}.\]
Hence we have \eqref{eq:uniform-udot}.\medskip

To prove \eqref{eq:udot-small} we separate the estimate of $\la'_\beta(t)-\la'_\beta$ from the one of $\hdu-\dot{u}$. We start with the estimate 
\begin{equation}\label{delta la' est}\abs*{\la'_\beta(t)-\la'_\beta(0)}\le C\norma{h}_{C^{2,\gamma}(\partial A)}\norma{h}_{H^1(\partial A)}.\end{equation}
by definition 
\[\la'_\beta(t)=-\dint_{\partial A_t} u_t g_t\,d\Hn,\] 
so that 
\[\begin{split}\abs*{\la'_\beta(t)-\la'_\beta}&\le \norma*{\widehat{u_t}J^\tau_t \hat{g}_t-ug_0}_{L^1(\partial A)}\\[5pt]
&\le \norma{g_0}_{L^2(\partial A)}\norma*{\widehat{u_t}-u}_{L^2(\partial A)}+\norma*{\widehat{u_t}}_{L^2(\partial A)}\norma*{J^\tau_t\hat{g}_t-g_0}_{L^2(\partial A)}\\[5pt]
&\le C\norma*{J^\tau_t\hat{g}_t-g_0}_{L^2(\partial A)}+C\norma{h}_{C^{2,\gamma}(\partial A)}\norma{h}_{H^1(\partial A)}, \end{split}\]
where in the last estimate we used the estimate \eqref{eq:u-small} on $\widehat{u_t}-u$. Then, to prove \eqref{delta la' est} we need to prove 
\begin{equation}\label{l-estimate}\norma*{J^\tau_t\hat{g}_t-g_0}_{L^2(\partial A)}\le C\norma{h}_{C^{2,\gamma}(\partial A)}\norma{h}_{H^1(\partial A)}.\end{equation}\medskip

Let $P_t$ denote the projection on the tangent space to $\partial A_t$, that is
\[P_t=I-\nu_t\otimes\nu_t,\]
then 
\[\widehat{\nabla^{\tau_t} f}=\widehat{P_t} [D\Phi_t]^{-T}\nabla \hat{f}=\widehat{P_t} [D\Phi_t]^{-T}\nabla^\tau \hat{f},\]
so that, letting \[M^\tau_t=J^\tau_t [D\Phi_t]^{-1}\widehat{P_t}[D\Phi_t]^{-T},\]
we have
\[J^\tau_t \widehat{\divv^{\tau_t}(h_t \nabla^{\tau_t} u_t)}=\divv^\tau(h\widehat{\alpha_t} M^\tau_t \nabla^\tau \widehat{u_t}),\]
and
\[J^\tau_t\hat{g}_t=\widehat{k_t}\widehat{u_t}\widehat{\alpha_t}J^\tau_t h + a J^\tau_t \widehat{\divv^{\tau_t}(h_t \nabla^{\tau_t} u_t)}=\widehat{k_t}\widehat{u_t}\widehat{\alpha_t}J^\tau_t h + a\divv^\tau(h\widehat{\alpha_t} M^\tau_t \nabla^\tau \widehat{u_t}).\]
Finally, then, 
\[\norma*{J^\tau_t\hat{g}_t-g_0}_{L^2(\partial A)}\le C\left( \norma{\widehat{k_t}\widehat{u_t}\widehat{\alpha_t}J^\tau_t-k_0u}_{L^\infty(\partial A)}\norma{h}_{L^2(\partial A)}+\norma*{\widehat{\alpha_t} M^\tau_t \nabla^\tau \widehat{u_t}-\nabla^\tau u}_{C^{1}(\partial A)}\norma{h}_{H^1(\partial A)}\right).\]
Using \eqref{eq:u-small}, \eqref{eq:jac-est-forte}, \eqref{eq:curv-est-forte}, and \eqref{eq:alpha,gamma-est-forte}  we can estimate
\[\norma{\widehat{k_t}\widehat{u_t}\widehat{\alpha_t}J^\tau_t-k_0u}_{L^\infty(\partial A)}+\norma*{\widehat{\alpha_t} M^\tau_t \nabla^\tau \widehat{u_t}-\nabla^\tau u}_{C^{1}(\partial A)}\le C\norma{h}_{C^{2,\gamma}(\partial A)},\] so that we have \eqref{l-estimate}, hence \eqref{delta la' est}.\medskip

To prove the estimate on the function $w=\hdu-\dot{u}$, we pull-back $\dot{u}_t$ on $A$ and use the energy estimate in \autoref{usefullemma}. As in \autoref{differentiability lemma}, we have that $\widehat{\dot{u_t}}$ is a weak solution to 
\begin{equation}\label{eq hdu 1}
    \begin{cases}
    -\Delta \hdu-\la_\beta(t)\hdu = \la_\beta(t)'\widehat{u_t} &\text{in }\Om,\\[5pt]
        -a\divv(M_t \nabla \hdu)-\lambda_\beta(t)J_t\hdu=\lambda'_\beta(t) \widehat{u_t} &\text{in }A\setminus\bar\Om,\\[5pt]
        \hdu^-=\hdu^+ &\text{on }\partial\Om,\\[5pt]
        \partial_{\nu_\Om}\hdu^-=a\partial_{\nu_\Om}\hdu^+ &\text{on }\partial\Om,\\[5pt]

        a (M_t \nabla \hdu)\cdot\nu_A +\beta J^\tau_t \hdu = J^\tau_t\widehat{g_t} &\text{on }\partial A,\\[10pt]
        \dint_A J_t \widehat{u_t}\hdu \,dx = -\dfrac{1}{2}\dint_{\partial A}  \widehat{u_t}^2 \widehat{\alpha_t} J^\tau_t h\,d\Hn, 
    \end{cases}
\end{equation}
where,
\[M_t = J_t [D \Phi_t]^{-1}[D\Phi_t]^{-T},\]
which, by definition, coincides with the identity in a neighbourhood of $\bar{\Om}$. Equivalently,  we can rewrite \eqref{eq hdu 1} so that the operator on the left-hand side is $-(\rho\Delta + \la_\beta)$, that is
\[
    \begin{cases}
    -\Delta \hdu-\la_\beta\hdu  = (\la_\beta(t)-\la_\beta)\hdu + \la_\beta(t)'\widehat{u_t}&\text{in }\Om,\\[5pt]
       -a\Delta \hdu-\la_\beta \hdu = -a\divv((I-M_t) \nabla \hdu) + (J_t\la_\beta(t)-\la_\beta)\hdu+\lambda'_\beta(t) \widehat{u_t} &\text{in }A\setminus\bar\Om,\\[5pt]
        \hdu^-=\hdu^+ &\text{on }\partial\Om,\\[5pt]
        \partial_{\nu_\Om}\hdu^-=a\partial_{\nu_\Om}\hdu^+ &\text{on }\partial\Om,\\[5pt]

        a\partial_{\nu_A} \hdu +\beta \hdu = a\left[(I-M_t)\nabla \hdu\right]\cdot\nu_A +\beta (1-J^\tau_t)\hdu + J^\tau_t\widehat{g_t} &\text{on }\partial A,\\[10pt]
        \dint_A u \hdu \,dx = \dint_A (u-J_t \widehat{u_t})\hdu \,dx  -\dfrac{1}{2}\dint_{\partial A}  \widehat{u_t}^2 \widehat{\alpha_t} J^\tau_t h\,d\Hn. 
    \end{cases}
\]
Finally, by linearity, we  have 
\[\begin{cases}
    -\Delta w-\la_\beta w  = (\la_\beta(t)-\la_\beta)\hdu + \la_\beta(t)'\widehat{u_t}-\la'_\beta u &\text{in }\Om,\\[5pt]
       -a\Delta w-\la_\beta w = -a\divv((I-M_t) \nabla \hdu) + (J_t\la_\beta(t)-\la_\beta)\hdu+\lambda'_\beta(t) \widehat{u_t} - \la'_\beta u &\text{in }A\setminus\bar\Om,\\[5pt]
        w^-=w^+ &\text{on }\partial\Om,\\[5pt]
        \partial_{\nu_\Om}w^-=a\partial_{\nu_\Om}w^+ &\text{on }\partial\Om,\\[5pt]

        a\partial_{\nu_A} w +\beta w = a\nu_A (I-M_t)\nabla \hdu +\beta (1-J^\tau_t)\hdu + J^\tau_t\widehat{g_t} -g_0 &\text{on }\partial A,\\[10pt]
        \dint_A u w \,dx = \dint_A (u-J_t \widehat{u_t})\hdu \,dx  -\dfrac{1}{2}\dint_{\partial A}  \left(\widehat{u_t}^2 \widehat{\alpha_t} J^\tau_t - u^2\right) h\,d\Hn.
    \end{cases}\]
Using the energy estimate in \autoref{usefullemma}, we have
\begin{equation}\label{incubo}
    \begin{split}
        \norma{w}_{H^1(A)}\le& C \left(\norma*{(I-M_t)\nabla \hdu}_{L^2(A)}+\norma*{(J_t\la_\beta(t)-\la_\beta)\hdu}_{L^2(A)}+\norma*{(1-J^\tau_t)\hdu}_{L^2(\partial A)}+\norma*{(u-J_t\widehat{u_t})\hdu}_{L^1(A)}\right)\\[5pt]
        &+ C\left(\norma*{\la_\beta(t)' \widehat{u_t}-\la'_\beta u}_{L^2(A)} +\norma*{J^\tau_t\hat{g}_t-g_0}_{L^2(\partial A)}+\norma*{(\widehat{u_t}^2\widehat{\alpha_t}J^\tau_t-u^2)h}_{L^1(\partial A)}\right)\\[10pt]
        \le&C\left(\norma*{I-M_t}_{L^\infty(A)}+\norma*{J_t\la_\beta(t)-\la_\beta}_{L^\infty(A)}+\norma*{1-J^\tau_t}_{L^\infty(\partial A)}+\norma*{u-J_t\widehat{u_t}}_{L^2(A)}\right)\norma{\hdu}_{H^1(A)}\\[10pt]
        &+C\left(\abs*{\la'_\beta(t)-\la'_\beta}\norma*{\widehat{u_t}}_{L^2(A)}+\abs*{\la_\beta'}\norma*{\widehat{u_t}-u}_{L^2(A)}+\norma*{J^\tau_t\hat{g}_t-g_0}_{L^2(\partial A)}\right)\\[10pt]
    &+C\left(\norma{h}_{L^2(\partial A)}\norma*{\widehat{u_t}^2-u^2}_{L^2(\partial A)}+\norma{h}_{L^2(\partial A)}\norma*{\widehat{\alpha_t}J^\tau_t -1}_{L^\infty(\partial A)}\right),
    \end{split}
\end{equation}
where in the last step we used that $J^\tau_t,\widehat{\alpha_t}$ and $u$ are bounded by a constant on $\partial A$.\medskip

By estimates \eqref{eq:jac-est-forte}, \eqref{eq:alpha,gamma-est-forte} and \eqref{eq:u-small} we obtain 
\begin{equation}\label{intermidiate estimate}\norma*{\widehat{u_t}-u}_{L^2(A)}+\norma*{\widehat{u_t}^2-u^2}_{L^2(\partial A)}+\norma*{\widehat{\alpha_t } J^\tau_t-1}_{L^\infty(\partial A)}\le C\norma{h}_{C^{2,\gamma}(\partial A)}.\end{equation}

Similarly, we can estimate
\[\norma*{I-M_t}_{L^\infty(A)}+\norma*{J_t\la_\beta(t)-\la_\beta}_{L^\infty(A)}+\norma*{1-J^\tau_t}_{L^\infty(\partial A)}+\norma*{u-J_t\widehat{u_t}}_{L^2(A)}\le C\norma{h}_{C^{2,\gamma}(\partial A)},\]
so that letting
\[\mathcal{R}=\left(\norma*{I-M_t}_{L^\infty(A)}+\norma*{J_t\la_\beta(t)-\la_\beta}_{L^\infty(A)}+\norma*{1-J^\tau_t}_{L^\infty(\partial A)}+\norma*{u-J_t\widehat{u_t}}_{L^2(A)}\right)\norma{\hdu}_{H^1(A)}\]
and, using \eqref{eq:uniform-udot}, we have
\begin{equation}\label{R-estimate}\mathcal{R}\le C\norma{h}_{C^{2,\gamma}(\partial A)}(\norma{w}_{H^1(A)}+\norma{\dot{u}}_{H^1(A)})\le   C\norma{h}_{C^{2,\gamma}(\partial A)}(\norma{w}_{H^1(A)}+\norma{h}_{H^1(\partial A)}).\end{equation}

Hence, putting together estimates \eqref{delta la' est}, \eqref{l-estimate}, \eqref{eq:uniform-udot}, \eqref{intermidiate estimate}  and \eqref{R-estimate}, we have that  \eqref{incubo} reduces to
\begin{equation}\label{ttl-estimate}\norma{w}_{H^1(A)}\le C\norma{h}_{C^{2,\gamma}(\partial A)}\norma{w}_{H^1(A)}+C\norma{h}_{C^{2,\gamma}(\partial A)}\norma{h}_{H^1(\partial A)}. \end{equation}
Choosing $\eta>0$ such that $C\eta<1/2$, we can reabsorb the first term of the right-hand side of \eqref{ttl-estimate} in the left-hand side, obtaining
\[\norma{w}_{H^1(A)}\le C\norma{h}_{C^{2,\gamma}(\partial A)}\norma{h}_{H^1(\partial A)},\]
hence \eqref{eq:udot-small}.
\end{proof}

\begin{oss}
Combining together estimates \eqref{eq:uniform-udot} and \eqref{eq:udot-small} we have
\begin{equation}\label{eq:uniform-udot_t}
\|\hdu\|_{H^1(A)}
+
|\lambda'_\beta(t)|
\leq
C\|h\|_{H^1(\partial A)} .
\end{equation}
\end{oss}

We can now finally prove \autoref{lem:Ic-condition}

\begin{proof}[Proof of \autoref{lem:Ic-condition}]
Recall that, by \autoref{prop lambda''} we have 
\[\begin{split}\displaystyle
        \lambda''_\beta(t)=&-2\left[\int_{A_t}\rho \abs{\nabla \dot{u_t}}^2\,dx+\beta\int_{\partial A_t} \dot{u_t}^2\,d\Hn-\lambda_\beta(t)\int_{A_t} \dot{u_t}^2\,dx-2\lambda'_\beta(t)\int_{ A_t} u_t\dot{u_t}\,dx\right]\\[10pt]
        &+\int_{\partial A_t} \left[a\abs{\nabla^{\tau_t} u_t}^2-k_tu_t^2\right]\left[\mathbf{B}_t(\gamma_t, \gamma_t)+2\gamma_t\cdot\nabla^{\tau_t}\alpha_t + \left(H_t-\dfrac{2\beta}{a}\right)\alpha_t^2\right]h^2\,d\Hn\\[10pt]&-2a\int_{\partial A_t} \mathbf{B}_t\left(\nabla^{\tau_t}u_t,\nabla^{\tau_t}u_t\right)h^2\alpha_t^2\,d\Hn+\beta\int_{\partial A_t}\left(\abs{\nabla^{\tau_t}\alpha_t}^2-\abs{\mathbf{B}_t}^2\alpha_t^2\right)u^2_th^2\,d\Hn\\[10pt]
        &+2\int_{\partial A_t} \left[a\abs{\nabla^{\tau_t} u_t}^2-k_tu_t^2\right]h\alpha_t \gamma_t\cdot \nabla^{\tau_t} h\,d\Hn+\beta\int_{\partial A_t}(2\alpha_t h\nabla^{\tau_t}h\cdot\nabla^{\tau_t}\alpha_t+\alpha_t^2\abs{\nabla^{\tau_t}h}^2)u_t^2\,d\Hn,
    \end{split}\]
  
In the notation of \autoref{lemmIC}, let
\[P_t=I-\nu_t\otimes\nu_t,\quad M_t = J_t [D \Phi_t]^{-1}[D\Phi_t]^{-T},\quad\text{and}\quad M^\tau_t=J^\tau_t [D\Phi_t]^{-1}\widehat{P_t}[D\Phi_t]^{-T}.\]
Then, we pull back the integrals to $A$, and split $\lambda_\beta''$ as
\[
\lambda''_\beta(t)=I_1(t)+I_2(t)+I_3(t),
\]
where 
\[\begin{split}
    I_1(t)=&2\left[\int_{A}\rho \nabla(\,\hdu\,)\, M_t\, \nabla (\,\hdu\,)\,dx+\beta\int_{\partial A} \hdu \,^2J^\tau_t\,d\Hn-\lambda_\beta(t)\int_{A} \hdu\,^2J_t\,dx-2\lambda'_\beta(t)\int_{ A} \widehat{u_t}\,\hdu \,J_t\,dx\right];\\[15pt]I_2(t)=&\int_{\partial A} \left[a\nabla^{\tau}(\, \widehat{u_t}\,)\, M_t^\tau\, \nabla^{\tau}(\, \widehat{u_t}\,)-\widehat{k_t}J^\tau_t\,\widehat{u_t}\right]\left[\widehat{\mathbf{B}_t}(\widehat{\gamma_t}, \widehat{\gamma_t})+2\widehat{\gamma_t}\,(\hat{P}_t [D\Phi_t]^{-T})\,\nabla^{\tau}\widehat{\alpha_t} + \left(\widehat{H_t}-\dfrac{2\beta}{a}\right)\widehat{\alpha_t}\,^2\right]h^2\,d\Hn\\[10pt]&-2a\int_{\partial A} \widehat{\mathbf{B}_t}\left((\hat{P}_t [D\Phi_t]^{-T})\,\nabla^{\tau}\widehat{u_t},(\hat{P}_t [D\Phi_t]^{-T})\,\nabla^{\tau}\widehat{u_t}\right)h^2\widehat{\alpha_t}\,^2\,J^\tau_t\,d\Hn\\[10pt]
    &+\beta\int_{\partial A}\left(\nabla^{\tau}(\, \widehat{\alpha_t}\,)\, M_t^\tau\, \nabla^{\tau_t}(\, \widehat{\alpha_t}\,)-\abs{\widehat{\mathbf{B}_t}}\,^2 \widehat{\alpha_t}\,^2J^\tau_t\right)\widehat{u_t}\,^2\,h^2\,d\Hn;\\[15pt]
    I_3(t)=&2\int_{\partial A} \left[a\nabla^{\tau}(\, \widehat{u_t}\,)\, M_t^\tau\, \nabla^{\tau}(\, \widehat{u_t}\,)-\widehat{k_t}J^\tau_t\,\widehat{u_t}\right]h\,\widehat{\alpha_t}\, \widehat{\gamma_t} \,\widehat{P_t}\, \nabla^{\tau} h\,d\Hn\\[10pt]&+\beta\int_{\partial A}(2\widehat{\alpha_t}\, h\nabla^{\tau}h\,(\widehat{P_t} [D\Phi_t]^{-T})\,\nabla^{\tau}\widehat{\alpha_t}+\widehat{\alpha_t}\,^2\nabla^{\tau}h\,\widehat{P_t}\,\nabla^\tau h)\widehat{u_t}\,^2 J^\tau_t\,d\Hn.
\end{split}\]

 Hence
\begin{equation}\label{eq:splitting}
|\lambda''_\beta(t)-\lambda''_\beta(0)|
\leq T_1+T_2+T_3,
\qquad
T_i:=|I_i(t)-I_i(0)|.
\end{equation}

We first estimate $T_1$ as
\[
T_1\leq T_{1,1}+T_{1,2}+T_{1,3}+T_{1,4},
\]
where
\begin{align}
T_{1,1}
:={}&
2\left|
\int_A
\rho \left[M_t\, \nabla (\,\hdu\,)\right]\cdot\nabla(\,\hdu\,)\,dx
-
\int_A
\rho|\nabla\dot u|^2\,dx
\right|,
\label{eq:T11}\\
T_{1,2}
:={}&
2\beta\left|
\int_{\partial A}
\widehat{\dot {u_t}}\,^{2} J_t^\tau\,d\Hn
-
\int_{\partial A}
\dot u^2\,d\Hn
\right|,
\label{eq:T12}\\
T_{1,3}
:={}&
2\left|
\lambda_\beta(t)\int_A
\widehat{\dot {u_t}}\,^{2}J_t\,dx
-
\lambda_\beta(0)\int_A
\dot u^2\,dx
\right|,
\label{eq:T13}\\
T_{1,4}
:={}&
4\left|
\lambda'_\beta(t)\int_A
\widehat{ u_t}\,\widehat{\dot {u_t}}\,J_t\,dx
-
\lambda'_\beta(0)\int_A
u\dot u\,dx
\right|.
\label{eq:T14}
\end{align}
All the terms can be easily estimated using \autoref{lemmIC} and estimates \eqref{eq:jac-est-forte}, and \eqref{eq:u-small}. For instance, we have 
\[\begin{split}
T_{1,1}\le& C\left[\norma{M_t-I}_{L^\infty(A)} \int_A \abs{\nabla\dot{u}}^2\,dx+\int_A \abs*{\abs{\nabla \dot{u}}^2-\abs{\nabla \hdu}^2}\,dx\right]\\[10pt]
\le&C \left[\norma{h}_{C^{2,\gamma}(\partial A)}\norma{\dot u}_{H^1(A)}^2+\norma{\dot{u}-\hdu}_{H^1(A)}\left(\norma{\dot u}_{H^1(A)}+\norma{\hdu}_{H^1(A)}\right)\right]\\[10pt]\le& C \norma{h}_{C^{2,\gamma}(\partial A)}\norma{h}_{H^1(\partial A)}^2.
\end{split}\]

Analogously, we can estimate $T_{1,2}, T_{1,3},$ and $T_{1,4}$, so that
\begin{equation}\label{eq:T1-est}
T_1
\leq C \norma{h}_{C^{2,\gamma}(\partial A)}\norma{h}_{H^1(\partial A)}^2.
\end{equation}

We now estimate $T_2$ and $T_3$. $I_2$ can be written as 
\[I_2(t)=\int_{\partial A} E_t h^2\,d\Hn,\]
for an appropriate function $E_t$.
Moreover, by estimates \eqref{eq:jac-est-forte}, \eqref{eq:curv-est-forte}, \eqref{eq:alpha,gamma-est-forte}, and \eqref{eq:u-small},  we have that $E_t$ is the sum of products of functions $f(t)$ satisfying 
\[\norma{f(t)-f(0)}_{L^\infty{\partial A}}\le C \norma{h}_{C^{2,\gamma}(\partial A)},\quad\text{and}\quad\norma{f(0)}_{L^\infty(\partial A)}\le C,\]
so that 
\[\norma{E_t-E_0}_{L^\infty(\partial A)}\le C\norma{h}_{C^{2,\gamma}(\partial A)},\]
and
\begin{equation}\label{eq:T2-est}T_2\le \int_{\partial A}\abs{E_t-E_0} h^2\,d\Hn\le C \norma{h}_{C^{2,\gamma}(\partial A)} \norma{h}_{H^1(\partial A)}^2.\end{equation}
Similarly, $I_3$ can be written as 
\[I_3(t)=\int_{\partial A} \left[hF_t\cdot \nabla^\tau h+ (G_t\, \nabla^\tau h)\cdot\nabla^\tau h\right]\,d\Hn,\]
with $F_0=0$, 
\[\norma{F_t}_{L^\infty(\partial A)}\le C\norma{h}_{C^{2,\gamma}(\partial A)},\quad\text{and}\quad\norma{G_t-G_0}_{L^\infty(\partial A)}\le C\norma{h}_{C^{2,\gamma}(\partial A)}.\]
So that, finally
\begin{equation}\label{eq:T3-est}T_3\le C \norma{h}_{C^{2,\gamma}(\partial A)} \norma{h}_{H^1(\partial A)}^2. \end{equation}\medskip

Therefore, combining \eqref{eq:splitting}, \eqref{eq:T1-est}, \eqref{eq:T2-est}, and
\eqref{eq:T3-est}, we conclude that
\begin{equation}\label{eq:lambda-second-continuity}
|\lambda''_\beta(t)-\lambda''_\beta(0)|
\leq  C \norma{h}_{C^{2,\gamma}(\partial A)} \norma{h}_{H^1(\partial A)}^2.
\end{equation}
\end{proof}

\section{Further remarks}\label{secFR}
\subsection{The case of Dirichlet Boundary conditions}\label{BB}
In this section, we consider the analogue of the previous analysis for the
two-phase eigenvalue problem with  Dirichlet boundary conditions on the exterior
boundary. The argument is parallel to the Robin case. We keep the notation of the previous sections and consider the eigenvalue problem
\begin{equation}
\begin{cases}
-\Delta u=\lambda^D u & \text{in }\Om,\\[5pt]
-a\Delta u=\lambda^D u & \text{in }A\setminus\bar\Om,\\[5pt]
u^-=u^+ & \text{on }\partial \Om,\\[5pt]
\partial_{\nu_{B_r}}u^-=
a\,\partial_{\nu_{B_r}}u^+
& \text{on }\partial \Om,\\[5pt]
u=0 & \text{on }\partial A.
\end{cases}
\label{eq:dirichlet-problem}
\end{equation}
We denote its first eigenvalue by $\lambda^D(\Om,A)$. Equivalently,
\begin{equation}
\lambda^D(\Om,A)
=
\min_{v\in H^1_0(A)\setminus\{0\}}
\frac{
\displaystyle\int_{A}\rho|\nabla v|^2\,dx
}{
\displaystyle\int_{A}v^2\,dx
}.
\label{eq:dirichlet-rayleigh}
\end{equation}
If $A$ is connected, the first eigenvalue is simple, and the first eigenfunction is positive in
$A$. We normalise it by
\begin{equation}
\int_{A}u^2\,dx=1.
\label{eq:dirichlet-normalization}
\end{equation}\medskip

If $\Om=B_r$ and $A=B_R$ are concentric balls, the first eigenfunction is
radial. We write
\[
u(x)=U(|x|).
\]
Then $U$ solves
\begin{equation}
\begin{cases}
     -U''(s)-\dfrac{n-1}{s}U'(s)=\lambda^D U(s) & \text{in}\quad (0,r)\\[5pt]
           -aU''(s)-a\dfrac{n-1}{s}U'(s)=\lambda^D U(s) & \text{in}\quad (r,R)\\[5pt]
U(r^-)=U(r^+),\\[5pt]
U'(r^-)=aU'(r^+),\\[5pt]
U(R)=0,\\[5pt]
U>0.
\end{cases}
\label{eq:dirichlet-U-equation}
\end{equation}

We prove the following theorem. 

\begin{teor}
\label{thm:dirichlet-a}
Let $a,r>0$, $R>r$. Let $\lambda^N(B_r)$ be the principal
Neumann eigenvalue of the Laplacian on $B_r$. Then:
\begin{itemize}
\item If
\begin{equation}
(1-a)\lambda^D(B_r,B_R)<(1-a)\lambda^N(B_r),
\label{eq:dirichlet-a-greater-instability}
\end{equation}
then $B_R$ is not a local minimum for $\lambda^D(B_r,\cdot)$ under the volume
constraint. More precisely, there exists $\eta>0$ such that for $t\in(0,\eta) $ and for any unit vector $\mathbf{v}$, $B_r\ssubset B_R+t\mathbf{v}$ and
\[
\lambda^D(B_r,B_R+t\mathbf{v})
<
\lambda^D(B_r,B_R).
\]

\item If
\begin{equation}
(1-a)\lambda^D(B_r,B_R)>(1-a)\lambda^N(B_r),
\label{eq:dirichlet-a-greater-stability}
\end{equation}
then $B_R$ is a $C^{2,\gamma}$-strictly stable local minimum for
$\lambda^D(B_r,\cdot)$ under the volume constraint in the sense of \autoref{def ottimalità locale}.
\end{itemize}
\end{teor}

The proof of \autoref{thm:dirichlet-a} relies on the same tools used for the proofs of \autoref{teor: main a<1} and \autoref{teor: main a>1}. However, many results on the shape gradient and shape Hessian of $\lambda^D(\Om,\cdot)$ are already known in the literature. We structure this section in parallel to what we did for the Robin problem. In the following, we omit the superscript $D$, writing $\lambda=\lambda^D$.

\subsubsection*{Shape derivatives}

By \cite{DK10} we know that $\lambda(\Om,\cdot)$ is differentiable and that the following proposition holds 

\begin{prop}\label{lemma lambda'D}
In the notations of \autoref{Teorema HAdamart}, we have that
  \begin{equation}
\ell_1[\lambda](A)(h)
=
-a\int_{\partial A}
\left(\partial_{\nu}u\right)^2 h\,d\Hn,
\label{eq:dirichlet-first-variation}
\end{equation}
  \begin{equation}
\ell_2[\lambda](A)(h,h)
=
a\int_{\partial A}
\left[2\dot{u}\partial_\nu\dot{u}+H\left(\partial_{\nu}u\right)^2 h^2\right]\,d\Hn,
\label{eq:dirichlet-second-variation}
\end{equation}

where,   $\dot u$ is the unique solution of 
\begin{equation}\label{eq forte u'^D}
\begin{cases}
-\Delta \dot u -\lambda \dot u=\lambda'u & \text{in }\Om, \\[5 pt]
-a\Delta \dot u -\lambda \dot u=\lambda'u & \text{in }A\setminus \bar\Om, \\[5 pt]
{\dot u}^-={\dot u}^+ & \text{on } \partial \Om, \\[5 pt]
\partial_{\nu_\Om}{\dot u}^- = a \partial_{\nu_\Om} {\dot u}^+& \text{on }\partial\Omega,\\[5 pt]
 \dot u =-h\partial_{\nu} u & \text{on } \partial A. 
\end{cases}
\end{equation}
such that 
\begin{equation}\label{derivata norma 1D}
    \int_{A}u \dot u\,dx=0.
\end{equation}

\end{prop}

\begin{oss}\label{DpallaLag}
If $\Om= B_r$ and $A=B_R$ with $R>r$,  then
\[\ell_1[\lambda](B_R)(h)=-aU'(R)^2
\int_{\partial B_R}h\,d\Hn,\]
and $B_R$ is a critical shape under volume constraint for $\lambda_\beta(B_r,\cdot)$ and the Lagrangian for the associated problem is
\[\mathcal{L}(A)=\lambda(B_r,A)+\Lambda \mathcal{V}(A),\]
where 
\begin{equation}
\Lambda=aU'(R)^2.
\label{eq:dirichlet-lagrange-multiplier}
\end{equation}
In particular, then
\[\ell_2[\mathcal{L}](B_R)(h,h)=2a\int_{\partial B_R}\left[\dot{u}\partial_\nu\dot{u}+H(\partial_\nu u)^2h^2\right]\,d\Hn.\]
\end{oss}

\subsubsection*{Proof of the main result}

In analogy with \autoref{PMT}, we use \autoref{DL} to prove \autoref{thm:dirichlet-a}; hence, we need the shape functional $\lambda(\Om,\cdot)$ to satisfy assumptions \ICref\, and \Csref . Both of these properties were proven in the one-phase case in \cite{DL19, D02}. The presence of the two-phase coefficient does not introduce any substantial
new difficulties in this part of the argument. For this reason, we state the following proposition without proof.

\begin{prop}\label{lem:Hs and IC-condition dir}
    Let $A\subset\R^n$ be a bounded open set with $C^{3}$ boundary such that $\Om\ssubset A$. Then the shape functional $\lambda(\Om,\cdot)$ satisfies condition \Csref\,  and \ICref\, at $A$ with $s=1/2$ and $X=C^{2,\gamma}$.
\end{prop}
Thus the proof of \autoref{thm:dirichlet-a} follows as the one of \autoref{teor: main a<1} and \autoref{teor: main a>1}, once we characterise the assumptions under which $B_R$ is a strictly stable shape under volume constraint for $\lambda(B_r,\cdot)$. We prove the following proposition.
\begin{prop}\label{lem:corcivity condition dir}
 Let $R>r>0$ and let \[\mathcal{T}(\partial B_R) = \Set{ h \in H^{\frac{1}2{}}(\partial B_R) \colon\, \int_{\partial B_R} h\,d\Hn=0  }.\]
   If 
   \[(1-a)\lambda^D(B_r,B_R)>(1-a)\lambda^N(B_r),\]
    then there exists $c>0$ such that, for all $h\in \mathcal{T}(\partial B_R)\setminus\set{0}$,
    \[\ell_2[\lambda^D+\Lambda \mathcal{V}](h,h)> 0\]
    If  \[(1-a)\lambda^D(B_r,B_R)<(1-a)\lambda^N(B_r),\]
    then there exists $h_1\in\mathcal{T}(\partial B_R)\setminus\set{0}$ such that  
    \[\ell_2[\lambda^D+\Lambda \mathcal{V}](h_1,h_1)<0,\]
    in particular, for any vector $\mathbf{v}\ne0$ we can choose 
    \[h_1(x)=\dfrac{\mathbf{v}\cdot x}{\abs{x}}.\]
\end{prop}

\subsubsection*{Coercivity of the Lagrangian}\label{section:coercivity dirichlet lag}
For the remainder of the section, we fix $\Om= B_r$ and $A=B_R$ with $R>r$ and 
\[H=\dfrac{n-1}{R}.\]

In contrast with the Robin case, we were not able to directly relate the shape Hessian of $\lambda$ with an appropriate eigenvalue problem. However, in the case of concentric balls we can use the spherical harmonics to prove the following proposition.

\begin{prop}\label{dotuD}
    let $\set{Y_{k,l}}$ be the orthonormal system of of spherical harmonics of $\partial B_R$. If $h\in \mathcal{T}(\partial B_R)$, then
\[h=\sum_{k\ge1}\sum_{l=1}^{d_k} h_{k,l} Y_{k,l}\]
for some coefficients $h_{k,l}$, and,
\[\dot u[h](x) = -U'(R)\sum_{k\ge 1} \sum_{l=1}^{d_k}h_{k,l} V_k(\abs{x}) Y_{k,l}\left(\dfrac{x}{\abs{x}}\right),\]
where $V_k$ is the unique bounded solution to
\begin{equation}\label{eq:Vell eq}
\begin{cases}
 -V_k''(s)-\dfrac{n-1}{s}V_k'(s)+\dfrac{k(n-2+k)}{s^2}V_k(s) =\lambda V_k & \text{in}\quad (0,r)\\[10pt]
            -aV_k''(s)-a\dfrac{n-1}{s}V_k'(s)+a\dfrac{k(n-2+k)}{s^2}V_k(s) =\lambda V_k & \text{in}\quad (r,R)\\[5pt]
        V_k(r^-)=V_k(r^+)\\[5pt]
        V_k'(r^-)=a V_k'(r^+)\\[5pt]
      V_k(R)=1.
\end{cases}
\end{equation}

\end{prop}
\begin{proof}

    If
\[
\int_{\partial B_R}h\,d\Hn=0,
\]
$\lambda'=0$ and
$\dot u=\dot u[h]$ solves
\[
\begin{cases}
-\Delta\dot u-\lambda\dot u=0
& \text{in }B_r,\\[5pt]
-a\Delta\dot u-\lambda\dot u=0
& \text{in }B_R\setminus{B_r},\\[5pt]
\dot u^-=\dot u^+
& \text{on }\partial B_r,\\[5pt]
\partial_{\nu_{B_r}}\dot u^-=
a\,\partial_{\nu_{B_r}}\dot u^+
& \text{on }\partial B_r,\\[5pt]
\dot u=-U'(R)h
& \text{on }\partial B_R,\\[10pt]
\displaystyle \int_{B_R}u\dot u\,dx=0.
\end{cases}
\]
Since the spherical harmonics are a complete orthonormal system of $L^2(\partial B_R)$, we have that

\[
h=
\sum_{k\geq1}\sum_{l=1}^{d_k}
h_{k,l}Y_{k,l},
\]

and, by separation of variables and linearity, we have
\[
\dot u[h](x)
=
-U'(R)
\sum_{k\geq1}\sum_{l=1}^{d_k}
h_{k,l}
V_k(|x|)
Y_{k,l}\left(\frac{x}{|x|}\right).
\]
\end{proof}

\begin{oss}\label{etakmonotone}
    We remark that \eqref{eq:Vell eq} is the same system as \eqref{vk} except for the coefficient on the right-hand side of the first two equations. Namely, in \eqref{eq:Vell eq} the two-phase Robin eigenvalue is replaced by the Dirichlet one. Nevertheless, all the properties we proved for the solution to \eqref{vk} in \autoref{steklov palla} hold (with the same proof) for the solution of \eqref{eq:Vell eq}. In particular $V_k$ is positive in $(0,R]$ and $V'_k(R)$ is strictly monotone in $k$.
\end{oss}

\begin{proof}[Proof of \autoref{lem:corcivity condition dir}]
    By \autoref{DpallaLag} and \autoref{dotuD},
    if 
    \[h=\sum_{k\ge1}\sum_{l=1}^{d_k}h_{k,l}Y_{k,l},\] we immediately have that 
    \[\ell_2[\lambda+\Lambda\mathcal{V}](B_R)(h,h)=2a (U'(R))^2\sum_{k\ge1}(V'_k(R)+H)\sum_{l=1}^{d_k} h_{k,l}^2.\]
    As observed in \autoref{etakmonotone}, $V'_k(R)$ is strictly increasing in $k$. Hence, if
    \[V_1'(R)+H>0,\]
    $B_R$ is a strictly stable shape under volume constraint, while if
    \[V_1'(R)+H<0,\]
    we have that 
    \[\ell_2[\lambda+\Lambda\mathcal{V}](B_R)(h,h)<0\]
    for every $h$ that is a linear combination of the spherical harmonics of degree $1$, that is 
    \[h(x)=\dfrac{\mathbf{v}\cdot x}{\abs{x}},\]
    for some vector $\mathbf{v}\ne0$.\medskip

    In order to characterise the sign of $V_1'(R)+H$, we repeat the same argument of \autoref{lem: sign f}.
In $(r,R)$, $V=V_1$ and $U'$ solve the same second-order differential equation.
Therefore, we consider the weighted Wronskian
\[
W(s)
=
s^{n-1}
\left(
U''(s)V(s)-V'(s)U'(s)
\right).
\]
Similarly to \autoref{lem: sign f}, $W$ is constant in $(r,R)$, Hence $W(R)=W(r^+)$.\medskip

At $s=R$, using $V(R)=1$, and \eqref{eq:dirichlet-U-equation}, we get
$U''(R)=-HU'(R)$ and
\begin{equation}
W(R)
=
-R^{n-1}U'(R)(V'(R)+H).
\label{eq:dirichlet-WR}
\end{equation}
Similarly, in $(0,r)$, $V$ and $U'$ solve the same
equation and are both regular at the origin. Hence, there exists
$c\in\R$ such that
\begin{equation}
V(s)=cU'(s)
\qquad\text{for }0<s<r.
\label{eq:dirichlet-V-cUp}
\end{equation}
Thus, arguing as in \autoref{lem: sign f},
\begin{equation}
W(r^+)
=
r^{n-1}
\frac{a-1}{a^2}
V(r^-)\,U''(r^-).
\label{eq:dirichlet-Wr}
\end{equation} 
Then, comparing \eqref{eq:dirichlet-WR} and
\eqref{eq:dirichlet-Wr} gives 
\[V'(R)+H=-\left(\dfrac{r}{R}\right)^{n-1}\dfrac{V(r^-)}{a U'(R)} (a-1)U''(r^-).\]
As $V$ is positive, and, by the Hopf lemma, $U'(R)<0$, we have that $V'(R)+H$ as the same sign as $(a-1)U''(r^-)$. Finally, repeating the arguments of \autoref{lem: sign U''}, we have that $U''(r^-)>0$ if and only if $\lambda>\lambda^N(B_r)$. Hence $V'(R)+H>0$ if and only if
\[(1-a)\left(\lambda-\lambda^N(B_r)\right)>0,\]
which concludes the proof.
\end{proof}

\subsection{The limit problem}\label{limsect}
As mentioned in \autoref{intro}, the two-phase optimisation problem has been extensively investigated in the thin-layer approximation. Namely, in \cite{DPO25,BBN17,BBN17pt2} the authors study the optimisation problem \[\inf\Set{\lambda_\beta(\Om,h): h: \partial \Om\to (0,+\infty),\, h \text{ is regular and } \int_{\partial \Om} h\,d\Hn=\tilde m}.\]

In particular, the authors prove that when $a\to0^+$, the eigenvalue  $\lambda_\beta(\Om, \Sigma_{ah}\cup\bar\Om)$ converges to $\lambda_\beta(\Om,h)$ the first eigenvalue of the following problem 
\begin{equation}\label{eq autofun lim}
\begin{cases}
-\Delta u=\lambda_\beta u & \text{in }\Om,\\[5pt]
\partial_\nu u+b_h u=0 & \text{on }\partial \Om,
\end{cases}
\end{equation}
where
\[
b_h=
\frac{\beta}{1+\beta h}
\]
in the Robin case, while in the Dirichlet case it is
obtained by taking $\beta=+\infty$, namely
\[
b_h=\frac1h.
\]
The first eigenvalue $\lambda_\beta (\Om,h)$ has the following variational characterization 
\begin{equation}
\lambda_\beta(\Om,h)=\min_{v\in H^(\Om)\setminus\set{0}}\dfrac{\dint_\Om |\nabla v|^2\,dx+\dint_{\partial \Om} b_h v^2\,d\Hn }{\dint_\Om v^2\,dx}.
\end{equation}\
Hence, if $\Om=B_r$ the optimisation problem considered is 
\begin{equation}\label{eq:min lim}
    \min\Set{{\lambda_\beta(B_r, h):h\in L^1(\partial \Om),\quad h\geq 0,\quad \int_{\partial B_r}hd\Hn=m}}
\end{equation}

The authors investigate under which assumptions the minimum is achieved by the constant configuration 
\[
h_0:=\frac{m}{P(B_r)},
\qquad
b_0:=b_{h_0}.\]
When $h_0$ is not a minimum in \eqref{eq:min lim}, we say that symmetry breaking occurs.\medskip

We now propose an alternative proof of the symmetry-breaking phenomenon based on a perturbative approach analogous to the one used to prove \autoref{teor: main a<1}. Namely, we prove the following result. 
    \begin{teor}\label{teor lim}    
Let $\beta, r>0$ and consider $\lambda^N(B_r)$ be the principal Neumann eigenvalue on $B_r$.
   If
        \[\lambda_\beta(B_r,h_0)>\lambda^N(B_r),\]
        then $h_0$ is not a minimum for the problem \eqref{minimisation in h}. In particular, for any unit vector $\mathbf{v}$ there exists $\eta>0$ such that for $t\in(0,\eta), $
    
        \[\lambda_\beta (B_r,h_0)>\lambda_\beta \left(B_r,h_0+t\dfrac{\mathbf{v}\cdot x}{\abs{x}}\right).\]
       
\end{teor}

Let $\Om=B_r$, then, if $h=h_0$ is constant, the first eigenfunction of \autoref{eq autofun lim} $u_0$ is radial, and we call $U$ its radial profile. 
Moreover, we consider
\[
h_t=h_0+t\xi,
\qquad
\int_{\partial B_r}\xi\,d\mathcal H^{n-1}=0.
\]
The zero-average condition ensures that  for all $t$
\[\int h_td\Hn=m.\]
We set
$\lambda_\beta(t):=\lambda_\beta(B_r,h_t)$ and $b_t:=b_{h_t}$. At $t=0$ one has
\[
\dot b_t=-b_t^2\xi,
\qquad
\ddot b_t=2b_t^3\xi^2.
\]

\begin{prop}\label{prop:first-second-variation-limit-prob}
Let $u_t$ be the positive $L^2(B_r)$-normalised eigenfunction associated with $\lambda_\beta(t)$, and let $\dot u$ be its derivative at $t=0$. Then
\[
\lambda_\beta'(0)=0.
\]
Moreover, $\dot u$ solves
\[
\begin{cases}
-\Delta \dot u-\lambda_\beta\dot u=0 & \text{in } B_r,\\[5pt]
\partial_\nu\dot u+b_0\dot u
=
b_0^2U(r)\xi & \text{on } \partial B_r,\\[10pt]
\displaystyle\int_{B_r}u_0\dot u\,dx=0,
\end{cases}
\]
and the second derivative is given by

\[
\lambda_\beta''(0)
=
2b_0^3U(r)^2
\int_{\partial B_r}\xi^2\,d\Hn
-
2b_0^2 U(r)\int_{\partial B_r}\dot u\xi\,d\Hn.
\]
\end{prop}

\begin{proof}
The assertions follow by differentiating the identity
\[\lambda_\beta(t)=\int_{\Om}|\nabla u_t|^2\,dx+\int_{\partial \Om} b_t u_t^2\,d\Hn\]
with respect to $t$ and using \eqref{eq autofun lim}. In particular
\[
\lambda_\beta'(t)
=
\int_{\partial B_r}\dot b_t u_t^2\,d\Hn.
\]
so that, as $u_0$ is radial, it is constantly equal to $U(r)$ on $\partial B_r$, and, using that $\xi$ has zero average, we deduce
\[
\lambda_\beta'(0)
=
-b_0^2U(r)^2
\int_{\partial B_r}\xi\,d\Hn
=
0.
\]
Similarly, differentiating the weak formulation of \eqref{eq autofun lim} and the formula for $\lambda_\beta'(t)$, we obtain the equation for $\dot u$ and the expression for $\lambda_\beta''(0)$, respectively. 
\end{proof}

Letting 
\[\tilde{Q} (u,v)=\int_{B_r}\nabla u\nabla v\,dx+\int_{\partial B_r} b_0 uv\,d\Hn-\lambda_\beta\int_{B_r}uv\,dx.\]
we have that 
\[\lambda''_\beta(0)
=
\int_{\partial B_r}\ddot b_0 u_0^2\,d\Hn
-2\tilde{Q}(\dot u,\dot u).\]

We introduce the associated Steklov-type spectral problem: find
$\sigma\in\mathbb R$ and $\varphi\not\equiv0$ such that
\[
\tilde{Q}(\varphi,w)
=
\sigma
\int_{\partial B_r}\varphi w\,d\mathcal H^{n-1}
\qquad
\forall w\in H^1(B_r).
\]
Equivalently, in strong form,
\begin{equation}\label{eigen steklov lim}
\begin{cases}
-\Delta \varphi=\lambda_0\varphi & \text{in }B_r,\\[5pt]
\partial_\nu\varphi+b_0\varphi=\sigma\varphi
& \text{on }\partial B_r.
\end{cases}
\end{equation}

As in the Steklov problem introduced in  \autoref{section steklov}, this
problem also admits a discrete spectrum
\[
0=\sigma_0<\sigma_1\leq\sigma_2\leq\cdots\to+\infty.
\]
In analogy with \autoref{steklov palla} we have the following proposition.

\begin{prop}\label{steklov palla lim} 
For every $k\in\N_0$, let   $\set{Y_{k,l}}_{l=1}^{d_k}$ be an orthonormal basis for the space of spherical harmonics of degree $k$ on $\partial B_r$ and let $V_k$ be the bounded solution to
\begin{equation}\label{vklim}\begin{cases}
    
        -V_k''(s)-\dfrac{n-1}{s}V_k'(s)+\dfrac{k(n-2+k)}{s^2}V_k(s) =\lambda_\beta V_k & \text{in}\quad (0,r)\\[5pt]

      V_k(r)=1,
\end{cases}\end{equation}
Then, the family of functions
\[\phi_{k,l}(x)=V_k\left(\abs{x}\right) Y_{k,l}\left(\dfrac{x}{\abs{x}}\right),\]
forms a complete $L^2(\partial B_r)-$orthonormal system of eigenfunctions of problem \eqref{eigen steklov lim}. Moreover, $V_k>0$ in $(0,r]$, the
 eigenvalue corresponding to $\phi_{k,l}$ is \[\sigma_{k,l}=V'_k(r)+b_0,\] 
has multiplicity $d_k$ and is monotone in $k$. In particular, the eigenspace associated with the principal eigenvalue $\sigma_b$ is generated by the system
\[\Set{V_1\left(\abs{x}\right)\dfrac{x_l}{\abs{x}}\colon\, l=1,\dots n}.\]
\end{prop}

\begin{oss}
    Since $\xi$ has zero average, we can write
\[
\xi
=
\sum_{k\geq1}\sum_{l=1}^{d_k}
\xi_{k,l}Y_{\ell,k}.
\]
By linearity, the solution of the equation for $\dot u$ is
\[
\dot u(x)
=
b_0^2U(r)
\sum_{k\geq1}\sum_{l=1}^{d_l}
\frac{\xi_{k,l}}{\sigma_{k,l}}
V_\ell(|x|)
Y_{\ell,k}\left(\frac{x}{|x|}\right).
\]
Indeed, on $\partial B_r$,
\[
\partial_\nu\dot u+b_0\dot u
=
b_0^2U(r)
\sum_{\ell,k}
\frac{\xi_{k,l}}{\sigma_{k,l}}
\bigl(V_\ell'(r)+b_0V_\ell(r)\bigr)
Y_{\ell,k}
=
b_0^2U(r)\xi.
\]

Substituting this expansion into the formula for $\lambda_\beta''(0)$ and using the
orthonormality of the spherical harmonics, we obtain the expression
\begin{equation}\label{riscrittura lambda'' lim}
\lambda_\beta''(0)
=
2b_0^3U(r)^2
\sum_{k\geq 1}\sum_{l=1}^{d_k}
\left(
1-\frac{b_0}{\sigma_{k,l}}
\right)
\xi_{k,l}^2.
\end{equation}
\end{oss}

Thus, since $\sigma_{k,l}$ are ordered, the spherical harmonics of degree one are the principal directions of the second variation.

\begin{prop}\label{prop:breaking equivalent conditions lim}
Let $\lambda^N$ denote the principal  Neumann eigenvalue of $B_r$. Then the following conditions are equivalent:
\begin{enumerate}
    \item For every unit vector $\mathbf v$, setting
    \[
    \xi(x)=\frac{x\cdot \mathbf v}{|x|}
    \qquad \text{on } \partial B_r,
    \]
    one has
    \[
    \lambda''(0)<0.
    \]

    \item The first coefficient in the second variation is negative, namely
    \[
    1-\frac{b_0}{\sigma_b}<0.
    \]

    \item The radial profile $U$ satisfies
    \[
    U''(r)>0.
    \]

    \item The limit eigenvalue is larger than the principal Neumann eigenvalue:
    \[
    \lambda(B_r,h_0)>\lambda^N(B_r).
    \]
\end{enumerate}
\end{prop}
\begin{proof}
    The equivalence between $1$ and $2$ follows immediately from \eqref{riscrittura lambda'' lim}.  
    In order to obtain the equivalence between $2$ and $3$, we notice that $U'$ solves the
same equation of $V_1$ and $V_1(r)=1$, thus we have
\[
V_1(s)=\frac{U'(s)}{U'(r)}.
\]
Hence
\[
\sigma_b
=
V_1'(r)+b_0
=
\frac{U''(r)}{U'(r)}+b_0.
\]
Using the Robin boundary condition $U'(r)+b_0U(r)=0$ and the differential equation for the radial profile
\[
U''(r)+H_{\partial B_r}U'(r)+\lambda_0U(r)=0,
\]
we obtain
\begin{equation}\label{eq U'' per prob lim}
    U''(r)=-\left(\lambda_0-H_{\partial B_r}b_0\right)U(r).
\end{equation}
 Therefore
\[
\sigma_b
=
\frac{\lambda_0+b_0^2-H_{\partial B_r}b_0}{b_0},
\]
and
\[
1-\frac{b_0}{\sigma_b}
=
\frac{\lambda_0-H_{\partial B_r}b_0}
{\lambda_0+b_0^2-H_{\partial B_r}b_0}.
\]
Since $\sigma_b>0$, the denominator is positive. Hence, the sign of the first coefficient is the sign of $\lambda_0-H_{\partial B_r}b_0$. Thus, by \eqref{eq U'' per prob lim}, we have that 
\[1-\dfrac{b_0}{\sigma_b}>0\iff U''(r)<0.\]

Finally, the equivalence between $3$ and $4$ follows step by step the proof of \autoref{lem: sign U''}. 
\end{proof}
 
\begin{proof}[Proof of \autoref{teor lim}]
The assertion follows immediately from \autoref{prop:breaking equivalent conditions lim}.
\end{proof}

\subsubsection*{Acknowledgements} 
The authors are members of and were partially supported by Gruppo Nazionale per l’Analisi Matematica, la Probabilità e le loro Applicazioni
(GNAMPA) of Istituto Nazionale di Alta Matematica (INdAM).

\printbibliography[heading=bibintoc]

\Addresses

\end{document}